\documentclass[]{amsart}

\usepackage{amsthm}
\usepackage{amsmath}
\usepackage{microtype}
\usepackage{mathtools}
\usepackage{amssymb}
\usepackage{enumerate}
\usepackage{overpic}
\usepackage{graphicx}
\usepackage[hidelinks,pagebackref,pdftex]{hyperref}
\usepackage{booktabs}
\usepackage{color}
\usepackage[dvipsnames]{xcolor}
\usepackage{import}
\usepackage{tikz-cd}

\AtBeginDocument{%
   \def\MR#1{}
}

\usepackage{marginnote}
\long\def\@savemarbox#1#2{\global\setbox#1\vtop{\hsize\marginparwidth 
  \@parboxrestore\tiny\raggedright #2}}
\renewcommand*{\backref}[1]{}
\renewcommand*{\backrefalt}[4]{
  \ifcase #1
  [No citations.]
  \or [#2]
  \else [#2]
  \fi }

\numberwithin{equation}{section}
\theoremstyle{plain}
\newtheorem{theorem}[equation]{Theorem}
\newtheorem{corollary}[equation]{Corollary}
\newtheorem{lemma}[equation]{Lemma}
\newtheorem{conjecture}[equation]{Conjecture}
\newtheorem{proposition}[equation]{Proposition}

\newtheorem{claim}[equation]{Claim}

\newtheorem*{namedtheorem}{\theoremname}
\newcommand{\theoremname}{testing}
\newenvironment{named}[1]{\renewcommand{\theoremname}{#1}\begin{namedtheorem}}{\end{namedtheorem}}

\theoremstyle{definition}
\newtheorem{definition}[equation]{Definition}
\newtheorem{remark}[equation]{Remark}

\numberwithin{figure}{section}

\newcommand{\from}{\colon} 

\newcommand{\bcover}{\hat} 

\newcommand{\HH}{{\mathbb{H}}}
\newcommand{\RR}{{\mathbb{R}}}
\newcommand{\ZZ}{{\mathbb{Z}}}
\newcommand{\NN}{{\mathbb{N}}}
\newcommand{\CC}{{\mathbb{C}}}
\newcommand{\QQ}{{\mathbb{Q}}}

\newcommand{\calA}{\mathcal{A}}
\newcommand{\calL}{\mathcal{L}}

\newcommand{\calS}{\mathcal{S}}

\newcommand{\refthm}[1]{Theorem~\ref{Thm:#1}}
\newcommand{\reflem}[1]{Lemma~\ref{Lem:#1}}
\newcommand{\refprop}[1]{Proposition~\ref{Prop:#1}}
\newcommand{\refcor}[1]{Corollary~\ref{Cor:#1}}
\newcommand{\refrem}[1]{Remark~\ref{Rem:#1}}
\newcommand{\refclaim}[1]{Claim~\ref{Claim:#1}}
\newcommand{\refconj}[1]{Conjecture~\ref{Conj:#1}}
\newcommand{\refeqn}[1]{\eqref{Eqn:#1}}
\newcommand{\refitm}[1]{\eqref{Itm:#1}}
\newcommand{\refdef}[1]{Definition~\ref{Def:#1}}
\newcommand{\refsec}[1]{Section~\ref{Sec:#1}}
\newcommand{\reffig}[1]{Figure~\ref{Fig:#1}}

\newcommand{\thsup}{{\rm th}}
\newcommand{\bdy}{\partial}

\newcommand{\vol}{\operatorname{vol}}

\newcommand{\bilip}{{\rm bilip}}

\newcommand{\area}{\operatorname{area}}
\newcommand{\len}{\operatorname{len}}

\newcommand{\id}{{\mathop{id}}}
\newcommand{\dhyp}{{d_{\mathrm{hyp}}}}

\newcommand{\abs}[1]{\left\vert #1 \right\vert}

\newcommand{\hmax}{{h_{\max}}}
\newcommand{\Rmin}{{R_{\min}}}

\newcommand{\Zmin}{{Z_{\min}}}
\newcommand{\zmin}{{z_{\min}}}

\newcommand{\smax}{{s_{\max}}}
\newcommand{\Rminhat}{{\hat R_{\min}}}
\newcommand{\Zminhat}{{\hat Z_{\min}}}

\newcommand{\delmax}{{\delta_{\max}}}
\newcommand{\zbd}{{z_{\mathrm Bd}}}
\newcommand{\rr}{{\mathbf{r}}}
\newcommand{\half}{\frac{1}{2}}

\newcommand{\systole}{{\operatorname{sys}}}
\newcommand{\sysmin}{{\operatorname{sysmin}}}
\newcommand{\CAT}{{\operatorname{CAT}(-1)}}

\newcommand{\injrad}{{\operatorname{injrad}}}

\newcommand{\haze}{{\operatorname{haze}}}
\newcommand{\ERROR}{{10^{-5}}}

\newcommand{\Hhat}{{\hat{\HH}^3}}
\newcommand{\Isom}{{\operatorname{Isom}}}

\DeclareMathOperator{\arccosh}{arccosh}
\DeclareMathOperator{\arcsinh}{arcsinh}
\DeclareMathOperator{\arctanh}{arctanh}

\title[Effective bilipschitz bounds on drilling and filling]{Effective bilipschitz bounds \\ on drilling and filling}

\author[D.~Futer]{David Futer}
\address[]{Department of Mathematics, Temple University,
Philadelphia, PA 19122, USA}
\email[]{dfuter@temple.edu}

\author[J.~Purcell]{Jessica S.~Purcell}
\address[]{School of Mathematics, Monash University, VIC 3800, Australia }
\email[]{jessica.purcell@monash.edu}

\author[S.~Schleimer]{Saul Schleimer}
\address[]{Department of Mathematics, 
University of Warwick, Coventry CV4 7AL, UK}
\email[]{s.schleimer@warwick.ac.uk}

\makeatletter
\@namedef{subjclassname@2020}{%
  \textup{2020} Mathematics Subject Classification}
\makeatother
\subjclass[2020]{57K32, 30F40, 57K10}
\thanks{\today}

\begin{document}

\begin{abstract}
This paper proves explicit bilipschitz bounds on the change in metric between the thick part of a cusped hyperbolic $3$--manifold $N$ and the thick part of any of its long Dehn fillings. Given a bilipschitz constant $J > 1$ and a thickness constant $\epsilon > 0$, we quantify how long a Dehn filling suffices to guarantee a $J$--bilipschitz map on $\epsilon$--thick parts. A similar theorem without quantitative control was previously proved by Brock and Bromberg, applying Hodgson and Kerckhoff's theory of cone deformations. We achieve quantitative control by bounding the analytic quantities that control the infinitesimal change in metric during the cone deformation.

Our quantitative results have two immediate applications. First, we relate the Margulis number of $N$ to the Margulis numbers of its Dehn fillings. In particular, we give a lower bound on the systole of any closed $3$--manifold $M$ whose Margulis number is less than $0.29$. Combined with Shalen's upper bound on the volume of such a manifold, this gives a procedure to compute the finite list of $3$--manifolds whose Margulis numbers are below $0.29$.

Our second application is to the cosmetic surgery conjecture. Given the systole of a one-cusped hyperbolic manifold $N$,  we produce an explicit upper bound  on the length of a slope involved in a cosmetic surgery on $N$. This reduces the cosmetic surgery conjecture on $N$ to an explicit finite search.
\end{abstract}

\maketitle

\tableofcontents

\section{Introduction}\label{Sec:Intro}

Dehn filling  is the process of changing a compact $3$--manifold by attaching solid tori to some number of its torus boundary components. For each boundary torus $T$, the choice of filling is determined by a \emph{slope}: that is, an isotopy class of simple closed curve on $T$ that will bound a disk in the attached solid torus.  In the 1960s, Wallace and Lickorish showed that any closed, orientable $3$--manifold is obtained by Dehn filling a link complement in $S^3$ \cite{Wallace, lickorish}.  This established Dehn filling as an important technique in the study of $3$--manifold topology.

Thurston pioneered the geometric study of Dehn surgery. When a compact $3$--manifold with torus boundary has interior admitting a complete hyperbolic structure, the non-compact ends become \emph{cusps} with torus cross-sections. The boundary torus of a cusp neighborhood inherits a Euclidean metric, and each slope inherits a Euclidean length. Thurston showed that complete hyperbolic structure on the interior of a manifold with torus boundary components can always be deformed to incomplete hyperbolic structures  \cite{thurston:notes}. The space of such structures is called \emph{hyperbolic Dehn surgery space}. The completions of such deformed structures are often not manifolds, but sometimes they are diffeomorphic to Dehn fillings of the original. When this happens, the completion is called a \emph{hyperbolic Dehn filling}. Thurston also showed that as the Euclidean lengths of Dehn filling slopes approach infinity, the corresponding hyperbolic Dehn fillings approach the original manifold in the Gromov--Hausdorff topology. It follows that hyperbolic Dehn filling is an important technique in the study of $3$--manifold geometry.

We are particularly interested in uniform and effective geometric estimates for Dehn filling. Here,   \emph{uniform} means that constants appearing in the estimates are independent of the underlying $3$--manifold, while  \emph{effective}  means that these constants are explicitly given.
Many uniform estimates  controlling fine-scale geometry under Dehn filling have previously been developed \cite{brock-bromberg:density,bromberg:conemflds, hk:univ}. These estimates have played an important role in 
proving theorems about spaces of Kleinian groups \cite{brock-bromberg:density, brock-bromberg:inflexibility, bromberg:BersDensity, magid:deformation}; see \refsec{PriorWork} for more details. 
However, apart from theorems establishing the existence of hyperbolic structures  \cite{agol:6theorem, hk:univ, lackenby:surgery}  and bounding  volume
 \cite{fkp:volume, hk:shape}, the previous results have not been effective. 
Explicit estimates are needed to apply Dehn filling techniques to the study of individual manifolds. In particular, where a computer algorithm depends on some theoretical bound in order to know when to stop searching, only an explicit bound can make the algorithm implementable.
Such explicit bounds on fine-scale geometry  are provided for the first time in this paper.

The difference between effective and ineffective results can be illustrated as follows. 
 Thurston's Dehn surgery theorem \cite{thurston:notes}, which says that all but finitely many surgeries on a hyperbolic manifold  yield closed hyperbolic manifolds, is powerful but not effective or uniform. It does not say which slopes one needs to exclude, or even how the number of excluded slopes depends on the manifold. By contrast, the $6$--theorem proved by  Agol \cite{agol:6theorem} and Lackenby \cite{lackenby:surgery} in 2000, which says that all surgeries of length greater than $6$ yield hyperbolic manifolds, is more powerful precisely because it is effective. (The conclusion that the filled manifold is hyperbolic depends on Perelman's proof of the geometrization conjecture, which occurred several years later.)
In many applications, the $6$--theorem  is used to break a problem into cases:  hyperbolic geometry handles the ``generic'' scenario, while ad-hoc methods handle the small, concrete list of exceptions. Our results have a similar effect, enabling computer-assisted proofs for \emph{all} fillings of a manifold. 

We present two applications.  First, for any hyperbolic knot complement $S^3 - K$, we prove an effective upper bound on the length of a cosmetic surgery on $K$. This means that if two different Dehn fillings on $K$ yield the same closed $3$--manifold, the pair of fillings must come from an explicit finite list. See \refcor{CosmeticKnot}  in \refsec{CosmeticIntro} for a precise statement. Thus a finite computer check establishes that knots up to $16$ crossings have no cosmetic surgeries (\refcor{CosmeticComputer}). Second,  Theorems~\ref{Thm:MargulisFilling} and~\ref{Thm:MargulisDrilling} stated in \refsec{MargulisIntro} provide explicit control on the Margulis numbers of closed hyperbolic $3$--manifolds.

\subsection{Prior work on cone deformations}\label{Sec:PriorWork}
In 2002, just before the resolution of the geometrization conjecture, Hodgson and Kerckhoff proved the first effective, uniform version of Thurston's Dehn surgery theorem \cite{hk:univ}. They showed that, for all but $60$ choices of slope $s$ on a one-cusped hyperbolic manifold $N$, the filled manifold $N(s)$ is also hyperbolic. The slopes excluded by their theorem are the ones that have shortest \emph{normalized length}; see \refdef{NormalizedLengthIntro}. Their method was to obtain a hyperbolic metric on $N(s)$ at the end of a one-parameter family of singular metrics, 
with cone-singularities of angle $0 \leq \alpha \leq 2\pi$ along the core of the Dehn filling solid torus. (See \refsec{ConeMfdBasics} for a careful definition of cone-manifolds and related notions.)
When the cone angle starts at $0$, the core is not present, and one obtains the complete hyperbolic metric on the cusped manifold $N$.
When the cone angle becomes $2\pi$, the singular solid torus becomes non-singular, and one has a complete hyperbolic metric on $N(s)$.
Thus one has succeeded in performing hyperbolic Dehn filling.

The technique of deformation through cone structures, initiated by Hodgson and Kerckhoff \cite{hk:ConeRigidity}, has been highly useful. In addition to proving uniform bounds on Dehn filling, Hodgson and Kerckhoff also  gave bounds on volume change under Dehn filling, on the lengths of core geodesics \cite{hk:univ}, and on the shape of hyperbolic Dehn surgery space \cite{hk:shape}. Purcell extended their techniques to give bounds on the change of cusp shape under cone deformation, applying the result to the geometry of knots in $S^3$ \cite{purcell:CuspsConeDeform, purcell:Volumes}. Bromberg applied their methods to study deformations that run from cone angle $2\pi$ to $0$, a process called \emph{drilling}.
Bromberg also extended their results from finite-volume to infinite-volume manifolds, and gave bounds on the change in length of a short, nonsingular geodesic \cite{bromberg:conemflds}.  
We remark that the above-mentioned results bounding the change in length of a closed geodesic \cite{bromberg:conemflds, hk:univ} are uniform (independent of manifold) but not effective. We prove and apply effective versions of these results; see \refcor{MagidLength} and \refcor{ParticularLenBound}, which are also stated later in the introduction.

The application of cone deformations most relevant to this paper is the bilipschitz drilling theorem of Brock and Bromberg \cite{brock-bromberg:density}. Building on Hodgson and Kerckhoff's methods, Brock and Bromberg obtained uniform bilipschitz bounds relating the hyperbolic metrics at the two ends of the deformation.  In the following theorem, $\mu_3$ is the $3$--dimensional Margulis constant. 
See \refdef{MargulisNumIntro} for a review of the thick-thin decomposition and \refdef{Horocusp} for a review of rank-one and rank-two cusps.
The hypothesis that $M$ is geometrically finite means that the convex core of $M$ has finite volume. In particular, finite-volume manifolds are geometrically finite and have no rank one cusps.

\begin{theorem}[Drilling theorem, \cite{brock-bromberg:density}]\label{Thm:BBDrillingThm}
Fix $0 < \epsilon \leq \mu_3$ and $J > 1$. Then there is a number $\ell_0 = \ell_0(\epsilon, J) > 0$ such that the following holds for every geometrically finite hyperbolic $3$--manifold $M$ without rank-one cusps. 
Suppose that $\Sigma \subset M$ is a link composed of closed geodesics, whose total length is less than $\ell_0$. Then 
the inclusion
\[ 
\iota \from (M- \Sigma ) \hookrightarrow M 
\]
restricts to a $J$--bilipschitz diffeomorphism on the complement of $\epsilon$--thin tubes about $\Sigma$. 
\end{theorem}

\refthm{BBDrillingThm} has had several important applications. Using earlier work of Bromberg \cite{bromberg:BersDensity}, Brock and Bromberg used this result to prove the Bers--Sullivan--Thurston density conjecture for freely indecomposable Kleinian groups without parabolics \cite{brock-bromberg:density}. (The proof of the full density conjecture relies upon the Ending Lamination Theorem, as in Ohshika \cite{ohshika:density} and Namazi--Souto \cite{namazi-souto:density}.)
In further applications, Bromberg \cite{bromberg:PunctTorus} and Magid \cite{magid:deformation} used the drilling theorem to show that deformation spaces of Kleinian surface groups are not locally connected. Purcell and Souto used the drilling theorem to show that a large class of hyperbolic manifolds occurs as geometric limits of knot complements in $S^3$ \cite{purcell-souto:KnotLimits}. Cooper, Futer,  and Purcell used it to show that there are knots in $S^3$ with long, geodesic unknotting tunnels \cite{cfp:tunnels}.  For each of these applications, it was important that the length cutoff $\ell_0$ is independent of the manifold $M$.

However, the drilling theorem also has limitations. In particular, the constants are not effective: the dependence of the length cutoff $\ell_0(\epsilon, J)$ on the thickness constant $\epsilon$ and the bilipschitz constant $J$ is not quantified. This means that, while \refthm{BBDrillingThm} can be used in geometric limit arguments as in the previous paragraph, it is less suitable for studying individual manifolds.  This is because it is never clear whether a given $M$ satisfies the hypotheses.  Furthermore, the ineffective form cannot be used in algorithms.

\subsection{Effective bilipschitz bounds}

One of the most important results of this paper is \refthm{BilipEndpoints}, which effectivizes \refthm{BBDrillingThm}.  In the following corollary, $M^{\geq \epsilon}$ denotes the $\epsilon$--thick part of $M$; that is, all points of injectivity radius at least $\epsilon/2$. See \refdef{ThickThin} for full details.

 \begin{theorem}\label{Thm:BilipEndpointsIntro}
Fix  $0 < \epsilon \leq \log 3$ and $J>1$. Let $M$ be a finite-volume hyperbolic $3$--manifold and $\Sigma$ a geodesic link in $M$ whose total length $\ell$ satisfies
\begin{equation*}
\ell \leq \min\left\{ \frac{\epsilon^5}{6771 \cosh^5(0.6 \epsilon + 0.1475)}, \, \frac{\epsilon^{5/2}\log(J)}{11.35} \right\}.
\end{equation*}
Then, setting $N = M - \Sigma$, and equipping it with its complete hyperbolic metric,
there are natural $J$--bilipschitz inclusions
\[
\varphi \from M^{\geq \epsilon} \hookrightarrow N^{\geq \epsilon/1.2}, 
\qquad
\psi \from N^{\geq \epsilon} \hookrightarrow M^{\geq \epsilon/1.2}.
\]
Here $M^{\geq \epsilon}$ and $N^{\geq \epsilon}$ are the $\epsilon$--thick parts of $M$ and $N$, respectively. The compositions $\varphi \circ \psi$ and $\psi \circ \varphi$ are the identity wherever both maps are defined. Furthermore, $\varphi$ and $\psi$ are equivariant with respect to the symmetry group of the pair $(M, \Sigma)$. 
\end{theorem}

Comparing the statements of  Theorems~\ref{Thm:BBDrillingThm} and~\ref{Thm:BilipEndpointsIntro} reveals several differences. Most notably, \refthm{BilipEndpointsIntro} is stronger, in that it gives looser hypotheses on $\epsilon$ as well as quantified hypotheses on $\ell = \len(\Sigma)$ that ensure a $J$--bilipschitz map.  However, \refthm{BilipEndpointsIntro} is slightly weaker in two respects. 
First, it assumes that $M$ has finite volume. This assumption is convenient for our line of argument, but is not crucial:  using algebraic and geometric limits, we have extended  \refthm{BilipEndpointsIntro}  to all hyperbolic $3$--manifolds with finitely generated fundamental groups \cite{FPS:InfiniteVolume}.
Second, \refthm{BilipEndpointsIntro} provides bilipschitz control on a smaller submanifold of $M$. While \refthm{BBDrillingThm} excludes the $\epsilon$--thin tubes about $\Sigma$, \refthm{BilipEndpointsIntro} excludes \emph{all} the $\epsilon$--thin regions of $M$, including all cusps as well as $\epsilon$--thin tubes about geodesics that are not involved in the cone deformation.

While we do not know how to extend \refthm{BilipEndpointsIntro} into the $\epsilon$--thin regions of $M$, we do have quantitative control over the change in complex length of a sufficiently short geodesic. Consider a closed geodesic $\gamma \subset M$, which corresponds to a loxodromic isometry $\varphi = \varphi(\gamma) \in \Isom^+ \HH^3$. This loxodromic isometry $\varphi$ has an invariant axis in $\HH^3$, which it translates by distance $\lambda$ and rotates by angle $\tau$. We define the \emph{complex length} of $\gamma$ to be $\calL(\gamma) = \lambda + i \tau$. The tubular neighborhood of $\gamma$ of some radius $r$, is determined up to isometry by $r$ and $\calL(\gamma)$; compare \refdef{ModelTube}. Thus controlling the change in complex length is the first step to controlling the geometry of an entire tube about $\gamma$.

We prove the following effective version of a result of Bromberg \cite[Proposition~4.3]{bromberg:conemflds}.

\begin{named}{\refcor{ParticularLenBound}}
Let $M$ be a complete, finite volume hyperbolic $3$--manifold. Let $\Sigma \cup \gamma$ be a geodesic link in $M$, where $\gamma$ is connected. 
Let $\calL_M(\gamma) = \len_M(\gamma) + i \tau_M(\gamma)$ be the complex length of $\gamma$ in the complete metric on $M$, 
and suppose that $\max(\len_M(\Sigma), \len_M(\gamma)) \leq 0.0735$. Then $\gamma$ is also a geodesic in the complete metric on $N = M - \Sigma$, of complex length 
$\calL_N(\gamma)$. Furthermore,
\[
1.9793^{-1} \leq \frac{\len_N(\gamma)}{\len_M(\gamma)} \leq 1.9793
\qquad \text{and} \qquad
|\tau_N(\gamma) - \tau_M(\gamma) | \leq 0.05417.
\]
\end{named}

\noindent
When either $\Sigma$ or $\gamma$ is much shorter than $0.0735$, the quantitative control over $\calL(\gamma)$ improves dramatically. See \refthm{ShortStaysShort} for the exact statement. We note that \refcor{ParticularLenBound} also has an extension to all hyperbolic 3-manifolds with finitely generated fundamental groups \cite{FPS:InfiniteVolume}.

\subsection{How to prove bilipschitz bounds}
Next, we outline some top-level steps in the proofs of \refthm{BilipEndpointsIntro} and \refcor{ParticularLenBound}. We begin by showing the existence of a one-parameter family of cone-manifolds interpolating between the complete hyperbolic metric on $M$ and the complete hyperbolic metric on $N = M - \Sigma$.

\begin{named}{\refthm{ConeDefExists}}
Let $M$ be a finite volume hyperbolic $3$--manifold. Suppose that $\Sigma = \sigma_1 \cup \dots \cup \sigma_n$ is a geodesic link in $M$, whose components have lengths satisfying
\[ 
\ell_j = \len_M(\sigma_j) \leq 0.0996 
\qquad \mbox{and} \qquad
\ell = \sum_{j=1}^n \ell_j \leq 0.15601.
\]
Then the hyperbolic structure on $M$ can be deformed to a complete hyperbolic structure on $M-\Sigma$ by decreasing the cone angle $\alpha_j$ along $\sigma_j$ from $2\pi$ to $0$. The cone angles on all components of $\Sigma$ change in unison.
\end{named}

\refthm{ConeDefExists} is due to Hodgson and Kerckhoff~\cite[Corollary~6.3]{hk:univ} in the special case where $\Sigma$ is connected. We extend the result to a link $\Sigma$ with an arbitrary number of components.
The cone-manifolds along the deformation are denoted $M_t$, where $t\in [0, (2\pi)^2]$. Every component of $\Sigma$ in $M_t$ has cone angle $\alpha = \sqrt{t}$. Thus $t=0$ corresponds to the complete metric on $N = M - \Sigma$, while $t = (2\pi)^2$ corresponds to the complete metric on $M$.

In fact, we show more: when $\ell = \len_M(\Sigma)$ is small, every cone-manifold $M_t$ has a large embedded tube about $\Sigma$. See \refthm{ConeDefExistsRBounds} for the full statement. In the work of Hodgson and Kerckhoff \cite{hk:univ, hk:shape}, the radius of this tube is the key ingredient in a number of analytic estimates that control the change in geometry. We work out effective versions of these estimates
in \refsec{BoundaryBound}.

These analytic estimates allow us to prove  \refthm{Bilip}, which provides bilipschitz control on submanifolds of $M$ that stay \emph{thick} throughout the deformation. In the following corollary of \refthm{Bilip}, the submanifold $M_t^{\geq \delta}$ is the $\delta$--thick part of the cone-manifold $M_t$ in its singular metric $g_t$.  See Definitions~\ref{Def:MargulisNumIntro} and~\ref{Def:ThickThin}.

\begin{named}{\refcor{EffectiveBBSpecial}}
Fix $0<\delta\leq 0.938$ and $J>1$.
Let $M$ be a complete, finite volume hyperbolic $3$--manifold. Let $\Sigma \subset M$ be a geodesic link whose total length $\ell$
satisfies
\[
\ell \leq \min\left\{\frac{\delta^2}{17.11}, \, \frac{\delta^{5/2}\log(J)}{7.193}\right\}.
\]
Let $W \subset M$ be any submanifold such that $W\subset M_t^{\geq\delta}$ for all $t$. Then, for all $a,b \in [0, (2\pi)^2]$, the identity map
$\id\from (W, g_a) \to (W, g_b)$ is $J$--bilipschitz.
\end{named}

We also prove a version of \refcor{EffectiveBBSpecial} whose hypotheses are on the cusped manifold $N = M - \Sigma$ instead of the filled manifold $M$. Stating this version requires a definition.

\begin{definition}\label{Def:NormalizedLengthIntro}
Let $N$ be a hyperbolic $3$--manifold with rank-two cusps $C_1, \ldots, C_n$. Choose a slope $s_j$ for each cusp torus $\bdy C_j$. The \emph{normalized length} of $s_j$ is
\[
L_j = L(s_j) = \frac{\len(s_j)}{\sqrt{\area(\bdy C_j)}},
\]
where $\len(s_j)$ is the length of a geodesic representative of $s_j$ on $\bdy C_j$. 

Let $\mathbf{s} = (s_1, \ldots, s_n)$ be the vector of all the slopes. We define the
  \emph{total normalized length} $L = L(\mathbf{s})$ via the formula
\[
\frac{1}{L^2} = \sum_{j=1}^n \frac{1}{L_j^2}.
\]
\end{definition}

Observe that each normalized length $L_j$ is scale-invariant, hence does not depend on the choice of horospherical neighborhood of a cusp $C_j$.

Hodgson and Kerckhoff  \cite{hk:shape} proved that if $\mathbf{s}$ is a vector of slopes in $N$ whose total normalized length is $L(\mathbf{s}) \geq 7.5832$, then there is a family of cone-manifolds $M_t$ interpolating from the complete metric on $N = M_0$ to the complete metric on $N(s_1, \ldots, s_n) = M_{(2\pi)^2}$. See \refthm{UpwardConeDefRBounds}.

Neumann and Zagier~\cite[Proposition~4.3]{neumann-zagier} showed that (asymptotically, for very long fillings) the normalized length of a slope $s_j \subset N$ closely predicts the length of the corresponding core curve in the filled manifold $N(s_1, \ldots, s_n)$. Using the work of Hodgson--Kerckhoff \cite{hk:univ, hk:shape} and Magid \cite{magid:deformation}, we make this relationship completely quantitative.

\begin{named}{\refcor{MagidLength}}
Suppose that $M$ is a complete, finite volume hyperbolic $3$--manifold and $\Sigma \subset M$ is a geodesic link such that one of the following hypotheses holds.
\begin{enumerate}
\item
In the complete structure on $N = M - \Sigma$, the total normalized length of the meridians of $\Sigma$ is $L \geq 7.823$.
\item
In the complete structure on $M$, each component of $\Sigma$ has length at most $0.0996$ and the total length of $\Sigma$ is $\ell \leq 0.1396$.
\end{enumerate}
Then 
\[
\frac{2\pi}{L^2 + 16.17} 
  < \ell
  < \frac{2\pi}{L^2 - 28.78}.
\]  
\end{named}

Using an estimate closely related to \refcor{MagidLength}, we can prove an analogue of \refcor{EffectiveBBSpecial} with hypotheses on the cusped manifold $N = M - \Sigma$.

\begin{named}{\refcor{EffectiveBBSpecialUp}}
Fix $0<\delta\leq 0.938$ and $J>1$. Let $M$ be a complete, finite volume hyperbolic $3$--manifold and $\Sigma$ a geodesic link in $M$.
Suppose that in the complete structure on $N = M-\Sigma$, the total normalized length $L$ of the meridians of $\Sigma$ satisfies
\[ 
L^2 \geq \max \left\{ \frac{107.6}{\delta^2}+14.41, \, 
\frac{45.20}{\delta^{5/2}\log(J)}+14.41 \right\}. 
 \]
Let $W \subset M$ be any submanifold such that $W\subset M_t^{\geq\delta}$ for all $t$. Then, for all $a,b \in [0, (2\pi)^2]$, the identity map
$\id\from (W, g_a) \to (W, g_b)$ is $J$--bilipschitz.
\end{named}

Now, to derive \refthm{BilipEndpointsIntro} from \refcor{EffectiveBBSpecial}, we need a way to ensure (using only hypotheses on $M$ or only hypotheses on $N = M - \Sigma$) that a given submanifold  $W$ remains in the thick part of a cone-manifold throughout a cone deformation. We do so via the following result.

\begin{named}{\refthm{ThickStaysThick}}
Fix $0 < \epsilon \leq \log 3$ and $1 < J \leq e^{1/5}$.
Let $M$ be a complete, finite volume hyperbolic $3$--manifold and $\Sigma \subset M$ a geodesic link. Suppose that  $\ell = \len(\Sigma)$ is bounded as follows:
\[
\ell 
 \leq \frac{\epsilon^5 \log J}{496.1 \, J^5 \cosh^5 (J \epsilon / 2 + 0.1475)}
\]
Then, for every $a, t \in[0,(2\pi)^2]$, the manifolds $M_a$ and $M_t$ in the deformation from $M - \Sigma$ to $M$ satisfy
\[
 M_a^{\geq \epsilon} \subset M_t^{> \epsilon/J}.
\]
\end{named}

The proof of \refthm{ThickStaysThick} combines our previous work \cite{FPS:Tubes} with a close analogue of \refcor{EffectiveBBSpecial} (which provides stronger estimates under stronger hypotheses) to control the injectivity radius at a point $x \in M_t$ for a sub-interval of time. Then, it uses a delicate \emph{crawling argument} (a continuous analogue of induction) to show that this sub-interval must be the entire time interval $[0, (2\pi)^2]$.

Setting $J_0 = 1.2 < e^{1/5}$ in \refthm{ThickStaysThick}, we learn that under appropriate hypotheses on $\ell$, the containment $M^{\geq \epsilon} \subset M_t^{\geq \epsilon/1.2}$ holds for all $t$. Thus, on the submanifold $W = M^{\geq \epsilon}$, \refcor{EffectiveBBSpecial} gives bilipschitz control for all $t$. This proves \refthm{BilipEndpointsIntro}.

A similar crawling argument, using the analytic estimates of \refsec{BoundaryBound}, also proves \refcor{ParticularLenBound}. See also \refcor{ParticularLenBoundUp} for a very similar statement with hypotheses on $N = M - \Sigma$ rather than $M$.

\subsection{Application to Margulis numbers}\label{Sec:MargulisIntro}
Observe that Brock and Bromberg's \refthm{BBDrillingThm} requires the thickness constant $\epsilon$ to be less than the Margulis constant $\mu_3$, whose value is currently unknown. (See \refthm{NonsingularMargulis} for the current state of knowledge.) By contrast, \refthm{BilipEndpointsIntro} makes no hypotheses regarding the Margulis constant. In fact, information flows in the opposite direction: we are able to use \refcor{EffectiveBBSpecial} and \refthm{ThickStaysThick} to control the topology of the thin parts of the cone-manifolds $M_t$ occurring during the deformation. This provides a strong application to Margulis numbers of (complete, non-singular) hyperbolic manifolds.

\begin{definition}\label{Def:MargulisNumIntro}
Let $M$ be a hyperbolic $3$--manifold. The \emph{$\epsilon$--thin part of $M$} is $M^{< \epsilon}$, which consists of all points of $M$ lying on essential loops of length less than $\epsilon$. We say that $\epsilon > 0$ is  \emph{a Margulis number} for $M$ if every component of  $M^{< \epsilon}$ is isometric to either a horocusp or an equidistant tube about a closed geodesic. 

Note that that if $\epsilon$ is a Margulis number for $M$, then so is every $\delta < \epsilon$, although $M^{< \delta}$ may have fewer components than $M^{< \epsilon}$. The \emph{optimal Margulis number} of $M$ is
\[
\mu(M) = \sup \{\epsilon: \epsilon \text{ is a Margulis number for } M\}.
\]
The ($3$--dimensional) \emph{Margulis constant} is
\[
\mu_3 = \inf \{ \mu(M) : M \text{ is a hyperbolic $3$--manifold} \}.
\]
\end{definition}

Margulis proved that $\mu_3 > 0$, but the exact value is unknown. 
The following theorem summarizes current knowledge about Margulis numbers and the Margulis constant.

\begin{theorem}\label{Thm:NonsingularMargulis}
Suppose that $M$ is a non-singular hyperbolic $3$--manifold.
\begin{enumerate}
\item\label{Itm:MeyerhoffMarg} 
$\mu(M) \geq 0.104$ for every $M$, hence $\mu_3 \geq 0.104$.
\item\label{Itm:WeeksMarg} $\mu(M_W) \leq 0.776$ for the Weeks manifold $M_W$, hence $\mu_3 \leq 0.776$.
\item\label{Itm:Log3Marg} $\mu(M) \geq \log 3 = 1.098\ldots$ for every $M$  that has infinite volume, and for every $M$ such that $\dim H_1(M,\QQ) \geq 3$.
\item\label{Itm:CullerShalenMarg} $\mu(M) \geq 0.292$ for every $M$ such that $\dim H_1(M,\QQ) \geq 1$. This includes all non-closed hyperbolic $3$--manifolds.
\item\label{Itm:ShalenVolume} $\mu(M) \geq 0.29$  for every $M$ with $\vol(M) \geq 52.78$.
\item\label{Itm:ShalenFiniteness} $\mu(M) \geq 0.29$ for all but finitely many hyperbolic $3$--manifolds $M$.
\end{enumerate}
\end{theorem}

\begin{proof}[Proof by references]
Conclusion~\refitm{MeyerhoffMarg} is a theorem of Meyerhoff \cite[Section 9]{meyerhoff}. Conclusion~\refitm{WeeksMarg} 
is the result of a computation by Yarmola.

Conclusion~\refitm{Log3Marg} is a consequence of the ``$\log 3$ theorem'' of Culler and Shalen \cite[Theorem 9.1]{CullerShalen:Paradoxical}, combined with the Tameness and Density theorems for Kleinian groups~\cite{agol:tameness, calegari-gabai:tameness, namazi-souto:density, Ohshika05}.
See Shalen \cite[Proposition 3.12]{Shalen:SmallOptimalMargulis} for the derivation.

Conclusion~\refitm{CullerShalenMarg} is a theorem of Culler and Shalen \cite{culler-shalen:margulis}. Conclusion~\refitm{ShalenVolume} is a special case of a theorem of Shalen \cite[Theorem 7.1]{Shalen:SmallOptimalMargulis}, substituting $\lambda = 0.29$.
Finally, \refitm{ShalenFiniteness} is a theorem of Shalen \cite{shalen:margulis-numbers}, proved using \refitm{CullerShalenMarg} and an algebraic limit argument.
\end{proof}

Observe that \refthm{NonsingularMargulis}.\refitm{ShalenFiniteness} is an ineffective statement: the algebraic limit argument does not give any way to find the finite list of manifolds with $\mu(M) < 0.29$. On the other hand, combining \refthm{NonsingularMargulis}.\refitm{ShalenVolume} with our results gives the following effective theorem. In this theorem, $\systole(M)$ denotes the \emph{systole} of $M$, namely the length of the shortest closed geodesic in $M$.

\begin{named}{\refthm{MargulisDrilling}}
Let $M$ be a non-singular hyperbolic $3$--manifold. 
\begin{enumerate}
\item If $\mu(M) \leq 0.2408$, then $M$ is closed and  $\vol(M) \leq 36.12$. Furthermore, $\systole(M) \geq 2.93 \times 10^{-7}$.
\item If $\mu(M) \leq 0.29$, then $M$ is closed and  $\vol(M) \leq 52.78$. Furthermore, $\systole(M) \geq 2.73 \times 10^{-8}$.
\item If $\mu(M) \leq 0.9536$, then $M$ has finite volume and  $k \in \{0,1,2\}$ cusps. The $(3-k)$ shortest geodesics in $M$ have total length at least $5.561 \times 10^{-5}$.
 \end{enumerate}
\end{named}

We emphasize that parts \refitm{TinyMargulis} and \refitm{SmallMargulis} of the above statement are completely effective, because there exist algorithms to produce the finite list of manifolds with volume bounded above and systole bounded below by the given numbers. See Kobayashi and Rieck \cite{KobayashiRieck:tet-number} for the details.

The proof of \refthm{MargulisDrilling} uses many of the ingredients that were mentioned above. We will prove the contrapositive. For example, in the second case, suppose $M$ contains a geodesic $\sigma$ whose length is $\ell < 2.73 \times 10^{-8}$. Then there is a cone-deformation between $M$ and $N = M - \sigma$. 
\refthm{ThickStaysThick} says that for every $t$, the thin part $M_t^{< 0.29}$ must be contained in the thin part $N^{< 0.292}$, which is a union of tubes and cusps by \refthm{NonsingularMargulis}.\refitm{CullerShalenMarg}. Then, a somewhat delicate argument using immersed tubes (see \refthm{MargulisTopology}) shows that $M_t^{< 0.29}$ is also a union of tubes and cusps. In particular, $\mu(M) \geq 0.29$.

By a similar argument, we can show that long Dehn fillings of a cusped $3$--manifold $N$ have Margulis numbers similar to those of $N$.
 
 \begin{named}{\refthm{MargulisFilling}}
Fix $0 < \epsilon \leq \log 3$ and $1 < J \leq e^{1/5}$. Let $N$ be a cusped hyperbolic $3$--manifold such that $\epsilon$ is a Margulis number of $N$. Let $\mathbf s$ be a tuple of slopes on cusps of $N$ whose normalized length $L = L(\mathbf s)$ satisfies
\begin{equation*}
L(\mathbf s)^2 \geq \frac{2\pi \cdot 496.1 \, J^5 \cosh^5 (J \epsilon / 2 + 0.1475)}{\epsilon^5 \log J} + 11.7.
\end{equation*}
Then $\delta = \min \{ \epsilon/J, \, 0.962 \}$ is a Margulis number for $M = N(\mathbf s)$.
\end{named}

\subsection{Application to cosmetic surgeries}\label{Sec:CosmeticIntro}
The next application of our results is topological: we control cosmetic surgeries on $3$--manifolds.

\begin{definition}\label{Def:Cosmetic}
Let $N$ be a compact oriented $3$--manifold whose boundary is a single torus. Let $s_1, s_2$ be distinct slopes on $\bdy N$. We call $(s_1, s_2)$ 
a \emph{cosmetic surgery pair} if  there is a homeomorphism $\varphi \from N(s_1) \to N(s_2)$. 
The pair is called    \emph{chirally cosmetic} if $\varphi$ is orientation-reversing, and  is called
\emph{purely cosmetic} if $\varphi$ is orientation-preserving.
\end{definition}

There are many examples of chirally cosmetic surgeries where $N$ is Seifert fibered. See Bleiler--Hodgson--Weeks \cite{BleilerHodgsonWeeks} for a survey, and Ni--Wu \cite{NiWu:Cosmetic} for more examples. There is also one known example of a chirally cosmetic surgery pair where $N$ and $N(s_i)$ are hyperbolic, discovered by Ichihara and Jong \cite{IchiharaJong:CosmeticBanding}.
By contrast, no purely cosmetic surgeries are known apart from the case where $N$ is a solid torus. This has led Gordon \cite{Gordon:ICM} to propose

\begin{conjecture}[Cosmetic surgery conjecture]\label{Conj:Cosmetic}
Let $N$ be a compact, oriented $3$--manifold such that $\bdy N$ is an incompressible torus. If $s_1, s_2$ are a purely cosmetic pair of slopes on $\bdy N$, then $s_1 = s_2$.
\end{conjecture}

A well-known classical argument, recorded by Bleiler, Hodgson, and Weeks \cite{BleilerHodgsonWeeks}, implies that \refconj{Cosmetic} holds for long fillings on a hyperbolic manifold.

\begin{theorem}[Bleiler--Hodgson--Weeks \cite{BleilerHodgsonWeeks}]\label{Thm:FolkloreCosmetic}
Let $N$ be a one-cusped hyperbolic $3$--manifold. Then there is a number $E > 0$, such that 
\refconj{Cosmetic} holds for all pairs of slopes longer than $E$.
\end{theorem}

This useful but ineffective result is a fairly direct application of Thurston's Dehn surgery theorem and Mostow rigidity. 
By contrast, we prove the following effective result.

\begin{named}{\refthm{CosmeticOneCusp}}
Let $N$ be a one-cusped hyperbolic $3$--manifold.
Suppose that $s_1$ and $s_2$ are distinct slopes, such that the normalized length of each $s_i$ satisfies
\[ L(s_i) \geq \max \left\{ 10.1, \sqrt{\frac{2\pi}{\systole(N)} + 58} \right\}. \]
Then $(s_1, s_2)$ cannot be a purely cosmetic pair. If $(s_1, s_2)$ is a chirally cosmetic pair, then there is a homeomorphism of $N$ sending $s_1$ to $s_2$.
In particular, \refconj{Cosmetic} holds if $\systole(N)\geq 0.1428$ and $L(s_i) \geq 10.1$
for $i=1,2$. 
\end{named}

In fact, \refthm{CosmeticOneCusp} is a special case of a theorem that also holds for 
tuples of slopes. Our result addresses the following generalization of \refconj{Cosmetic}. 

\begin{conjecture}[Hyperbolic cosmetic surgery conjecture]\label{Conj:CosmeticMultiCusp}
Let $N$ be a finite-volume hyperbolic $3$--manifold with one or more cusps. 
Let $\mathbf{s}_1$ and $\mathbf{s}_2$ be tuples of slopes on the cusps of $N$. If there is an orientation-preserving homeomorphism $\varphi \from N(\mathbf{s}_1) \to N(\mathbf{s}_2)$, and this manifold is hyperbolic, then $\varphi$ restricts (after an isotopy) to a homeomorphism $N \to N$ sending $\mathbf{s}_1$ to $\mathbf{s}_2$.
\end{conjecture}

Compare Kirby \cite[Problem 1.81(B)]{Kirby:Problems} and Jeon \cite[Section 1.1 and Theorem 1.6]{Jeon:ZilberPinkCosmetic} for related statements. The above-mentioned result of Ichihara and Jong \cite{IchiharaJong:CosmeticBanding} shows that restricting to purely cosmetic surgeries is necessary, even when $N$ has a single cusp. 

We prove \refconj{CosmeticMultiCusp} for sufficiently long tuples of slopes, where ``long'' is explicitly quantified. For these long tuples of slopes, the only purely or chirally cosmetic surgeries come from symmetries of $N$ itself.

\begin{named}{\refthm{Cosmetic}}
Let $N$ be a hyperbolic $3$--manifold with cusps. Suppose that $\mathbf s_1, \mathbf s_2$ are distinct tuples of slopes on the cusps of $N$, whose normalized length satisfies
\[
L(\mathbf s_i) \geq \max \left \{ 10.1, \, \sqrt{ \frac{2\pi}{\systole(N)}  + 58 } \right \}.
\]
Then any homeomorphism $\varphi \from N(\mathbf s_1) \to N(\mathbf s_2)$ restricts (after an isotopy) to a self-homeomorphism  of $N$ sending $\mathbf s_1$ to $\mathbf s_2$. 
\end{named}

Given any lower bound on the systole of $N$, Theorems~\ref{Thm:Cosmetic} and~\ref{Thm:CosmeticOneCusp} provide an effective estimate on the normalized length after which the cosmetic surgery conjecture holds. However, even if $N$ has just one cusp, there could hypothetically be infinitely many purely cosmetic pairs, where $s_1$ is short but $s_2$ is arbitrarily long. This possibility is ruled out in \refthm{EffectiveList} below. Before setting up that result, we treat the special case where $N$ is the complement of a knot in $S^3$. The following result follows by combining \refthm{CosmeticOneCusp} with the work of Ni and Wu \cite{NiWu:Cosmetic}.

\begin{corollary}\label{Cor:CosmeticKnot}
Let $K \subset S^3$ be a hyperbolic knot. Let $\mu, \lambda$ be the meridian and longitude of $K$. Suppose that $s, s'$ are a purely cosmetic pair  on the cusp of $S^3 - K$. 
Then, after possibly swapping $s$ and $s'$, the following holds.
\begin{enumerate}

\smallskip
\item\label{Itm:CosmeticShort}
\(\displaystyle{
L(s) < 
\max \left \{ 10.1, \, \sqrt{ \frac{2\pi}{\systole(S^3 - K)}  + 58 } \right \}. }
\)
\smallskip
\item\label{Itm:CosmeticNegatives} If $s = p\mu + q \lambda$, then $s' = p\mu - q\lambda$, and furthermore $p$ divides $q^2 +1$.
\end{enumerate}
\end{corollary}

\begin{proof}
Conclusion \refitm{CosmeticShort} is a restatement of \refthm{CosmeticOneCusp}. 
Meanwhile, Conclusion \refitm{CosmeticNegatives} is part of the statement of \cite[Theorem 1.2]{NiWu:Cosmetic}.
\end{proof}

To verify \refconj{Cosmetic} for a given manifold $N = S^3 - K$, it suffices to check the finitely many pairs $(s, -s)$ where $L(s)$ satisfies \refitm{CosmeticShort}. This is a practical computational task \cite{FPS:Cosmetic}.
In practice, the vast majority of knot complements enumerated by Hoste, Thistlethwaite, and Weeks \cite{HTW:knots-16-crossings} have systole greater than $0.15$, which means that the normalized length cutoff in the corollary is $10.1$.
By work of Agol~\cite[Lemma~8.2]{agol:6theorem}, there are at most 104 slopes on $\bdy N$ of normalized length less than $10.1$. 
Among those short slopes, there are typically at most $8$ slopes 
that have the form $s = p\mu + q \lambda$ where 
$p$ divides $q^2 +1$. Thus checking the cosmetic surgery conjecture for a typical knot $K$ amounts to distinguishing $8$ or fewer pairs of closed manifolds. 
We ran a computer program  to show the following:

\begin{corollary}\label{Cor:CosmeticComputer}
  The cosmetic surgery \refconj{Cosmetic} holds for all prime knots with at most 16 crossings. 
\end{corollary}

Contemporaneously with our work, Hanselman has obtained an independent proof of \refcor{CosmeticComputer} \cite{Hanselman:Cosmetic}.
To do this, he proved a finiteness theorem in the same spirit as \refcor{CosmeticKnot}, constraining the slopes to check to a short (and frequently empty) list that depends on the knot genus $g(K)$ and the thickness of the knot Floer homology of $K$.
For knots up to $16$ crossings, his criterion only requires checking slopes $\pm 1$ and $\pm 2$ for 337 knots.

Very recently, Detcherry discovered a criterion on the Jones polynomial at the fifth root of unity that severely constrains the slopes involved in cosmetic surgeries \cite{Detcherry:Cosmetic}. By combining his criterion with Hanselman's results, he verified \refconj{Cosmetic} for knots up to 17 crossings.

\begin{proof}[Sketch proof of \refcor{CosmeticComputer}]
For each of the 1,701,935 prime knots with at most 16 crossings, we begin by computing the (symmetrized) Alexander polynomial $\Delta_K(t)$ and its second derivative. By a theorem of Boyer and Lines \cite[Proposition 5.1]{BoyerLines}, any knot $K$ such that $\Delta_K''(1) \neq 0$ has no purely cosmetic surgery. This criterion eliminates 1,513,776 knots, roughly $89\%$ of the total.

For the remaining knots, we compute the Jones polynomial $V_K(t)$ and its third derivative. By a theorem of Ichihara and Wu \cite[Theorem 1.1]{IchiharaWu}, any knot $K$ such that $V_K'''(1) \neq 0$ has no purely cosmetic surgery. This criterion eliminates another 152,740 knots, roughly $9\%$ of the total.

The remaining 35,419 knots are all hyperbolic. For each knot $K$ on the remaining list, we compute $\systole(N) = \systole(S^3 - K)$ and calculate the list of short slopes satisfying the conclusion of \refcor{CosmeticKnot}. For each short slope $s$, we compute the hyperbolic structures on $N(s)$ and $N(-s)$ using SnapPy
and compare the verified volume and verified Chern--Simons invariants. In each case, these invariants distinguish the pair. Code will be included with \cite{FPS:Cosmetic}.
\end{proof}

Returning to the setting of a general one-cusped hyperbolic manifold, we describe a practical finiteness theorem that can be used to verify whether $N$ has any cosmetic surgeries at all. We need the following definition.

\begin{definition}\label{Def:SetsToCheck}
Let $N$ be a one-cusped hyperbolic $3$--manifold. Define the finite set of slopes
\[
\calS_1(N) = \left\{ s \: \bigg\vert \: L(s) < \max \left ( 10.1, \, \sqrt{ \frac{2\pi}{\systole(N)}  + 58 } \right ) \right\}.
\]
Next, define 
\[
V(N) = \max \big \{ \vol(N(s)) \mid s \in \calS_1(N) \big \}.
\]
Here, we employ the convention that the volume of any non-hyperbolic manifold is $0$, hence will not realize the maximum. 
Using the theorem of 
Gromov and Thurston \cite[Theorem 6.5.6]{thurston:notes} that $V(N) < \vol(N)$, define the finite set of slopes
\[
\calS_2(N) = \left\{ s \: \bigg\vert \: \len(s) \leq 2\pi \left( 1 - \left(\frac{V(N)}{\vol(N)}\right)^{2/3}\right)^{-1/2} \right\}.
\]
Here, $\len(s)$ is ordinary Euclidean length on the boundary of the maximal cusp in $N$.
\end{definition}

\begin{theorem}\label{Thm:EffectiveList}
Let $N$ be a one-cusped hyperbolic $3$--manifold.  Then each of \refconj{Cosmetic} and \refconj{CosmeticMultiCusp} holds
for $N$ if and only if it holds for all pairs of slopes in the finite set $\calS_1(N) \times \calS_2(N)$.
\end{theorem}

\begin{proof}
The ``only if'' direction is obvious.  For the ``if'' direction, suppose that $(s_1, s_2)$ are a (purely or chirally) cosmetic pair for $N$. Assume, without loss of generality, that $L(s_1) \leq L(s_2)$. Then \refthm{CosmeticOneCusp} implies $s_1 \in \calS_1(N)$. With $V(N)$ as above, Futer, Kalfagianni, and Purcell proved \cite[Theorem~1.1]{fkp:volume} that if some Dehn filling $N(s)$ satisfies $\vol(N(s)) \leq V(N)$, then $s \in \calS_2(N)$. Thus any potential counterexample to Conjectures~\ref{Conj:Cosmetic} or~\ref{Conj:CosmeticMultiCusp} %
must lie in $\calS_1(N) \times \calS_2(N)$.
\end{proof}

For manifolds with reasonable systole, the set $\calS_1(N) \times \calS_2(N)$ is practical to compute using SnapPy, and not too large in size. In forthcoming work  \cite{FPS:Cosmetic}, we use \refthm{EffectiveList} to verify \refconj{Cosmetic}  for all one-cusped manifolds in the SnapPy census. We also use the work of Detcherry~\cite{Detcherry:Cosmetic}  and Hanselman~\cite{Hanselman:Cosmetic} to verify the conjecture for all knots up to 19 crossings.

\subsection{Organization}

In \refsec{ConeMfdBasics}, we review background on cone-manifolds and their properties. \refsec{TubeDist} reviews a number of results on tubes in cone-manifolds, as well as distances between nested tubes; many of these results were proved in \cite{FPS:Tubes}. In \refsec{MultiTube}, we control the areas  of embedded multi-tubes, analogous to similar results in \cite{hk:univ, hk:shape}, which will ensure that cone deformations exist. We also prove \refthm{MaxTubeInjectivity}, which controls the injectivity radius on the tube boundary and may be of independent interest. 

The above results are combined in \refsec{ConeDef} to produce cone deformations that maintain a large embedded tube about the singular locus $\Sigma$. The technical \refsec{BoundaryBound} presents results (phrased as bounds on so-called \emph{boundary terms}) that will be needed in future sections to control the change in geometry during the cone deformation.  

\refsec{ShortGeodesic} contains the first pair of our main results: \refthm{ShortStaysShort} and \refthm{ShortStaysShortUpward}, which bound the change in length of short geodesics under cone deformation. At the end of \refsec{ShortGeodesic}, we apply these results to the cosmetic surgery conjecture, proving Theorems~\ref{Thm:Cosmetic} and~\ref{Thm:CosmeticOneCusp}.

\refsec{Bilip} proves \refthm{Bilip}, which provides effective bilipschitz bounds on submanifolds of $M$ that stay thick throughout the cone deformation. 

\refsec{Margulis} contains results related to the thick-thin decomposition and Margulis numbers. A main result, \refthm{ThickStaysThick}, ensures that submanifolds of $M$ stay thick throughout the deformation. This result is applied to show \refthm{MargulisDrilling} about Margulis numbers, as well as \refthm{BilipEndpoints}, which provides bilipschitz bounds without any hypotheses on cone-manifolds.

Finally, there is a short appendix, on hyperbolic trigonometry, that we use throughout the paper. 

\subsection{Acknowledgements}
As should be clear, this paper owes an enormous debt to the ideas of Jeff Brock, Ken Bromberg, Craig Hodgson, and Steve Kerckhoff. We thank all four of them for sharing their insights and explaining many technical points. 

We thank Steve Boyer, Jonathan Hanselman, Kazuhiro Ichihara, Tye Lidman, Tim Morris, and Matthew Stover for many enlightening suggestions regarding cosmetic surgeries. We thank Marc Culler, David Gabai, Peter Shalen, and Andrew Yarmola for helpful input about Margulis numbers. Finally, we thank the intrepid referee for carefully reading the paper and making a number of suggestions that improved our exposition.

During this project, Futer was supported by the NSF and the Simons Foundation. Purcell was supported by the NSF and the ARC.  Schleimer was supported by EPSRC. All three authors were able to gather in a common location on a few occasions, thanks to support from MSRI and the Institute for Advanced Study, as well as support from NSF grants DMS--1107452, 1107263, 1107367, ``RNMS: Geometric Structures and Representation Varieties'' (the GEAR Network).

\section{Cone-manifold basics}\label{Sec:ConeMfdBasics}

In this section we set up notation and definitions about cone-manifolds, geodesics, and injectivity radii.

\begin{definition}\label{Def:SingularModel}
Let $\sigma \subset \HH^3$ be a bi-infinite geodesic. Let $\Hhat$ denote the metric completion of the universal cover of $(\HH^3 - \sigma)$. Let $\bcover{\sigma}$ be the set of points added in the completion.

The space $\Hhat$ can be regarded as an infinite cyclic branched cover of $\HH^3$, branched over $\sigma$. The branch set $\bcover{\sigma} \subset \Hhat$ is a singular geodesic of cone angle $\infty$.

There is a natural action of $\CC$ (thought of as an additive group) on $\Hhat$, where $z = \zeta+ i \theta \in \CC$ translates $\bcover{\sigma}$ by distance $\zeta$ and rotates by angle $\theta$. Since $\bcover{\sigma}$ has cone angle $\infty$, angles of rotation are indeed real-valued. Conversely, every isometry $\varphi$ of $\Hhat$ that preserves orientation on both $\Hhat$ and $\bcover{\sigma}$ comes from this action, and has a well-defined \emph{complex length} $z = \zeta+i\theta$. We can therefore write $\varphi = \varphi_{\zeta+i\theta}$.

We endow $\Hhat$ with a system of cylindrical coordinates $(r,\zeta, \theta)$, as follows. Choose a ray perpendicular to $\bcover{\sigma}$, and let the points of this ray have coordinates $(r,0,0)$, where $r \geq 0$ measures distance from $\bcover{\sigma}$. Then, let $(r,\zeta,\theta)$ be the image of $(r,0,0)$ under the isometry $\varphi_{\zeta+i\theta}$. The distance element in these coordinates is
\begin{equation}\label{Eqn:CylindricalCoords}
ds^2 = dr^2 + \cosh^2 r \, d\zeta^2 + \sinh^2 r \, d\theta^2.
\end{equation}
\end{definition}

\begin{definition}\label{Def:ModelSolidTorus}
Consider a group $G = \ZZ \times \ZZ$ of isometries of $\Hhat$, generated by an elliptic $\psi_{i\alpha}$ and a loxodromic $\varphi = \varphi_{\lambda+i\tau}$, where $\alpha > 0$ and $\lambda > 0$. The quotient space $N_{\alpha,\lambda,\tau}$ is an open solid torus whose core curve is a closed geodesic of length $\lambda$, and with a cone singularity of angle $\alpha$ at the core. We call $N = N_{\alpha,\lambda,\tau}$ a \emph{model solid torus}.

\label{Def:ModelTube}
For $r > 0$, a \emph{model tube} $U_{\alpha,\lambda,\tau}$ is the open $r$--neighborhood of the core curve of $N_{\alpha,\lambda,\tau}$. We note that the closure $\overline{U}_{\alpha,\lambda,\tau}$ is compact.
\end{definition}

\begin{definition}\label{Def:ConeManifold}
A \emph{hyperbolic cone-manifold} $(M, \Sigma)$ is a metric space where every point has a neighborhood isometric to a ball in a model solid torus. More precisely, $M$ is a topological $3$--manifold and $\Sigma$ is a link in $M$, such that every component of $\Sigma$ has a neighborhood isometric to a model tube $U_{\alpha,\lambda,\tau}$.
Meanwhile, every point $x \in M - \Sigma$ has a neighborhood isometric to a ball in $\HH^3$. 

We allow components of $\Sigma$ to be non-singular, i.e.\ have cone angle $2\pi$. 
When the link $\Sigma$ is clear from context, we will often suppress it from the notation. 
\end{definition}

\subsection{Covers and deck transformations}

\begin{definition}\label{Def:UniversalBranchedCover}
Let $(M, \Sigma)$ be a hyperbolic cone-manifold. Then the \emph{universal branched cover} of $(M, \Sigma)$, denoted $\bcover{M}$, is the metric completion of $\widetilde{X}$, where $\widetilde{X}$ is the universal cover of $X = M - \Sigma$. 
Every component of $\Sigma$ lifts to a disjoint union of singular geodesics in $\bcover{M}$, with cone angle $\infty$. Thus $\bcover{M}$ is locally modeled on $\Hhat$, as in \refdef{SingularModel}.
The deck transformation group for $\bcover{M}$ is isomorphic to $\pi_1(M-\Sigma)$.

Let $\widetilde{\sigma} \subset \bcover{M}$ be a lift of $\sigma \subset \Sigma$. Then there is a map
\[
D\from \bcover{M} \to \Hhat,
\]
which shares some features of an exponential map based on the normal bundle to a geodesic.
   Let $\widetilde U \subset \bcover M$ be a regular $r$--neighborhood of $\widetilde \sigma$, and let $D \from \widetilde U \hookrightarrow \Hhat$ be an isometric embedding sending $\widetilde \sigma$ to $\widehat \sigma \subset \Hhat$. Then, extend the map along geodesic rays: 
if $\gamma \subset \bcover{M}$ is a geodesic ray orthogonal to $\widetilde \sigma$, then $D \vert_\gamma$ is an isometry to a geodesic ray orthogonal to $\bcover{\sigma}$. 

In the special case where $M$ is a model solid torus with core curve $\Sigma$, we have $\bcover{M} = \Hhat$ by \refdef{ModelSolidTorus}, hence $D$ is a global isometry. On the other hand, if $(M,\Sigma)$ is non-elementary, hence a singular geodesic $\sigma \subset M$ has multiple preimages in $\bcover{M}$, the map $D$ will fail to be even a \emph{local} isometry outside a neighborhood of $\widetilde{\sigma}$. Indeed, given a geodesic segment $\gamma_0 \subset \bcover{M}$ that runs from $\widetilde \sigma$ to another singular geodesic, there is a one-parameter family of distinct geodesic rays in $\bcover{M}$ that all contain $\gamma_0$ and then separate; these rays will be mapped to the unique ray in $\Hhat$ containing $D(\gamma_0)$.
\end{definition}

An important property of the universal branched cover is

\begin{proposition}\label{Prop:CAT(-1)}
Let $(M, \Sigma)$ be a hyperbolic cone-manifold. Then $\bcover M$ is a complete $\CAT$ space.
\end{proposition}

\begin{proof}
This is a consequence of the Cartan--Hadamard theorem \cite[Chapter II.4, Theorem 4.1(2)]{bridson-haefliger}. See Soma \cite[Lemma 1.2]{Soma:Shrinkwrapping} for the derivation.
\end{proof}

If $M$ is a finite-volume hyperbolic manifold and $\Sigma \subset M$ is a geodesic link, Kerckhoff showed that $M - \Sigma$ admits a complete metric of negative sectional curvature. (See Agol \cite{agol:drilling} for a summary and for more details of the construction.) The same construction applies if $(M, \Sigma)$ is a cone-manifold of finite volume. By Thurston's hyperbolization, this implies $M - \Sigma$ admits a complete hyperbolic metric.

\begin{definition}\label{Def:Horocusp}
Let $G \subset \Isom^+(\HH^3)$ be a discrete, free abelian group of parabolic isometries of $\HH^3$. The quotient $\HH^3 / G$ is called a \emph{model cusp}. If $H \subset \HH^3$ is an open horoball stabilized by $G$, the quotient $U = H/G$ is called a \emph{horocusp}. A horocusp or model cusp is called \emph{rank one} if $G \cong \ZZ$ and \emph{rank two} if $G \cong \ZZ^2$.
\end{definition}

\begin{definition}\label{Def:Peripheral}
Let $(M, \Sigma)$ be a complete, finite volume hyperbolic cone-manifold.
Let $\varphi \in \pi_1(M - \Sigma)$ be a non-trivial element. We say that $\varphi$ is \emph{peripheral} if a loop representing $\varphi$ is freely homotopic into a horocusp of $M - \Sigma$.
\end{definition}

\begin{lemma}\label{Lem:PeripheralParabolic}
Let $(M, \Sigma)$ be a complete, finite volume hyperbolic cone-manifold, with universal branched cover $\bcover M$.  Let $\varphi \in \pi_1(M - \Sigma)$ be a non-trivial deck transformation of $\bcover M$. Then the following are equivalent:
\begin{enumerate}
\item \label{Itm:Peripheral} $\varphi$ corresponds to a peripheral homotopy class.
\item \label{Itm:ParabolicOrSingular} $\varphi$ stabilizes either a horoball in $\bcover M$ or a singular geodesic covering a component of $\Sigma$.
\item \label{Itm:NotLoxodromic} $\varphi$ does not act by translation on any non-singular geodesic in $\bcover M$.
\end{enumerate}
\end{lemma}

\begin{proof}
$\refitm{Peripheral} \Rightarrow \refitm{ParabolicOrSingular}$: Suppose $\varphi$ is peripheral. Then a loop representing $\varphi$ is homotopic into the neighborhood of some cusp of $M - \Sigma$. In the cone-metric on $(M, \Sigma)$, this cusp of $M - \Sigma$ either stays a cusp or becomes a neighborhood of some component $\sigma_i \subset \Sigma$. In the first case, the deck transformation  $\varphi$ stabilizes the universal cover of a horocusp in $\bcover M$. In the second case, the deck transformation $\varphi$ stabilizes some preimage of $\sigma_i$ in $\bcover M$.

\smallskip 
$ \refitm{ParabolicOrSingular} \Rightarrow \refitm{Peripheral}$: Suppose that the deck transformation $\varphi$ stabilizes a horoball  $H \subset \bcover M$ that covers a horocusp in $(M, \Sigma)$. Then, for an appropriate choice of basepoint, a path-lift $\widetilde \varphi$ of the loop $\varphi$ starts and ends in $H$. After a free homotopy of $\varphi$, we may assume that the entire path-lift $\widetilde \varphi$ lies in a sub-horoball of $H$ that covers a horocusp of $(M, \Sigma)$, hence $\varphi$ is peripheral. The case of tubes is similar.


\smallskip
$ \refitm{ParabolicOrSingular}  \Leftrightarrow \refitm{NotLoxodromic} $: The deck transformation $\varphi$ must act on $\bcover M$ by isometry. This isometry is either \emph{elliptic} (meaning it has fixed points), \emph{parabolic} (meaning it has no fixed points, but the infimal translation is $0$), or \emph{hyperbolic} (meaning that the infimal translation distance is $d>0$ and is realized). If $\varphi$ is elliptic, then recalling that it is a deck transformation implies that it must rotate about a singular axis; hence \refitm{ParabolicOrSingular} and \refitm{NotLoxodromic} both hold. If $\varphi$ is parabolic, then it stabilizes a horoball, hence \refitm{ParabolicOrSingular} and \refitm{NotLoxodromic} both hold. 

Finally, suppose $\varphi$ is hyperbolic. Since $\bcover M$ is a complete $\CAT$ space by \refprop{CAT(-1)}, $\varphi$ must  translate along a unique geodesic axis. If this axis is singular, then \refitm{ParabolicOrSingular} and \refitm{NotLoxodromic} both hold. If the geodesic is non-singular, then \refitm{ParabolicOrSingular} and \refitm{NotLoxodromic} both fail.
\end{proof}

\subsection{Injectivity radii, cusps, and tubes}

We will be looking at tubes with injectivity radius less than some value $\epsilon > 0$ or $\delta>0$. At times we will need to discuss the injectivity radius over an entire cone-manifold. At other times we only need to consider injectivity radius within a tube. We encapsulate these separate notions in \refdef{Injectivity} and \refdef{InjRadTube}. 

\begin{definition}\label{Def:Injectivity}
Let $(M, \Sigma)$ be a hyperbolic cone-manifold and $x \in M$. Then the \emph{injectivity radius}, denoted $\injrad(x)$, is the supremal radius $r$ such that a metric $r$--ball about $x$ is isometric to a ball $B_r(y) \subset \HH^3$. (Since we are using open balls, the supremal radius is attained, unless $M = \HH^3$, in which case $\injrad(x) = \infty$.) If $x$ lies on the singular set of $M$, we set $\injrad(x) = 0$.
\end{definition}

\begin{lemma}\label{Lem:InjectivityLoop}
Let $(M, \Sigma)$ be a hyperbolic cone-manifold, where every component of $\Sigma$ is singular. Choose a point $x \in M - \Sigma$ and a lift $\widetilde x \in \bcover M$. Then  $\injrad(x)$ can be characterized as follows:
\begin{align}
2 \,  \injrad(x)
& = \inf \{ \len(\gamma) : \gamma \mbox{ is a non-trivial loop in $M - \Sigma$ based at } x \} \label{Eqn:InjectivityLoop} \\
& =  \min \{ d( \widetilde x, \varphi \widetilde x)
 : 1 \neq \varphi \in \pi_1 (M - \Sigma)  \}. \label{Eqn:InjectivityTrans}
 \end{align}
Furthermore, the infimum in \refeqn{InjectivityLoop} is realized by a pointed geodesic, unless $d(x,\Sigma) = \injrad(x)$ and the isometry $\varphi$ in \refeqn{InjectivityTrans} is elliptic.
\end{lemma}

In the case where $M$ is a non-singular hyperbolic manifold, hence $\Sigma = \emptyset$, the result of \reflem{InjectivityLoop} is well-known. In the case where $M$ is a model solid torus, the result is contained in  \cite[Lemma 2.5]{FPS:Tubes}. Thus \reflem{InjectivityLoop} generalizes those previously known cases to general cone-manifolds.

\begin{proof}[Proof of \reflem{InjectivityLoop}]
Let $\epsilon = 2\injrad_\Sigma(x)$. Then, for an arbitrary $y \in \HH^3$, there is an isometric embedding $f\from B_{\epsilon/2}(y) \to M $, such that $f(y) = x$. It follows that any non-trivial loop through $x$ must have length at least $ \epsilon$. Similarly, any non-trivial element $\varphi \in \pi_1 (M - \Sigma)$ must translate $B_{\epsilon/2}(\widetilde x)$ by distance at least $\epsilon$. Thus both \refeqn{InjectivityLoop} and \refeqn{InjectivityTrans} give lower bounds on $\epsilon$.

Next, we show that these expressions give upper bounds on $\epsilon$.
Since $\injrad(x) = \epsilon/2$,  the continuous extension of $f$ to $\overline{B_{\epsilon/2}(y)}$ either hits $\Sigma$ or fails to be $1$--$1$. We consider these cases in turn.

First, suppose that the image ball $f(\overline{B_{\epsilon/2}(y)})$ has a point of self-tangency in $M - \Sigma$. This means that two distinct lifts of this ball, namely $B_{\epsilon/2}(\widetilde x)$ and $B_{\epsilon/2}(\varphi \widetilde x)$, are tangent in $\bcover M$, which means that $d( \widetilde x, \varphi \widetilde x) = \epsilon$. The geodesic $\widetilde \gamma$ connecting $\widetilde x$ to $\varphi \widetilde x$ projects to a geodesic loop  $\gamma \subset M - \Sigma$ of length exactly $\epsilon$. This means that \refeqn{InjectivityLoop} is an equality, hence the infimum is realized in this case.

Next, suppose there is a point $z \in  f(\overline{B_{\epsilon/2}(y)}) \cap \Sigma$. Then we construct a closed loop $\gamma$ of length $\epsilon + \delta$, for arbitrarily small $\delta$. This closed loop has the form of an ``eyeglass'': walk from $x$ to a point near $z$, walk around a loop of length $\delta$ about $\Sigma$, and then return to back to $x$. Thus \refeqn{InjectivityLoop} is an upper bound on $\epsilon$. The homotopy class $[\gamma] \in \pi_1 (M - \Sigma, x)$ corresponds to an elliptic isometry of $\bcover M$, which fixes a lift $\widetilde z$ of $z$. This elliptic isometry must move the ball  $B_{\epsilon/2}(\widetilde x)$ to a disjoint ball $B_{\epsilon/2}(\varphi \widetilde x)$, with the two balls tangent at $\widetilde z$. Thus $d( \widetilde x, \varphi \widetilde x) = \epsilon$.
\end{proof}

\begin{definition}\label{Def:ThickThin}
Let $(M, \Sigma)$ be a hyperbolic cone-manifold.
For $\epsilon > 0$, the \emph{$\epsilon$--thick part} of $M$ is 
\[
M^{\geq \epsilon} = \{ x \in M : \injrad(x) \geq \epsilon / 2 \} .
\]
The \emph{$\epsilon$--thin part} is $M^{< \epsilon} = M - M^{\geq \epsilon}$.
We define $M^{\leq \epsilon}$ and $M^{> \epsilon}$ similarly.

We emphasize that our definition of the \emph{$\epsilon$--thick part} corresponds to injectivity radius $\epsilon/2$ (hence, translation length $\epsilon$) rather than injectivity radius $\epsilon$.
 Both choices seem to be common in the literature on Kleinian groups. Our convention agrees with that of Minsky~\cite{minsky:punctured-tori, minsky:models-bounds} and Brock--Canary--Minsky~\cite{brock-canary-minsky:elc}, while differing from the convention of Brock--Bromberg \cite{brock-bromberg:density} and Namazi--Souto \cite{namazi-souto:density}. 
\end{definition}

\begin{definition}\label{Def:MargulisNum}
  Let $(M, \Sigma)$ be a hyperbolic cone-manifold. We say that $\epsilon > 0$ is a \emph{Margulis number} for $M$ if every component of the $\epsilon$--thin part $M^{< \epsilon}$ is isometric to either a model tube (\refdef{ModelTube}) or a horocusp (\refdef{Horocusp}). 
The \emph{optimal Margulis number} $\mu(M)$ is the supremum of all Margulis numbers for $M$. 
\end{definition}

In \refthm{MargulisConeMfld}, we will prove an effective Margulis lemma for cone-manifolds: $0.29$ is a Margulis number for all cone-manifolds satisfying certain hypotheses. See also \refthm{MargulisConeMedConst}.

In addition to studying embedded tubes and cusps in a cone-manifold $M$, we will study their immersed analogues.

\begin{definition}\label{Def:ImmersedTube}
Let $(M, \Sigma)$ be a hyperbolic cone-manifold. An \emph{immersed tube in $M$} is a local isometry $f \from U \to M$, where $U$ is a model tube, and furthermore $f^{-1}(\Sigma)$ is either the core of $U$ or $\emptyset$. (The ``furthermore'' condition is automatic when each component of $\Sigma$ is singular.)
Similarly, an \emph{immersed horocusp in $M$} is a local isometry $f \from U \to M$, where $U$ is a horocusp.

If $f$ is an embedding, we refer to the image $f(U)$ as an \emph{embedded tube} in $M$. 
In this case, we will often conflate the domain $U$ with the image $f(U)$.
\end{definition}

\begin{definition}\label{Def:InjRadTube}

Let $(M,\Sigma)$ be a hyperbolic cone-manifold. 
Let $U \subset M$ be an embedded tube or horocusp in $M$. Let $\pi \from \bcover U \to U$ be the universal covering map. Thus $\pi$ is an ordinary cover with deck group $G \cong \ZZ$ if $U \cap \Sigma = \emptyset$, and a branched cover with deck group $G \cong \ZZ \times \ZZ$ if $U \cap \Sigma \neq \emptyset$. If $U$ is a horocusp, then $\bcover U$ is the usual universal cover  with deck group $G \in \{ \ZZ, \ZZ^2 \}$.

Let $x\in U$. Let $\bcover{x}$ be a lift of $x$ in $\bcover{U}$, and consider all translates of $\bcover{x}$ under the action of $G$. Define
\begin{equation}\label{Eqn:InjRadTube} 
\injrad(x,U) = \frac{1}{2} \min \{d(\bcover{x}, \varphi(\bcover{x})) :1 \neq  \varphi \in G  \} .
\end{equation}
For $x\in\bdy U$, we may define $\injrad(x,U)$ by extending \refeqn{InjRadTube} by continuity.
\end{definition}

When every component of $\Sigma$ is singular, the injectivity radius is well-behaved under immersions of tubes or cusps.

\begin{lemma}\label{Lem:InjRadRelation}
  Let $(M,\Sigma)$ be a hyperbolic cone-manifold, where every component of $\Sigma$ is presumed to be singular.
Let $f \from U \to M$ be an immersed tube or horocusp.
Then, for all $x \in \overline{U}$,
  \[ \injrad(f(x)) \leq \injrad(x,U) .  \]
\end{lemma}

\begin{proof}

By continuity, it suffices to assume that $x \in U$.
As in \refdef{InjRadTube}, let $\bcover U$ be the universal cover of $U$, which is branched if $U$ is singular. Let $G$ be the group of deck transformations of $\bcover U$.

The local isometry $f \from U \to M$ gives an elevation, a local isometry $\bcover f \from \bcover U \to \bcover M$. We claim that $\bcover f$ must be one-to-one: since $\bcover U$ is convex, any pair of points are connected by a geodesic segment. The image of this segment is a geodesic segment in $\bcover M$, which necessarily has distinct endpoints by \refprop{CAT(-1)}. Since $\bcover f$ is one-to-one, we get  an inclusion $f_* \from G \hookrightarrow \pi_1(M - \Sigma)$. (Compare Baker and Cooper \cite[Propositions 2.1 and 2.2]{BakerCooper:Combination}.)
Thus the minimum in \refeqn{InjRadTube} is taken over a smaller set than the minimum in \refeqn{InjectivityTrans}, hence $\injrad(f(x)) \leq \injrad(x,U)$.
\end{proof}

\section{Distance estimates in tubes and cusps}\label{Sec:TubeDist}

This section contains several estimates about tubes and cusps that will be needed in subsequent arguments. Most of the results listed here are proved in \cite{FPS:Tubes}. We begin with a general estimate that applies to all cone-manifolds.

\begin{lemma}\label{Lem:LowerBoundEasy}
Let $M$ be a hyperbolic cone-manifold. Let $x,y$ be points of $M$ such that $2\injrad(x) = \delta > 0$ and $2 \injrad(y) = \epsilon > \delta$. Then
\[
d(x,y) \geq \frac{\epsilon - \delta}{2}.
\]
\end{lemma}

\begin{proof}
This was observed in \cite[Lemma 5.1]{FPS:Tubes} in the case where $M$ is a model tube. The same proof works in general.

Let $h = d(x,y)$. If $h \geq \epsilon/2$, there is nothing to prove. Thus we may assume that $h < \epsilon/2$. 
By \refdef{Injectivity}, there is an embedded ball $B = B_{\epsilon/2}(y)$ that is isometric to a ball in $\HH^3$. Since $h < \epsilon/2$, we have $x \in B$. By the triangle inequality, there is an embedded ball $B_{\epsilon/2 - h}(x)$ contained in $B$, implying
\[
 \injrad(x) = \delta/2 \geq \epsilon/2 - h. \qedhere
\]
\end{proof}

\subsection{Tube radii}

\begin{definition}\label{Def:TubeRad}
Let $N=N_{\alpha,\lambda,\tau}$ be a model solid torus, as in \refdef{ModelSolidTorus}. For $\epsilon > \lambda$, let $U^\epsilon$ be a component of $N^{< \epsilon}$. Then $U^\epsilon$ is  a tube about a core geodesic $\gamma$, and $T^\epsilon = \bdy U^\epsilon$ is a torus consisting of points whose injectivity radius is exactly $\epsilon/2$. All of the points of  $T^\epsilon$ lie at the same radius from $\gamma$. We denote this radius
\[
r(\epsilon) = r_{\alpha, \lambda, \tau} (\epsilon).
\]
We let $T_r$ denote the equidistant torus at radius $r$ from the core of $N$. Subscripts denote radius, while superscripts denote thinness. Thus
\[
T^\epsilon = T_{r(\epsilon)}.
\]

If $N$ is a model solid torus, modeling a neighborhood of $\sigma_i$, a component of $\Sigma$ in $M$, we will often write $r(\epsilon,\sigma_i)$, $T^\epsilon(\sigma_i)$, $U^\epsilon(\sigma_i)$ to refer to the radius, equidistant torus, and tube (respectively) about a particular component of $\Sigma$. 
\end{definition}

\begin{lemma}\label{Lem:EpsilonArea}
Let $ T^\epsilon = \bdy U^\epsilon$, where $U^\epsilon$ is a tube about a singular geodesic $\sigma$ with cone angle  $\alpha < 2\pi$.
Then
\[
\area(T^\epsilon) \geq \frac{\sqrt{3}}{2}  \epsilon^2.
\] 
\end{lemma}

\begin{proof}
This follows by a standard packing argument, because $T^\epsilon$ contains an embedded disk of radius $\epsilon/2$. See \cite[Equation~(7.4)]{FPS:Tubes}. Compare \refprop{Ellipse} and \refthm{MaxTubeArea} for a much more involved packing argument.
\end{proof}


\begin{lemma}\label{Lem:SingularTubeRad}
Let $N = N_{\alpha, \lambda, \tau}$ be a model solid torus whose core has cone angle $\alpha < 2\pi$.
Then 
\[
\sinh 2 r_{\alpha,\lambda,\tau}(\epsilon) 
\: \geq \:  \frac{\sqrt{3} \, \epsilon^2}{ \alpha \lambda}
\: > \: \frac{\sqrt{3} \, \epsilon^2}{ 2 \pi \lambda}.
\]
\end{lemma}

\begin{proof}
This follows from \reflem{EpsilonArea} and \refeqn{CylindricalCoords}. See also \cite[Lemma 7.2]{FPS:Tubes}.
\end{proof}

\subsection{Distances between tori of fixed injectivity radius}

\begin{definition}\label{Def:TubeDistance}
Let $0 < \delta \leq \epsilon$. For a model solid torus $N = N_{\alpha,\lambda,\tau}$ with $\lambda \leq \delta$, the distance between equidistant tori $T^\delta = T_{r(\delta)}$ and $T^\epsilon = T_{r(\epsilon)}$ is
\[
d_{\alpha,\lambda, \tau}(\delta,\epsilon)  = d(N^{\leq \delta}, N^{\geq \epsilon})  = r_{\alpha,\lambda, \tau}(\epsilon) - r_{\alpha,\lambda, \tau}(\delta).
\]
For a model horocusp $N$, we similarly define
\[
d_{N}(\delta,\epsilon)  = d(N^{\leq \delta}, N^{\geq \epsilon}).
\]
\end{definition}

If $N$ is a tube, the distance $d_{\alpha,\lambda, \tau}(\delta,\epsilon)$ depends on the parameters of the tube. Nevertheless, we have upper and lower bounds on $d_{\alpha,\lambda, \tau}(\delta,\epsilon)$ that hold independent of the parameters $\alpha,\lambda, \tau$.

\begin{theorem}\label{Thm:EffectiveDistLog3}
Suppose that $0 < \delta < \epsilon \leq \log 3$. Let $N = N_{\alpha, \lambda, \tau}$ be a model solid torus with cone angle $\alpha \leq 2\pi$ and core geodesic of length $\lambda \leq \delta$, or a model horocusp whose $\epsilon$--thick part is not empty.
Then
\[
\max \left\{ \frac{\epsilon - \delta}{2}, \: 
\arccosh \left( \frac{\epsilon}{\sqrt{7.256 \, \delta}} \right) - 0.1475 \right\}
\leq
d_{\alpha,\lambda, \tau}(\delta,\epsilon)
\leq
\arccosh \sqrt{ \frac{\cosh \epsilon - 1}{\cosh \delta - 1}  }.
\]
\end{theorem}

We remark that the argument of $\arccosh$ in the lower bound of \refthm{EffectiveDistLog3} may be less than $1$, making $\arccosh(\cdot)$ undefined. To remedy this, we employ the convention that an undefined value does not realize the maximum.  
Observe that the
lower bound $\frac{\epsilon - \delta}{2}$ follows by \reflem{LowerBoundEasy}.

\begin{proof}
If $N$ is a model solid torus, this is a special case of \cite[Theorem 8.8]{FPS:Tubes}. In the notation of that theorem, substituting $\epsilon_{\max} = \log(3)$ implies a value $j_{\max} = 0.14798\ldots$, which gives an additive constant $\arcsinh(j_{\max}) \leq 0.1475$ in the lower bound  on $d_{\alpha,\lambda, \tau} (\delta,\epsilon)$.

If $N$ is a horocusp, the estimate follows by taking a geometric limit of model solid tori converging to $N$.
\end{proof}

Under stronger hypotheses on $\epsilon$, we obtain a stronger lower bound on $d_{\alpha,\lambda, \tau} (\delta,\epsilon)$. 

\begin{theorem}\label{Thm:EffectiveDistTubes}
Suppose that $0 < \delta < \epsilon \leq 0.3$. Let $N = N_{\alpha, \lambda, \tau}$ be a model solid torus with cone angle $\alpha \leq 2\pi$ and core geodesic of length $\lambda \leq \delta$, or a model horocusp whose $\epsilon$--thick part is not empty.
Then
\[
\max \left\{ \frac{\epsilon - \delta}{2}, \: 
\arccosh \left( \frac{\epsilon}{\sqrt{7.256 \, \delta}} \right) - 0.0424 \right\}
\leq
d_{\alpha,\lambda, \tau}(\delta,\epsilon)
\leq
\arccosh \sqrt{ \frac{\cosh \epsilon - 1}{\cosh \delta - 1}  }.
\]
\end{theorem}

\begin{proof}
If $N$ is a model solid torus, this is \cite[Theorem 1.1]{FPS:Tubes}. If $N$ is a cusp, take a geometric limit of tubes.
\end{proof}

\subsection{Euclidean bounds}\label{Sec:Euclidean}

Consider an equidistant torus $T_r = \bdy U_r$. Then the Euclidean path-metric on $T_r$ lifts to a Euclidean metric on $\widetilde T_r$, which we denote $d_E$.

\begin{lemma}\label{Lem:EucInjectivityGeneral}
Let $\widetilde T_r \subset \Hhat$ be a plane at fixed distance $r > 0$ from the singular geodesic $\bcover{\sigma}$.
Let $p,q \in \widetilde T_r$ be points whose $\theta$--coordinates differ by at most $A \leq \pi$ and whose $\zeta$--coordinates differ by at most $B$. Then
\[
\frac{1 - \cos A}{A^2} \, d_E (p,q)^2 \leq \cosh d(p,q) - 1 \leq \frac{\cosh B - 1}{B^2} \, d_E (p,q)^2.
\]
\end{lemma}

\begin{proof}
See  \cite[Lemma 6.2]{FPS:Tubes}.
\end{proof}

If an equidistant plane in $\Hhat$ is replaced by a horosphere in $\HH^3$, 
\reflem{EucInjectivityGeneral} becomes the following (well-known) statement.

\begin{lemma}\label{Lem:EucInjectivityCusp}
Let $\widetilde T \subset \HH^3$ be a horosphere, and let $p,q \in \widetilde T$. Let $d_E(p,q)$ be the distance between $p$ and $q$ in the Euclidean metric on $\widetilde T$. Then 
\[
2 \sinh \frac{d(p,q)}{2} = d_E(p,q),
\qquad \text{or equivalently} \qquad
\cosh d(p,q) -1 = \frac{1}{2} d_E(p,q)^2.
\]
\end{lemma}

\begin{proof}
See  \cite[Lemma A.2]{cfp:tunnels}.
\end{proof}

We observe that as $A, B \to 0$, the upper and lower bounds in  \reflem{EucInjectivityGeneral} both approach $\frac{1}{2} d_E(p,q)^2$. Thus \reflem{EucInjectivityCusp} realizes this limiting value.

\section{Maximal tubes and multi-tubes}\label{Sec:MultiTube}

The goal of this section is to control the area and injectivity radius of maximal tubes in cone-manifolds. The main results are \refthm{MaxTubeArea}, giving a lower bound on area, and \refthm{MaxTubeInjectivity}, giving a lower bound on injectivity radius. \refthm{MaxTubeArea} is essentially due to Hodgson and Kerckhoff \cite[Theorem 4.4]{hk:univ}, while \refthm{MaxTubeInjectivity} is new. 
Before getting to those results, we must carefully construct maximal tubes of many components.

\begin{definition}\label{Def:MultiTube}
Let $(M,\Sigma)$ be a hyperbolic cone-manifold. Let $\Sigma^+$ be a geodesic link in $M$, such that $\Sigma \subset \Sigma^+$. 
Let $\sigma_1, \ldots, \sigma_n$ be the components of $\Sigma^+$. For a positive vector $\rr = (r_1, \ldots, r_n)$, let $U_i = U_{r_i}(\sigma_i)$ be the set of all points whose distance to $\sigma_i$ is less than $r_i$. Let
\[
U_\rr = \cup_{i=1}^n U_{r_i}(\sigma_i).
\]

We say that $U_\rr$ is a \emph{multi-tube} if every $U_i$ is isometric to a model tube as in \refdef{ModelTube}, and the $U_i$ are pairwise disjoint.
\end{definition}

We choose a particular construction of maximal tubes.

\begin{definition}\label{Def:MaximalTube}
Let $M$ be a non-elementary hyperbolic cone-manifold, and $\Sigma^+$ a geodesic link containing the singular locus. We construct a maximal multi-tube about $\Sigma^+$, as follows: 
\begin{enumerate}[1.]
\item For a sufficiently small $r>0$, choosing a constant vector  $\rr = (r, \ldots, r)$ produces an embedded multi-tube $U_\rr$. Let $R_1$ be the largest value of $r$ for which this holds. The hypothesis that $M$ is non-elementary ensures that such an $R_1$ exists.

Setting $r = R_1$, we have an (open) multi-tube $U_\rr$ whose closure $\overline{U}_\rr$ is not a disjointly embedded union of closed tubes. In other words, either some tube has bumped into itself, or some number of tubes have bumped into one another. Any tube $U_i$ that cannot be expanded further without intersecting itself or another tube is declared \emph{maximal}, and its radius will remain fixed for the rest of the construction.

\item\label{Itm:bump} Suppose, after relabeling, that $U_1, \ldots, U_k$ are maximal, but $U_{k+1}, \ldots, U_n$ are not. We expand the radii $r_{k+1}, \ldots, r_n$ at a uniform rate until some tube $U_i$ for $i \geq k+1$ bumps into some other tube $U_j$ (it may happen that $i = j$ or $j \leq k$). We then declare the tubes that have just bumped to be \emph{maximal}, and freeze their radii.

\item Repeat step \ref{Itm:bump} as needed until no tube can be expanded further. We call the resulting union of maximal tubes the \emph{maximal multi-tube}, and denote it $U_{\max}(\Sigma^+)$ or $U_{\max}$ for short. We order the components of $\Sigma^+$ so that the vector of radii $\rr = (R_1, \ldots, R_n)$ appears in non-decreasing order, with $R_1$ the smallest radius.
\end{enumerate}
\end{definition}

\begin{proposition}\label{Prop:Ellipse}
Let $M$ be a non-elementary hyperbolic cone-manifold, and $\Sigma^+$ a geodesic link containing the singular locus $\Sigma$. Let $U_{\max}(\Sigma^+)$ be the maximal multi-tube about $\Sigma^+$, with smallest tube radius $R = R_1$. Suppose that a component $U_i  \subset U_{\max}$ that becomes maximal by bumping into a component $U_j$. Then
the radii of these tubes satisfy $R_1 \leq R_j \leq R_i$ and
$T_i = \bdy U_i$ contains an embedded ellipse whose semi-major axes are
\begin{equation}\label{Eqn:Ellipse}
a(R_i, R_j) = \frac {\cosh R_i \sinh R_j }{S(R_j)  \cosh (R_i + R_j)}
\quad \mbox{and} \quad
b(R_i, R_j) = \frac{\sinh R_i \sinh R_j } { \sinh (R_i + R_j)}.
\end{equation}
Here
$S(R_j)$ is defined via
\begin{equation}\label{Eqn:SDefine}
S(r) = 
\def\arraystretch{2.2}
\begin{dcases}
\dfrac{\sqrt{2}/4}{\arcsinh (\sqrt{2}/4 ) } = 1.02013\ldots, & \mbox{ if } \sinh r \leq \dfrac{1}{\sqrt{2}} , \\
\dfrac{ \sinh r / \cosh(2 r) }{\arcsinh \left( \sinh r / \cosh(2 r)  \right) } , &  \mbox{ if }\sinh r \geq \dfrac{1}{\sqrt{2}}.
\end{dcases}
\end{equation}

Furthermore, if the tube $U_i$ became maximal by bumping into itself,
$T_i = \bdy U_{i}$ contains two disjoint ellipses as in \refeqn{Ellipse}, with parallel major axes. 
\end{proposition}

\begin{proof}
This result is essentially due to Hodgson and Kerckhoff \cite[Section 4]{hk:univ}. They prove an identical statement when $\Sigma^+$ is connected, and briefly mention that the argument extends to multiple tubes. We make small modifications to their argument in order to handle maximal multi-tubes of disparate radii. It is worth noting that our construction of maximal multi-tubes differs somewhat from that of Hodgson and Kerckhoff.

As in \refdef{UniversalBranchedCover}, let $\bcover{M}$ be the universal branched cover of $M$, branched over $\Sigma^+$. Every component of $\Sigma^+$ lifts to a singular geodesic in $\bcover{M}$, with cone angle $\infty$, and the space is locally modeled on $\Hhat$, as in \refdef{SingularModel}. For any singular basepoint in $\bcover{M}$, we have the exponential-like map
$D\from \bcover{M} \to \Hhat $.

Consider the tube $U_i \subset U_{\max}$. Since $U_i$ is maximal, the expansion of $U_i$ came to a halt because $U_i$ bumped into some other tube $U_j$, which became maximal no later than $U_i$ did. 
By the construction of \refdef{MaximalTube}, we have 
\begin{equation}\label{Eqn:RadOrdering}
 R_i \geq R_j \geq R_1 = R.
 \end{equation}

Since $\bdy U_i$ is tangent to $\bdy U_j$,
 there is a geodesic arc $\gamma$ of length $R_i + R_j$ that travels radially outward from the core circle $\sigma_i$, enters into $U_j$ at a point of tangency, and meets the core circle $\sigma_j$ perpendicularly at its endpoint. This arc $\gamma$ lifts to an arc $\widetilde{\gamma} \subset \bcover{M}$ from a lift $\widetilde{\sigma_i}$ of $\sigma_i$ to a lift $\widetilde{\sigma_j}$ of $\sigma_j$. By construction, $\widetilde{\gamma}$ is a shortest geodesic from $\widetilde{\sigma_i}$ to any other singular geodesic in $\bcover{M}$. In other words, we have: 

\begin{claim}\label{Claim:Develop}
The geodesic $\widetilde{\sigma_i} \subset \bcover{M}$ has an embedded neighborhood $\widetilde{V_i}$ of radius $R_i + R_j$. If we choose a basepoint on $\widetilde{\sigma_i}$, the map $D: \widetilde{V_i} \to \Hhat$ is an isometric embedding. \qed
\end{claim}

Let $\Gamma$ be the set of all lifts of $\gamma$ starting at $\widetilde{\sigma_i}$, oriented outward from $\widetilde{\sigma_i}$. Let $Q$ be the set of their forward endpoints. As in  \refdef{ConeManifold}, there is a 
$\ZZ \times \ZZ$ group of deck transformations of $\bcover{M}$ acting effectively and transitively on $\Gamma$, hence on $Q$. 

\begin{claim}\label{Claim:FarEndpoints}
Let $q, q' \in Q$ be endpoints of distinct lifts of $\gamma$. Then 
\[ d (D(q), D(q')) = d(q, q') \geq 2 R_j . \]
\end{claim}

The equality of distances holds by \refclaim{Develop}. The inequality holds because $q$ and $q'$ lie on distinct lifts of $\sigma_j$, and $\sigma_j$ has an embedded tube of radius $R_j$. See \cite[Lemma 4.1]{hk:univ} for more details.

For each $q \in Q$, let $B(q) \subset \Hhat$ be a ball of radius $R_j$ centered at $D(q) \in \Hhat$. By \refclaim{FarEndpoints}, these balls are disjointly embedded in $\Hhat$. In other words, each $B(q)$ is disjoint from its translates under $\ZZ \times \ZZ$.

Let $(r,\theta,\zeta)$ be cylindrical coordinates on $\Hhat$, as in \refdef{SingularModel}. We normalize things so that the geodesic arc $D(\widetilde{\gamma})$ lies on the geodesic ray $\{(r,0,0):r>0\}$. Let $B(q)$ be the ball of radius $R_j$ centered at $(R_i + R_j, 0, 0)$. By \cite[Lemma 4.3]{hk:univ}, the projection of $B(q)$ to the $(\theta, \zeta)$-plane consists of all points $(\theta, \zeta)$ that satisfy 
\begin{equation}\label{Eqn:BallProjection}
 \sinh^2 \zeta \cosh^2 (R_i + R_j) + \sin^2 \theta \sinh^2 (R_i + R_j) \leq \sinh^2 R_j.
\end{equation}

\begin{claim}\label{Claim:Sbound}
Let $(\theta,\zeta)$ be a point in the projection of $B(q)$ to the $(\theta,\zeta)$-plane. 
Setting $S(R_j)$ is as in \refeqn{SDefine}, we have \: $\abs{ \sinh \zeta } \leq S(R_j) \abs{\zeta}$. Furthermore, $S(R_j)$ is a decreasing function of $R_j$.
\end{claim}

To see this, observe that   \refeqn{BallProjection}, combined with  \refeqn{RadOrdering}, implies
\begin{equation}\label{Eqn:ZetaCoordBound}
| \sinh \zeta | \leq \frac{\sinh R_j}{\cosh(R_i + R_j)} \leq \frac{\sinh R_j}{\cosh(2 R_j)}.
\end{equation}
Now, setting $x = \sinh r$, observe that the function
\[
\frac{\sinh r}{\cosh(2 r)} = \frac{x}{2x^2 + 1}
\]
reaches a global maximum value of $\sqrt{2}/4$ when $x = 1/\sqrt{2}$, and declines to $0$ thereafter. Since $\abs{\sinh \zeta / \zeta}$ is increasing in $\abs{\zeta}$, we have
\[
\abs{ \frac{\sinh \zeta}{\zeta} } \leq \frac{ \sinh \left( \arcsinh (\sqrt{2}/4) \right)}{  \arcsinh (\sqrt{2}/4) } = S(0) = 1.02013\ldots \quad \text{for all values of $R_j$},
\]
as observed by Hodgson--Kerckhoff \cite[Pages 401--402]{hk:univ}. Furthermore, when $x = \sinh R_j \geq 1/\sqrt{2}$, we have
\[
\abs{ \frac{\sinh \zeta}{\zeta} } \leq \dfrac{ \sinh R_j / \cosh(2 R_j) }{\arcsinh \left( \sinh R_j / \cosh(2 R_j)  \right) } = S(R_j),
\]
which is increasing in $\sinh R_j / \cosh(2 R_j)$, hence decreasing  in $R_j$.
This proves the claim.

Combining \refeqn{BallProjection}, \refclaim{Sbound}, and the standard fact $\abs{\sin \theta} \leq | \theta |$ gives:

\begin{claim}\label{Claim:BallProjEllipse}
Let $B(q)$ be the ball of radius $R_j$ centered at $(R_i + R_j, 0, 0)$. Then the projection of $B(q)$ to the $(\theta, \zeta)$-plane contains the elliptical region consisting of points $(\theta,\zeta)$ that satisfy
\begin{equation}\label{Eqn:BallProjEllipse}
 \zeta^2 \, S(R_j)^2 \cosh^2 (R_i + R_j) +  \theta^2 \sinh^2 (R_i + R_j) \leq \sinh^2 R_j.
\end{equation}
\end{claim}

Now, recall the tube $U_i$ about $\sigma_i$, with radius $R_i$. This tube lifts to a tube $\widetilde U_i$ about $\widetilde \sigma_i $, which isometrically embeds in $\Hhat$ by \refclaim{Develop}. We consider the shadow of $B(q)$ on the Euclidean plane $D(\bdy \widetilde U_i)$. Since $(\zeta \cosh R_i)$ and $(\theta \sinh R_i)$ are Euclidean coordinates on the plane at radius $R_i$ from the singular geodesic of $\Hhat$, the elliptical region of \refeqn{BallProjEllipse} can be rewritten in coordinates as
\begin{equation*}
\left( \frac{ S(R_j)  \cosh (R_i + R_j)}{\cosh R_i \sinh R_j } \right)^{\! 2} \! (\zeta \cosh R_i)^2 +   \left( \frac{ \sinh (R_i + R_j)}{\sinh R_i \sinh R_j }\right)^{\! 2} \! (\theta \sinh R_i )^2 \leq 1.
\end{equation*}
Since the elliptical region is disjoint from its translates under $\ZZ \times \ZZ$, it follows that the quotient torus $T_i = \bdy U_i$ contains an embedded ellipse whose semi-major axes are
\begin{equation*}
a(R_i, R_j) = \frac {\cosh R_i \sinh R_j }{S(R_j)  \cosh (R_i + R_j)}
\quad \mbox{and} \quad
b(R_i, R_j) = \frac{\sinh R_i \sinh R_j } { \sinh (R_i + R_j)},
\end{equation*}
as required in \refeqn{Ellipse}.

If the tube $U_i$ became maximal by bumping into itself, the arc $\gamma$ must have both of its endpoints on $\sigma_i$. This means there are two distinct $\ZZ \times \ZZ$ orbits of lifts of $\gamma$ with an endpoint on $\widetilde{\sigma_i}$, giving rise to two orbits of balls $B(q)$ and two disjoint ellipses on $\bdy U_i$. This is the case that Hodgson and Kerckhoff analyze in \cite[Theorem 4.4]{hk:univ}. 
\end{proof}

\begin{remark}\label{Rem:ThetaRange}
It follows from Equations \refeqn{BallProjEllipse} and \refeqn{RadOrdering} that the ellipse constructed in the last proof has $\theta$--coordinate bounded as follows:
\[
|\theta| \leq \frac{\sinh R_j}{\sinh(R_i +R_j)} \leq \frac{\sinh R_j}{\sinh(2 R_j)} = \frac{1}{2 \cosh R_j} < \frac{1}{2}.
\]
\end{remark}

\subsection{Areas of maximal tubes}
We present two applications of \refprop{Ellipse} that will be crucial in the sequel. The first application, developed by Hodgson and Kerckhoff \cite{hk:univ, hk:shape}, concerns the area of maximal multi-tubes. We will use the area of the ellipse constructed in \refprop{Ellipse} to get a lower bound on $\area(T_i)$. To do this, we need to remove the dependence on the tube $T_j$.

\begin{lemma}\label{Lem:EllipseAreaMonotonic}
Let $a(R_i, R_j)$ and $b(R_i, R_j)$ be as in Equation~\refeqn{Ellipse}. Then the function that is their product 
\[
 ab(R_i, R_j) =  \frac {\sinh R_i \cosh R_i \sinh^2 R_j }{ S(R_j) \sinh (R_i + R_j) \cosh (R_i + R_j)} 
\]
is increasing in both variables.
\end{lemma}

\begin{proof}
Since $S(R_j)^{-1}$ is increasing in $R_j$ by  \refclaim{Sbound}, it suffices to show that $S(R_j)  ab(R_i, R_j)$ is increasing.
To that end, we substitute the variable names $x = R_i$ and $y = R_j$, and simplify:
\[  
Sab(x,y) = \frac {(\sinh x \cosh x ) \sinh^2 y }{  \sinh (x + y) \cosh (x + y)} 
= \frac {\sinh (2 x) \cdot \frac{1}{2} (\cosh (2y) -1) }{  \sinh (2 x + 2 y) } \, .
\]
Now, we can compute the partial derivatives:
\begin{align*}
\frac{\partial Sab}{\partial x} 
&= \frac {2 \sinh (2 x + 2 y) \cosh (2 x) - 2 \sinh (2 x) \cosh (2x+2y)   }{  \sinh^2 (2 x + 2 y) } \cdot \frac{1}{2} (\cosh (2y) -1) \\
&= \frac { \sinh (2 y)   }{  \sinh^2 (2 x + 2 y) } \cdot  (\cosh (2y) -1) > 0
\end{align*}
when $y > 0$. Similarly,
\begin{align*}
\frac{\partial Sab}{\partial y} 
&= \frac { \sinh (2 x + 2 y) \sinh (2 y) - (\cosh (2y) -1)  \cosh (2x+2y)   }{  \sinh^2 (2 x + 2 y) } \cdot \sinh(2x) \\
&= \frac { [\sinh (2 x + 2 y) \sinh (2 y) - \cosh (2y)   \cosh (2x+2y) ]+  \cosh (2x+2y) }{  \sinh^2 (2 x + 2 y) } \cdot \sinh(2x) \\
&= \frac {  \cosh (2x+2y) - \cosh (2 y)   }{  \sinh^2 (2 x + 2 y) } \cdot   \sinh(2x) > 0.
\end{align*}
when $x >0$ and $y >0$.
\end{proof}

We can now show the following. 

\begin{theorem}\label{Thm:MaxTubeArea}
Let $M$ be a non-elementary hyperbolic cone-manifold, and $\Sigma^+$ a geodesic link containing the singular locus $\Sigma$. Let $U_{\max}(\Sigma^+)$ be the maximal multi-tube about $\Sigma^+$, with smallest tube radius $R$. Let $T_i = \bdy U_i$ be the boundary torus of any component of $U_{\max}$. Then
\begin{equation}\label{Eqn:MaxTubeArea}
\area(T_i) \geq \frac{\sqrt{3} \sinh^2 R}{S(R) \cosh (2 R) } \geq 1.69785 \frac{ \sinh^2 R}{\cosh (2 R) }.
\end{equation}
If the tube $U_i$ became maximal by bumping into itself (for instance, if $\Sigma^+$ is connected), then $\area(T_i)$ is bounded below by \emph{twice} the estimate of \refeqn{MaxTubeArea}.
\end{theorem}

\begin{proof}
This result is essentially \cite[Theorem 4.4]{hk:univ}. To derive the theorem from \refprop{Ellipse},  perform an area-preserving affine transformation on $T_i$ that turns the ellipse of \refprop{Ellipse} into a circle. The area of this circle is
\[
\pi ab(R_i, R_j) \geq \pi ab(R, R) =  \frac {\pi \sinh R \cosh R \sinh^2 R}{ S(R) \sinh (2 R) \cosh (2 R)} =  \frac {\pi  \sinh^2 R}{2 S(R)  \cosh (2 R)},
\]
where the inequality is \reflem{EllipseAreaMonotonic}.
By a theorem of B\"or\"oczky \cite{boroczky}, the maximal density of a circle packing in the torus is $\frac{\pi}{2 \sqrt{3}}$. Therefore,
\[
\area(T_i) \geq \frac{2 \sqrt{3}}{\pi} \pi a b (R,R)  = 2 \sqrt{3} \, a b(R,R) =  \frac{\sqrt{3} \sinh^2 R}{S(R) \cosh (2 R) }.
\]
Recall from \refclaim{Sbound} that $S(R) \leq S(0) = 1.02013\ldots$, hence $\sqrt{3} / S(R) \geq 1.69785$.

If the tube $U_1$ became maximal by bumping into itself, the two ellipses guaranteed by \refprop{Ellipse} become two circles of identical radius. Thus the estimate of \refeqn{MaxTubeArea} becomes doubled. This is the case that Hodgson and Kerckhoff analyze in \cite[Theorem 4.4]{hk:univ}. 
\end{proof}

We apply \refthm{MaxTubeArea} to bound the visual area of $\Sigma^+$.

\begin{definition}\label{Def:VisualArea}
Let $M$ be a non-elementary hyperbolic cone-manifold, and let $\Sigma^+ = \sigma_1 \cup \ldots \cup \sigma_n$ be a geodesic link containing the singular locus. Let $\alpha_j$ be the cone angle along $\sigma_j$, and let $\lambda_j = \len(\sigma_j)$ be the length of $\sigma_j$. We define the \emph{visual area of $\sigma_j$} to be
\[
\calA_j = \alpha_j \lambda_j.
\]
The \emph{visual area of $\Sigma^+$} is defined by summation: $\calA = \sum_j \calA_j$. Note that if $T_j$ is the boundary of some tube $U_j \subset U_{\max}$, then \refeqn{CylindricalCoords} implies
\begin{equation}\label{Eqn:TorusArea}
\area(T_j) = \calA_j \sinh R_j \cosh R_j = \calA_j \sinh(2 R_j) / 2.
\end{equation}
\end{definition}

\begin{definition}\label{Def:Hfunction}
Define a function
\[
h(r) = 3.3957 \, \frac{\tanh(r)}{\cosh(2r)}
= 3.3957 \, \frac{z(1-z^2)}{1+z^2},
\]
where $z = \tanh r$.
\end{definition}

\begin{theorem}\label{Thm:VisualAreaControl}
Let $M$ be a non-elementary hyperbolic cone-manifold, and $\Sigma^+$ a geodesic link containing the singular locus. Let $U_{\max}(\Sigma^+)$ be the maximal multi-tube about $\Sigma^+$, with smallest tube radius $R$. 
Then
\begin{gather*}
  \calA \geq h(R).
\end{gather*}
\end{theorem}

\begin{proof}
This result is due to Hodgson and Kerckhoff \cite[Theorem 5.6]{hk:shape}. 
We repeat the short proof for completeness. If $T_i$ is a boundary torus of some component of $U_{\max}$, equation~\refeqn{TorusArea} gives
\[
\area(T_i) 
= \calA_i \sinh R_i \cosh R_i.
\]

If $U_1$ became maximal by bumping into itself, \refthm{MaxTubeArea} guarantees
\[
\calA \geq \calA_1  = \frac{\area(T_1)}{\sinh R \cosh R} \geq  3.3957 \, \frac{\sinh^2 R / \cosh(2R)}{\sinh R \cosh R} = h(R),
\]
as desired. Meanwhile, if $U_1$ became maximal by bumping into another tube $U_2$, then $R_1 = R_2 = R$, and \refthm{MaxTubeArea} bounds the area of each of $T_1 = \bdy U_1$ and $T_2 = \bdy U_2$. Therefore,
\[
\calA \geq \calA_1 + \calA_2  = \frac{\area(T_1)+ \area(T_2)}{\sinh R \cosh R} \geq  3.3957 \, \frac{\sinh^2 R / \cosh(2R)}{\sinh R \cosh R} = h(R). \qedhere
\] 
\end{proof}

In \refsec{ConeDef}, we will apply \refthm{VisualAreaControl} to prove the existence of cone deformations maintaining a given tube radius about $\Sigma$. To set up this application, we need to establish some important properties of $h(r)$.

\begin{lemma}[Lemma~5.2 of \cite{hk:univ}]\label{Lem:HBound}
For $r > 0$, the function $h(r)$ of \refdef{Hfunction} has a unique critical point at $r_0=\arctanh\sqrt{\sqrt{5}-2} \approx 0.5306375$. This critical point is a global maximum, hence $h(r)$ is strictly decreasing when $r \geq 0.531$. \qed
\end{lemma}

\begin{remark}\label{Rem:Haze}
Recall from \refdef{Hfunction} that $h(r)$ can be expressed as a rational function of $z = \tanh r$. This leads us to define 
\[ \haze(z) = h(\tanh^{-1}(z)) = 3.3957 \, \frac{z(1-z^2)}{1+z^2}. \]
By \reflem{HBound}, $\haze(z)$ is decreasing and invertible in the range
$z \in \big[ \sqrt{\sqrt{5}-2}, \, 1 \big)$.
The functions $\haze$ and $\haze^{-1}$ will play an important role in \refsec{BoundaryBound} and \refsec{ShortGeodesic}. Because inverting $\haze(z)$ amounts to solving a cubic equation, Cardano's Formula can be used to obtain a closed-form expression for $\haze^{-1}(h)$: 
\begin{equation}\label{Eqn:HazeInv}
\haze^{-1}( 3.3957 x) =  \frac{2 \sqrt{ x^2 + 3 }}{3}   \cos \left(   \frac{\pi}{3}  + \frac{1}{3} \tan^{-1} \! \left(\frac{-3  \sqrt{ -3x^4 - 33x^2 + 3 } }{ x^3 + 18x} \right)  \right) - \frac{x}{3}.
\end{equation}
\end{remark}

Returning to the function $h(r)$, we define $\hmax = h(0.531) \approx 1.01967$. By \reflem{HBound}, this is slightly less than the true maximal value of $h$. Now, \refthm{VisualAreaControl} and \reflem{HBound} have the following immediate corollary.

\begin{corollary}\label{Cor:HInverse}
The function $h$ of \refdef{Hfunction} has a well-defined inverse
\[
h^{-1} \from (0, \hmax] \to [0.531, \infty),
\]
which can be computed via \refeqn{HazeInv}. Furthermore, $h^{-1}$ is a decreasing function such that the maximal tube radius satisfies
\[
R \geq h^{-1}(\calA),
\]
provided $R \geq 0.531$. \qed
\end{corollary}

\subsection{Injectivity radii}

Recall the definition of $\injrad(x, U_i)$ from \refdef{InjRadTube}. 

\begin{theorem}\label{Thm:MaxTubeInjectivity}
Let $M$ be a non-elementary hyperbolic cone-manifold, and $\Sigma^+$ a geodesic link containing the singular locus. Let $U_{\max}(\Sigma^+)$ be the maximal multi-tube about $\Sigma^+$, with smallest tube radius $R$. 

Then, for every tube $U_i \subset U_{\max}$ and every $x \in \bdy U_i$,
\begin{align}
2 \, \injrad(x, U_i) 
&\geq  1.361 \sqrt{1-\cos \left( \frac{1}{\cosh(R)} \right) }  \cdot \frac{\sinh R}{S(R)} \label{Eqn:MaxTubeInjecGeneral} \\
& >   1.1227 \tanh R  - 0.1604. \label{Eqn:MaxTubeInjecLinear}
\end{align}
where $S(R)$ is defined in \refeqn{SDefine}. The functions on the right-hand side of  \refeqn{MaxTubeInjecGeneral} and \refeqn{MaxTubeInjecLinear} are increasing in $R$.
\end{theorem}

\begin{proof}
First, we check that the function expressing the lower bound in \refeqn{MaxTubeInjecGeneral} is increasing in $R$, and calculate its limit as $R \to \infty$. We define
\[
f_1(R) =  \sqrt{1-\cos \left( \frac{1}{\cosh(R)} \right) } \cdot \cosh(R),
\qquad 
f_2(R) = \frac{1}{ S(R)},
\qquad
f_3(R) = \tanh R,
\]
so that the lower-bound function in \refeqn{MaxTubeInjecGeneral} becomes $f(R) = 1.361 \, f_1(R) \, f_2(R) \, f_3(R)$.
The first non-constant term in the product can be written as
\[
f_1(R) = \sqrt{1-\cos \left( \frac{1}{\cosh(R)} \right) } \cdot \cosh(R) = \sqrt{ \frac{ 1- \cos (A) }{A^2}  } \qquad \text{for} \qquad A = \frac{1}{\cosh (R) }.
\]
For $A$ between $0$ and $1$, the function $f_1(R)$ is decreasing in $A$, hence increasing in $R$. It satisfies $\displaystyle{\lim_{R \to \infty} f_1(R) = \tfrac{1}{\sqrt{2}}}$.
The second term is $f_2(R) = S(R)^{-1}$, which is increasing by \refclaim{Sbound} and approaches $1$ as $R \to \infty$. Finally, the third term is $f_3(R) = \tanh R $, which is also increasing in $R$ and approaches $1$ as $R \to \infty$. Thus $f(R) = 1.361 \, f_1(R) \, f_2(R) \, f_3(R)$ increases at least as fast as $Z = \tanh R$, and satisfies
\begin{equation}\label{Eqn:InjecLimit}
\lim_{R \to \infty} f(R) =  1.361 \,  \lim_{R \to \infty} f_1(R) \, f_2(R) \,  f_3(R)  = \frac{1.361 }{ \sqrt{2} } = 0.96237 \ldots
\end{equation}

Next, we check that the function in \refeqn{MaxTubeInjecGeneral} is larger than the one \refeqn{MaxTubeInjecLinear}. 
Set $Z = \tanh R$, as above. 
When $Z \in [0.99995, \, 1)$, the increasing function in \refeqn{MaxTubeInjecGeneral} is bounded below by $0.9623$, whereas  $1.1227 Z  - 0.1604$ is bounded above by $0.9623$. Meanwhile, when $Z \in [0, 0.99995]$, we check using interval arithmetic in Sage that  the function in \refeqn{MaxTubeInjecGeneral} is larger than $1.1227 Z  - 0.1604$. This is established by breaking the domain $[0, 0.99995]$ into small intervals and checking the desired inequality on each sub-interval. See the ancillary files \cite{FPS:Ancillary} for details.

\smallskip
 
Now, we proceed to the main portion of the proof: the lower bound on $\injrad(x, U_i)$ expressed in \refeqn{MaxTubeInjecGeneral}. Consider the torus $T_i = \bdy U_i$. \refprop{Ellipse} has the following consequence.

\begin{claim}\label{Claim:Circle}
The torus $T_i = \bdy U_i$ contains an embedded open disk of radius
\[
\frac{b(R,R)}{S(R)} = \frac{\sinh^2 R }{S(R) \sinh(2R)} = \frac{ \tanh R}{2 S(R)} \, ,
\]
where $b(\cdot, \cdot)$ is a semi-major axis as in \refeqn{Ellipse}, and $S(\cdot)$ is as in \refeqn{SDefine}.
\end{claim}

This can be seen as follows. 
By \refprop{Ellipse}, torus $T_i$ contains an embedded ellipse, whose semi-major axes are
\[
a(R_i, R_j) = \frac{\cosh R_i \sinh R_j}{S(R_j) \cosh (R_i + R_j)}
\quad \mbox{and} \quad
b(R_i, R_j) = \frac{\sinh R_i \sinh R_j}{\sinh(R_i + R_j)}.
\]
This ellipse contains a disk of radius $\min\{a,b\}$. We would like to determine this minimum. By \reflem{SinhCoshGrowth},
\[
\frac{\cosh R_i}{\cosh (R_i + R_j)} > \frac{\sinh R_i}{\sinh(R_i + R_j) },
\quad \mbox{hence} \quad 
S(R_j) a(R_i, R_j) > b(R_i, R_j).
\]
Since $S(R_j) > 1$, it follows that
\[
\min \left\{ a(R_i, R_j) , \, b(R_i, R_j) \right\} \geq \min \left\{ a(R_i, R_j), \, \frac{b(R_i, R_j)}{S(R_j)} \right\} = \frac{b(R_i, R_j)}{S(R_j)} 
 \geq \frac{b(R,R) }{S(R) }.
\]
Here, the last inequality follows because $b(R_i, R_j)$ is monotonically increasing in both variables, by
a calculation similar to 
\reflem{EllipseAreaMonotonic}. Meanwhile, $S(R)$ is monotonically decreasing by \refclaim{Sbound}, hence the quotient is increasing. This proves the claim.

Proceeding toward the main proof, let
$\bcover{M}$ be the universal branched cover of $M$, branched over $\Sigma^+$.
Choose a preimage $\widetilde U_i$ of $U_i$. Then $\widetilde T_i = \bdy \widetilde U_i$  is a Euclidean plane that covers $T_i$. Our goal, following \refdef{InjRadTube}, is to give a lower bound on the distance between a lift $\widetilde{x}$ of $x$ and any of its translates under $\pi_1(T_i) = \ZZ \times \ZZ$.

By \refclaim{Circle}, $\widetilde{T_i}$ contains a  $\ZZ \times \ZZ$--equivariant family of disjoint disks, of radius $b(R,R)/S(R)$.
Fix $p = \widetilde{x}$, and let $q = \varphi(\widetilde{x})$ be the closest translate of $p$. Since $\injrad(x,U_i)$ is constant over points of $T_i$, we may assume that $p$ and $q$ lie at centers of disks in this family. Thus $d_E(p,q) \geq 2b(R,R)/S(R)$, where $d_E$ denotes the Euclidean distance on $\widetilde{T_i}$, as in \refsec{Euclidean}.

\begin{claim}\label{Claim:pqDist}
We have $d(p,q) \geq f(R) = \displaystyle{1.361 \sqrt{1-\cos \left( \frac{1}{\cosh(R)} \right) } \cdot  \frac{\sinh R}{S(R)} }$.
\end{claim}

Before proving this claim, we make some quick reductions. First, as we computed in \refeqn{InjecLimit}, the function $f(R)$ is bounded above by $0.9625$. Thus it suffices to assume $d(p,q) \leq 0.9625$.
Second, it suffices to assume that the disks of radius $b/S$ centered at $p$ and $q$ are tangent, because any lower bound on distance for tangent disks will still apply as $p,q$ are moved further apart.

Our lower bound on $d(p,q)$ will come from \reflem{EucInjectivityGeneral}. In preparation for applying that lemma, we note that since $d(p,q) \leq 0.9625$, we have
\begin{equation}\label{Eqn:CoshTaylor}
\frac{ \cosh d(p,q) - 1}{d(p,q)^2} \leq \frac{ \cosh 0.9625 -1}{0.9625^2} = 0.53981 \ldots .
\end{equation}
By \refrem{ThetaRange}, the $\theta$--coordinates of an ellipse centered at $(r,0,0)$ must satisfy
\begin{equation}\label{Eqn:AcoshR}
| \theta | \leq \frac{1}{2 \cosh ( R_j)} \leq \frac{1}{2 \cosh(R)} = \frac{A}{2},
\end{equation}
where recall that we defined $A = 1/\cosh(R)$.
This means that the $\theta$--coordinates of $p$ and $q$, whose disks are assumed to be tangent, must differ by at most $A$, which is at most $1<\pi$. 
Finally, if the disks centered at $p,q$ are tangent, \refclaim{Circle} implies
\begin{equation}\label{Eqn:EucTangentCircles}
d_E(p,q) = \frac{ 2 b(R,R)}{S(R)} = \frac{ \tanh R}{ S(R)}.
\end{equation}

Now, we may plug \refeqn{CoshTaylor} and \refeqn{EucTangentCircles} into the lower bound of \reflem{EucInjectivityGeneral}. Using the upper bound from \refeqn{CoshTaylor}, we obtain
\[
0.53982 \, d(p,q)^2 \geq \cosh d(p,q) - 1 \geq \frac{1 - \cos A}{A^2} d_E (p,q)^2 
= \frac{1 - \cos A}{A^2}
 \cdot \frac{ \tanh^2 R}{S(R)^2} .
\]
Using the value $A = 1/\cosh R$ from \refeqn{AcoshR}, this simplifies to
\begin{align*}
d(p,q) 
&\geq \sqrt{ \frac{1}{0.53982 } } \cdot \frac{ \sqrt{1 - \cos A}}{A}  \cdot \frac{ \tanh R}{S(R)} \\
&= 1.36105 \ldots \sqrt{1-\cos \left( \frac{1}{\cosh(R)} \right) } \cdot \cosh R \cdot \frac{\tanh R}{S(R)} ,
\end{align*}
proving the claim. Since $q$ was assumed to be the closest translate of $p = \widetilde{x}$,  \refclaim{pqDist} 
proves the theorem.
\end{proof}

\begin{remark}\label{Rem:CuspTubeInjectivity}
If the cone manifold $M$ has cusps, the constructions and results of this section also apply to a maximal neighborhood consisting of tubes and horocusps in $M$. To extend \refdef{MaximalTube}, first construct a maximal multi-tube as in that definition. Then, choose any ordering on the cusps and expand each cusp neighborhood until it bumps into a tube or a previously expanded cusp.

After such a construction, Theorems~\ref{Thm:MaxTubeArea} and \ref{Thm:MaxTubeInjectivity} hold for the boundary tori of both tubes and horocusps. One way to see this is to view horocusps as limiting cases of tubes with radius $R_i \to \infty$. A key point in the proofs of both  \refthm{MaxTubeArea} and \refthm{MaxTubeInjectivity} is that the relevant estimates are monotonically increasing in $R_i$. Thus they will also hold if $R_i$ is replaced by $\infty$. 

If there are no compact tubes at all, but only a union of maximal cusps, both theorems become well-known statements from the literature.
\refthm{MaxTubeArea} becomes the well-known estimate due to Meyerhoff \cite[Section 5]{meyerhoff} that every cusp torus $T_i = \bdy U_i$ satisfies
\[
\area(T_i) \geq \frac{\sqrt{3}}{2} = \lim_{R \to \infty} \frac{\sqrt{3} \sinh^2 R}{S(R) \cosh (2 R) }.
\]
Meanwhile, \refthm{MaxTubeInjectivity} becomes a well-known estimate observed by Adams \cite[Lemma 2.4]{Adams:waist}: every non-trivial element $\varphi \in \pi_1(T_i)$ corresponds to a horocycle of length $\geq 1$. In other words, for every $\widetilde x \in \widetilde T_i$ and every $1 \neq \varphi \in \pi_1(T_i)$, we have $d_E(\widetilde x, \varphi \widetilde x) \geq 1$. By \reflem{EucInjectivityCusp}, this implies
\[
2 \injrad(x, U_i) \: = \: \min \{ d(\widetilde x, \varphi(\widetilde x)) : \varphi \neq 1 \}
 \: \geq \: 2 \arcsinh(1/2) \: = \: 0.96242 \ldots, 
\]
which is nearly the same as the asymptotic limit computed in \refeqn{InjecLimit}.
\end{remark}

\section{Existence of cone deformations}\label{Sec:ConeDef}

This section proves that if $M$ is a hyperbolic manifold and $\Sigma \subset M$ is a geodesic link that is sufficiently short, then there exists a cone deformation interpolating between $M$ and $M - \Sigma$. See \refthm{ConeDefExists} for a precise statement.
This result is closely related to theorems of Hodgson and Kerckhoff~\cite{hk:univ} and Bromberg~\cite{bromberg:conemflds} showing that cone deformations exist under certain conditions. However, we need a version that has explicitly quantified hypotheses, allows for multiple components of $\Sigma$, and allows $M$ to be a cusped manifold. Such a version did not previously appear in the literature. Still, our proof in this section relies heavily on the cone deformation theory developed by Hodgson and Kerckhoff \cite{hk:ConeRigidity, hk:univ, hk:shape}. In order to explain the statement and set up the proof, we review necessary background material from their work. Reviewing background from cone deformation theory will also allow us to define several important quantities and set up notation that will be used in the subsequent sections.

On the way to proving \refthm{ConeDefExists}, we will establish \refthm{ConeDefExistsRBounds}, which provides quantitative control on the radius of a maximal multi-tube about $\Sigma$. This result will be used repeatedly in the sequel.

A related theorem of Hodgson and Kerckhoff \cite[Theorem~1.2]{hk:shape} provides an interpolation by cone manifolds from $M - \Sigma$ to $M$ (i.e.\ in the opposite direction of \refthm{ConeDefExists}), provided that all meridians on the cusps of $M - \Sigma$ are sufficiently long. We recall their result below, in \refthm{UpwardConeDefRBounds}, again adding quantitative control over the radius of a multi-tube about $\Sigma$. 

In this section, and in the sequel, $\Sigma = \sigma_1 \cup \ldots \cup \sigma_n$ is a geodesic link. We use the notation 
$\ell_j = \len(\sigma_j)$ to denote the initial length of $\sigma_j$ in a non-singular metric, and $\lambda_j = \lambda_j(t) = \len_t(\sigma_j)$ to denote the length of $\sigma_j$ in a changing metric $g_t$.

\begin{theorem}\label{Thm:ConeDefExists}
Let $M$ be a finite volume hyperbolic $3$--manifold. Suppose that $\Sigma = \sigma_1 \cup \dots \cup \sigma_n$ is a geodesic link in $M$, whose components have lengths satisfying
\[ \ell_j = \len(\sigma_j) \leq 0.0996 \: \: \mbox{for all $j$}
\qquad \mbox{and} \qquad
\ell = \sum_{j=1}^n \ell_j \leq 0.15601.
\]

Then the hyperbolic structure on $M$ can be deformed to a complete hyperbolic structure on $M-\Sigma$ by decreasing the cone angle $\alpha_j$ along $\sigma_j$ from
$2\pi$ to $0$. The cone angles on all components of $\Sigma$ change in unison.
\end{theorem}

Hodgson and Kerckhoff have shown this result in the setting where $M$ is a closed hyperbolic 3--manifold and $\Sigma$ is connected~\cite[Corollary~6.3]{hk:univ}. In this special case, it suffices to assume that $\ell = \len(\Sigma) \leq 0.11058$.
Bromberg extended their result to geometrically finite manifolds without rank one cusps~\cite[Theorem~1.2]{bromberg:conemflds}. However, his hypotheses are not explicitly quantified, while we need explicit bounds under explicit hypotheses.

\subsection{Background on cone deformations}\label{Sec:Background}

Hodgson and Kerckhoff~\cite{hk:ConeRigidity} show that an infinitesimal deformation of a cone manifold structure on $M$, with singular locus $\Sigma$, can be represented as a harmonic $1$--form $\omega$ with values in the bundle $E$ of infinitesimal isometries of $X = M - \Sigma$.  Explicit information about $\omega$ is used to determine the effect of the deformation on the singular locus.

Since $X$ is a hyperbolic $3$--manifold, its bundle of infinitesimal isometries can be identified with $TX \otimes \CC \cong TX \oplus i\,TX$.  Here $(v,iw) \in TX\oplus i\,TX$ corresponds to an infinitesimal translation in the direction of $v$ and an infinitesimal rotation about an axis in the direction of $w$.  In \cite{hk:shape}, Hodgson and Kerckhoff show that $\omega$ can be taken to be \emph{harmonic}, which means it will have the form
\begin{equation}\label{Eqn:OmegaEta}
\omega = \eta + i*D\eta
\end{equation}
where $\eta$ is a $TX$--valued $1$--form on $X$, and $*$ is the Hodge star operator on forms on $X$ that takes the vector-valued 2--form $D\eta$ to a vector-valued 1--form. The forms $\eta$ and $*D\eta$ are both symmetric and traceless. Under an appropriate $L^2$ integrability condition, $\omega$ is the unique closed and co-closed harmonic form in its cohomology class; see \refrem{OmegaUniqueness} for details.

Given any component $\sigma_j$ of the singular locus $\Sigma$, Hodgson and Kerckhoff use cylindrical coordinates about $\sigma_j$ to compute two explicit closed and co-closed forms.  The first, $\omega_m = \eta_m + i*D\eta_m$, represents an infinitesimal deformation which decreases the cone angle but does not affect the real part of the complex length of the meridian.  The second, $\omega_\ell = \eta_\ell+i*D\eta_\ell$, stretches the singular locus but leaves the holonomy of the meridian unchanged.  The effects of $\omega_m$ and $\omega_\ell$ on the complex length of any peripheral curve were computed in \cite[pages~32--33]{hk:ConeRigidity} and recorded in
\cite[Lemma~2.1]{hk:univ}. 

In the following lemma, $t$ is a dummy variable expressing the ``direction'' of an infinitesimal change of metric. Part of the content of \refthm{LocalConeDeformation} will be that infinitesimal deformations can actually be promoted to local deformations, parametrized by $t$.

\begin{lemma}[Lemma~2.1 of \cite{hk:univ}]
\label{Lem:hk2.1}  
The effects of $\omega_m$ and $\omega_\ell$ on the complex length $\mathcal{L}$ of any peripheral curve are as follows.
\begin{enumerate}
\item If $\omega = \omega_m$, then $\frac{d}{dt}(\mathcal{L}) = -2\mathcal{L}$.
\item If $\omega = \omega_\ell$, then $\frac{d}{dt}(\mathcal{L}) = 2\rm{Re}(\mathcal{L})$.
\end{enumerate}
\end{lemma}

Any harmonic infinitesimal deformation affecting $\sigma_j$ alone can be written in terms of these forms:
\begin{gather}
\label{Eqn:omega}
\omega = s_j\,\omega_m + (x_j+i\,y_j)\,\omega_\ell + \omega_c,
\end{gather}
where $s_j$, $x_j$, and $y_j$ are real constants, and $\omega_c$ is an infinitesimal deformation that does not affect the holonomy of the meridian and longitude on the torus $T_j$ of distance $R$ from $\sigma_j$.  We define $\omega_0$ to be $\omega - \omega_c$. 

Because only $\omega_m$ affects the cone angle, the coefficient $s_j$ determines the change in cone angle at $\sigma_j$ for our given parametrization.

\reflem{hk2.1} implies that the effect of $\omega_0$ on the complex length $\mathcal{L}_j$ of $\sigma_j$ is given by
\begin{gather}\label{Eqn:ComplexLengthDeriv}
  \frac{d}{dt}(\mathcal{L}_j) = -2s_j\mathcal{L}_j + 2(x_j + iy_j){\rm{Re}}(\mathcal{L}_j)
\end{gather}

A central result of Hodgson and Kerckhoff \cite{hk:ConeRigidity} is that there always exists a local cone deformation that changes the cone angle on each component of $\Sigma$ at the desired rate. In fact, we may let the deformation preserve some number of closed geodesics whose cone angle is not changing. The following is a special case of \cite[Theorem 4.8]{hk:ConeRigidity}, with parametrization information added as in \cite[Page 1073]{hk:shape}.

\begin{theorem}\label{Thm:LocalConeDeformation}
Let $M$ be a finite volume hyperbolic cone manifold with singular locus $\Sigma = \sigma_1 \cup \ldots \cup \sigma_n$, such that each component of $\Sigma$ has cone angle $\alpha_j \leq 2\pi$. Let $\Sigma^+ = \sigma_1 \cup \ldots \cup \sigma_m$ be a geodesic link containing $\Sigma$. Pick a vector $(s_1, \ldots, s_m) \in \RR^m$, where $s_j = 0$ for $j > n$. Then there is a local cone deformation $(M, \Sigma^+, g_t)$, parametrized by $t$, such that
\begin{equation}\label{Eqn:SjGeneral}
\frac{d \alpha_j}{dt} = 
-2 \alpha_j s_j \, .
\end{equation}
Furthermore, the metric $g_t$ is determined up to isometry by the vector $(\alpha_1(t), \ldots, \alpha_n(t),\alpha_{n+1}, \ldots, \alpha_m)$.
\end{theorem}

In our setting, we will consider deformations where each component of $\Sigma$ has the same cone angle. As in \cite{hk:shape}, we choose the parametrization $t= \alpha^2$ for $0 \leq \alpha \leq 2\pi$, and insist that $\alpha_j = \alpha$ for every $j \leq n$.
When $\alpha > 0$, equation \refeqn{SjGeneral} becomes
\begin{gather}\label{Eqn:SjDef}
  s_j = -\frac{1}{2\alpha}\frac{d\alpha}{dt} = \frac{-1}{4\alpha^2} \qquad \mbox{for all $j \leq n$}.
\end{gather}

\subsection{Visual area and maximal tubes}
Recall from \refdef{VisualArea} that the visual area of the $j^\thsup$ component of $\Sigma$ is $\calA_j = \alpha_j \lambda_j$, and the total visual area is $\calA = \sum \calA_j$. Recall as well the notion of a maximal multi-tube from \refdef{MaximalTube}. Our goal is to ensure that the radius of $U_{\max}$ does not degenerate to $0$ during the course of the cone deformation. We do this by showing that $\calA$ is monotonic in $\alpha$ (\reflem{LenMonotonicity}) and that small visual area implies deep tubes (\refcor{HInverse}).

\begin{lemma}\label{Lem:LenMonotonicity}
Consider a local cone deformation $(M, \Sigma, g_t)$, parametrized by $t = \alpha^2$.
Let $U_{\max}(\Sigma)$ be the maximal multi-tube about $\Sigma$, and let $R$ be the smallest radius of the tubes in $U_{\max}$. Let $\len_t(\Sigma) = \sum_j \lambda_j$ denote the total length of $\Sigma$ in the cone-metric $g_t$. If $Z = \tanh R \geq 1/\sqrt{3}$,  and $t > 0$, then
\begin{equation}\label{Eqn:LenDerivBound}
\frac {d }{dt} \len_t(\Sigma) \geq \frac{\len_t(\Sigma)}{2 t} \cdot \frac{3Z^2 - 1}{Z^2(3-Z^2)} \geq 0.
\end{equation}
Furthermore,
\begin{equation}\label{Eqn:AreaDerivBound}
\frac {d \calA}{dt} \geq \frac{\calA}{2 t}  \left(\frac{3Z^2 - 1}{Z^2(3-Z^2)} + 1 \right) \geq \frac{\calA}{2 t} > 0.
\end{equation}
\end{lemma}

\begin{proof}
In our setting, every component of $\Sigma$ has the same angle $\alpha_j = \alpha$. 
Define $v = \calA/\alpha^2 = \len(\Sigma)/ \alpha$. Then
\[
\len(\Sigma) = \sum \lambda_j = \frac{\calA}{\alpha} = \alpha \, v = \sqrt{t} \, v.
\]
Consequently,
\[
\frac{d \len(\Sigma)}{dt} = 
\frac{d (\sqrt{t} \, v)}{dt} = \sqrt{t} \, \frac{dv}{dt} + \frac{v}{2 \sqrt{t} } = \frac{v}{2 \sqrt{t} } \left(  \frac{2t}{v}\frac{dv}{dt} + 1  \right)
= \frac{\len(\Sigma)}{2 t } \left(  \frac{2t}{v}\frac{dv}{dt} + 1  \right) .
\]

By \cite[Proposition~5.5]{hk:shape}, the hypothesis $R \geq \arctanh(1/\sqrt{3})$ implies
\begin{gather*}
\frac{1}{v}\frac{dv}{dt}  \geq -\frac{1}{\sinh^2 R }\left( \frac{2\cosh^2 R - 1}{2\cosh^2 R + 1}\right)\frac{1}{\alpha}\frac{d\alpha}{dt}
   = - \left( \frac{1- Z^4}{Z^2(3-Z^2)} \right) \frac{1}{2t},
\end{gather*}
where the last equality uses $Z = \tanh R$ and  $t=\alpha^2$. Then 
\begin{gather}\label{Eqn:VDerivBound}
  \frac{2t}{v}\frac{dv}{dt} + 1 \geq -\frac{1- Z^4}{Z^2(3-Z^2)} + 1 = \frac{3Z^2 - 1}{Z^2(3-Z^2)} .
\end{gather}
Since $Z = \tanh R < 1$, the denominator of the last expression is always positive. The numerator will be non-negative whenever $Z \geq 1/\sqrt{3}$, hence the whole expression in \refeqn{VDerivBound} is non-negative. Thus
\[
\frac{d \len(\Sigma)}{dt} 
= \frac{\len(\Sigma)}{2 t } \left(  \frac{2t}{v}\frac{dv}{dt} + 1  \right)
\geq \frac{\len_t(\Sigma)}{2 t} \cdot \frac{3Z^2 - 1}{Z^2(3-Z^2)} \geq 0,
\]
establishing \refeqn{LenDerivBound}. For \refeqn{AreaDerivBound}, we recall that $\calA =  \alpha^2  v =  tv$. Thus
\[
\frac{d \calA}{dt} = t \frac{dv}{dt} + v = \frac{v}{2}  \left(  \frac{2t}{v}\frac{dv}{dt} + 2  \right)
= \frac{\calA}{2t}   \left(  \frac{2t}{v}\frac{dv}{dt} + 2  \right)
\geq \frac{\calA}{2t}  \left(\frac{3Z^2 - 1}{Z^2(3-Z^2)} + 1 \right).
\]
Since the expression in \refeqn{VDerivBound} is non-negative, \refeqn{AreaDerivBound} follows.
\end{proof}

Our goal is to bound the tube radius throughout a cone deformation. Following Hodgson and Kerckhoff,
we do this by using \refcor{HInverse}, which can be rephrased as follows: 
 if the tube radius at some initial time $t$ is larger than $0.531$, and $\calA(t)$ remains  smaller than $\hmax \approx 1.0196$ throughout the cone deformation, then the tube radius will remain large.

To apply \reflem{LenMonotonicity}, we need to ensure that the cone-locus $\Sigma$ has a tube of radius  $R \geq \arctanh(1/\sqrt{3}) > 0.531$. This minimal assumption on tube radius will appear in many results below.

\subsection{Decreasing cone angles to $0$}

Recall that by \refthm{LocalConeDeformation}, there always exists a local cone deformation on $(M, \Sigma)$ that decreases the cone angle on each component of $\Sigma$ from $\alpha$ to $\alpha - \epsilon$, for some small $\epsilon > 0$. 
To show that the cone deformation can be continued, we apply a result of Hodgson and Kerckhoff.

\begin{theorem}[Theorem~3.12 of~\cite{hk:univ}]\label{Thm:hk3.12}
Suppose $M_t$ for  $t \in [0, t_\infty)$ is a smooth path of finite volume hyperbolic cone manifold structures on $(M, \Sigma)$ with cone angle $\alpha_j(t)$ along the $j^\thsup$ component of the singular locus $\Sigma$. Suppose $\alpha_j(t) \in [0,2\pi]$ for all $t$, and $\alpha_j(t) \to a_j$ as $t\to t_\infty$. Suppose there is a constant $\Rmin > 0$ such that there is an embedded tube of radius at least $\Rmin$ around $\Sigma$ for all $t$. Then the path extends continuously to $t=t_\infty$ so that as $t\to t_\infty$, $M_t$ converges in the bilipschitz topology to a cone manifold structure $M_\infty$ on $M$ with cone angle $a_j$ along the $j^\thsup$ component of $\Sigma$.
\end{theorem}

\begin{proof}
This theorem is exactly \cite[Theorem~3.12]{hk:univ}, except for three minor differences in the statement:
\begin{enumerate}
\item \label{Itm:Closed} The result \cite[Theorem~3.12]{hk:univ} is stated for closed manifolds rather than finite-volume manifolds.
\item \label{Itm:SameAngle} It is stated for cone structures where all cone angles around the singular locus agree. In fact, we can be more flexible with parametrizing the deformation.
\item \label{Itm:Volume} It is stated for cone-manifolds satisfying a uniform upper volume bound, independent of $t$.
\end{enumerate}

Hypothesis~\refitm{Volume} can be omitted because it holds automatically. This follows from a construction of Agol \cite{agol:drilling}, as follows. Agol uses the cone-metric $g_t$ (which is non-singular outside a tube about $\Sigma$) to construct a complete metric of pinched negative curvature on $M - \Sigma$, which we denote $h_t$. The sectional curvatures of this metric are bounded in terms of the constant $\Rmin > 0$, while $\vol(h_t)$ differs from $\vol(g_t)$ by a multiplicative factor that depends only on $\Rmin$. Furthermore, by a result of Boland, Connell, and Souto \cite{boland-connell-souto}, $\vol(h_t)$  differs by a bounded multiplicative factor from the volume of the complete hyperbolic metric, denoted $\vol(M - \Sigma)$. Consequently, Agol's work gives a uniform upper bound on $\vol(M_t)$ as a function of $\Rmin$ and $\vol(M -\Sigma)$.\protect\footnote{The main result of Agol's paper \cite[Theorem 2.1]{agol:drilling} uses this construction to bound the ratio $\vol(M - \Sigma)/\vol(M_t)$ from above. However the same ingredients also suffice to bound $\vol(M - \Sigma)/\vol(M_t)$ from below.}

Hypothesis~\refitm{SameAngle} can be omitted because it is never used in the proof of \cite[Theorem~3.12]{hk:univ}. The proof goes through verbatim without this assumption. 

Issue \refitm{Closed} can now be resolved by an appeal to \refitm{SameAngle}. Let $\Sigma^+$ consist of geodesics and cusps, where the cusps have cone angle $0$. Now apply \cite[Theorem~3.12]{hk:univ} to $\Sigma^+$, so that the cone angle remains $0$ on all cusps that remain unfilled. This immediately gives the result for finite volume manifolds.
\end{proof}

\begin{theorem}\label{Thm:ConeDefExistsRBounds}
Suppose $M$ is a finite volume hyperbolic $3$--manifold, and $\Sigma = \sigma_1 \cup \dots \cup \sigma_n$ is a geodesic link in $M$, of total length $\ell = \len(\Sigma) \leq 0.15601 \approx h(\arctanh(1/\sqrt{3}))/(2\pi)$. Let $R$ be the radius of a maximal embedded tube  about $\Sigma$, and assume $R \geq 0.531$. 
Define  $\Rmin = h^{-1}(2\pi \ell)$, and note that this value exists by \refcor{HInverse}.

Then the hyperbolic structure on $M$ can be deformed to the complete hyperbolic structure on $M-\Sigma$ by decreasing the cone angles on $\Sigma$ from $2\pi$ to $0$ in such a way that at any time $t$,
\begin{enumerate}
\item\label{Itm:SameAngle2} Every component of $\Sigma$ has cone angle $\alpha = \sqrt{t}$,
\item\label{Itm:BigRad} The tube radius in $M_t$ about $\Sigma$ is $R(t) \geq \Rmin \geq \arctanh(1/\sqrt{3})$,
\item\label{Itm:SmallA} If $t > 0$, we have $\calA'(t) > 0$. 
\end{enumerate}
\end{theorem}

\begin{proof}
By \refthm{LocalConeDeformation}, there exists a cone deformation with cone angles near $2\pi$, parametrized by $t=\alpha^2$. At the maximal value of $t$, namely $t=(2\pi)^2$, we have
\[
\calA(t) = 2\pi \ell = h(\Rmin) \leq h( \arctanh(1/\sqrt{3} )) < \hmax,
\]
hence $R \geq \Rmin\geq \arctanh(1/\sqrt{3})$ by \refcor{HInverse}. By \reflem{LenMonotonicity}, we have $\calA'((2\pi)^2) > 0$.

Let $I\subset[0,(2\pi)^2]$ be the maximal subinterval containing $(2\pi)^2$, such that conclusions \refitm{SameAngle2}, \refitm{BigRad}, and \refitm{SmallA} all hold for $t \in I$. In the previous paragraph, we checked that 
$I$ contains $(2\pi)^2$, so is not empty. 

Next, we show that $I$ is open. Suppose $t_0\in I$, so there exists a hyperbolic cone manifold structure on $M$ with cone angles $\alpha_0=\sqrt{t_0}$. On a small neighborhood of $t_0$, condition \refitm{SameAngle2} holds as a consequence of \refthm{LocalConeDeformation}: there exists a local cone deformation with cone angles near $\alpha_0$, parametrized by $t=\alpha^2$ for $t$ near $t_0$.
Condition \refitm{SmallA} is an open condition, hence $\calA'(t) > 0$ in a small neighborhood of $t_0$ in $[0, (2\pi)^2]$. Therefore, in the union of $I$ and this small neighborhood,
we have $\calA(t)\leq \calA((2\pi)^2) \leq h(\Rmin)$,  hence $R \geq \Rmin \geq\arctanh(1/\sqrt{3})$ by \refcor{HInverse}. So condition \refitm{BigRad} is satisfied as well in this neighborhood, and $I$ is open.

Now, we show that $I$ is closed. Let $t_\infty = \inf I$. By \refthm{hk3.12}, the assumption that $R(t) \geq \arctanh(1/\sqrt{3})$ for $t \in I$ implies that the cone deformation extends to time $t_\infty$, hence \refitm{SameAngle2} holds. Second, note that \refitm{BigRad} is a closed condition, hence $R(t_\infty) \geq \Rmin$ by continuity. Third, by \reflem{LenMonotonicity}, $R(t_\infty) \geq \arctanh(1/\sqrt{3})$ implies that if $t_\infty >0$, then $\calA'(t_\infty) > 0$, hence condition \refitm{SmallA} holds. Finally, if $t_\infty = 0$, then condition \refitm{SmallA} holds vacuously. Therefore, $I$ is closed.

Since $I$ is open, closed, and non-empty, it follows that $I=[0,(2\pi)^2]$, hence the desired cone deformation interpolates all the way between cone angle $2\pi$ and $0$.
\end{proof}

The style of argument in the above proof will be employed several more times in the paper. Conditions \refitm{SameAngle2}--\refitm{SmallA} are mutually reinforcing, with the property that if they hold on an interval $I$, then they also hold on a slightly larger interval. If $I$ is closed, the conclusions hold on a neighborhood of the endpoint; if $I$ is open, they hold  on the closure. This continuous analogue of induction will be called a \emph{crawling argument}.

We will prove \refthm{ConeDefExists} by applying \refthm{ConeDefExistsRBounds}. \refthm{ConeDefExistsRBounds} needs a hypothesis on the length $\ell$ 
and a hypothesis on the radius of the maximal tube. Meanwhile, \refthm{ConeDefExists} only has hypotheses on length. 
It turns out that for non-singular manifolds, the tube radius can be estimated from length alone.

\begin{lemma}\label{Lem:Meyerhoff}
Let $M$ be a hyperbolic $3$--manifold. Let $\Sigma \subset M$  be a geodesic link with components $\sigma_1, \dots, \sigma_n$, such that $\len(\sigma_j) \leq 0.0996$ for every $j$.
Then the maximal embedded tube about $\Sigma$ has radius $R > 0.531$.
\end{lemma}

See \cite[Proposition 1.10]{GabaiMeyerhoffThurston} for a very similar statement, with slightly different numbers in the hypotheses and the conclusion. Our proof, using results of Meyerhoff~\cite{meyerhoff}, is based on the proof of that proposition.

\begin{proof}[Proof of \reflem{Meyerhoff}]
This follows from a theorem of Meyerhoff \cite[Section~3]{meyerhoff}.
For each $j$, let $\calL_j = \ell_j + i \tau_j$ be the complex length of $\sigma_j$. For each $j$, Meyerhoff constructs an embedded tube about $\sigma_j$ whose radius $r_j$ satisfies
\[
\sinh^2 r_j =  \frac{\sqrt{1 - 2 k(\calL_j)}}{2 k(\calL_j)} - \frac{1}{2} ,
\quad
\text{where}
\quad
k(\calL_j) = \min_{m \in \NN}   \{ \cosh (m \ell_j) - \cos (m \tau_j)  \} .
\]
Furthermore, the tubes about different components are disjoint \cite[Section~7]{meyerhoff}.

Observe that $\frac{\sqrt{1-2k}}{2k}$ is a decreasing function of $k$ when $k \in (0, \sqrt{2}-1)$, and that $r_j = 0.531$ when $k(\calL_j) = 0.34932 \ldots$. Thus it remains to show that $k(\calL_j) \leq 0.34932$ for all $\ell_j \in [0, 0.0996]$ and $\tau_j \in [0, 2\pi]$. Since $\cosh(m \ell_j)$ is an increasing function of $\ell_j$, it suffices to set $\ell_j = 0.0996$. Since $\cos(m \tau_j)$ is an even function of $\tau_j$, it suffices to consider values $\tau_j \in [0, \pi]$.

Finally, we claim that for every $\tau_j \in [0, \pi]$, there is an integer $m \in \{1, \ldots, 8\}$ such that $\cosh (m \cdot 0.0996) - \cos (m \tau_j) \leq 0.34932$. This is verified using interval arithmetic in Sage; see the ancillary files \cite{FPS:Ancillary} for details.
\end{proof}

\begin{proof}[Proof of \refthm{ConeDefExists}]
Suppose that $\Sigma = \sigma_1 \cup \dots \cup \sigma_n$ is a geodesic link in $M$, such that each component has length $\len(\sigma_j) \leq 0.0996$ and $\sum \len(\sigma_j) \leq 0.15601$. Since $\len(\sigma_j) \leq 0.0996$ for each $j$, \reflem{Meyerhoff} says the maximal tubular neighborhood of $\Sigma$ has radius $R > 0.531$. Since $\ell = \sum \len(\sigma_j) \leq 0.15601$, \refthm{ConeDefExistsRBounds} implies that we may deform the cone angles on $\sigma_j$ downward from $2\pi$ to $0$.
\end{proof}
  
\subsection{Increasing cone angles from $0$}

Next, we present a companion result to \refthm{ConeDefExistsRBounds}, whose hypotheses are on the drilled manifold $M-\Sigma$ rather than the on the filled manifold $M$ where $\Sigma$ is non-singular. Recall that normalized length was defined in \refdef{NormalizedLengthIntro}. If the total normalized length of all meridians in $M - \Sigma$ is sufficiently large, one obtains a cone deformation from $M - \Sigma$ to $M$, with control on tube radii.

\begin{definition}\label{Def:Ifunction}
Define a function $I \from (0,1) \to \RR$ by
\[
I(z) = \frac{(2 \pi)^2}{3.3957 \, (1-z)} \, \exp \left( \int_z^1 \frac{1+4w+6w^2+w^4}{(1+w)(1+w^2)^2} dw \right),
\]
where $z= \tanh r$ as usual. This function has a unique critical point: a global minimum when $z = \sqrt{\sqrt{5}-2}$, with minimum value $56.469\ldots$.
The function is monotonically increasing for larger $z$, hence for $r \geq 0.531$. It blows up as $z \to 1$. See \cite[Pages 409--410]{hk:univ}.
\end{definition}

Hodgson and Kerckhoff proved the following result.

\begin{theorem}\label{Thm:UpwardConeDefRBounds}
Let $M$ be a $3$--manifold with empty or toroidal boundary, and $\Sigma$  a smoothly embedded link in $M$. Suppose that $M-\Sigma$ is a cusped hyperbolic manifold such that the total normalized length of the meridians of $\Sigma$ satisfies 
\[
L^2 \geq I(\Zmin),
\quad \mbox{where} \quad 
\Zmin = \tanh(\Rmin) \geq 1/\sqrt{3}. 
\]
Then $M$ admits a hyperbolic metric in which $\Sigma$ is isotopic to a union of geodesics. Furthermore, the hyperbolic structure on $M-\Sigma$ can be deformed to that of $M$ via a family of cone-manifolds $M_t$, while maintaining the following properties.
\begin{enumerate}
\item Every component of $\Sigma$ has the same cone angle in $M_t$,
\item The tube radius in $M_t$ about $\Sigma$ is $R(t) \geq \Rmin \geq \arctanh(1/\sqrt{3})$,
\item $\calA(t) < \hmax$.
\end{enumerate}
\end{theorem}

\begin{proof}
This is essentially \cite[Theorem~5.11]{hk:shape}, with information about tube radius extracted from the proof. 
By the remark following \cite[Theorem~4.8]{hk:ConeRigidity}, there is a family of cone-manifolds $(M,\Sigma, g_t)$ for $t \in [0,\epsilon)$, in which the cone angles on $\Sigma$ agree for each $t$.
In \cite[Theorem~5.8]{hk:shape}, Hodgson and Kerckhoff prove that the deformation can be continued so long as $R(t) \geq \arctanh(1/\sqrt{3})$. Meanwhile, in 
\cite[Theorem~5.7]{hk:shape} and the discussion preceding the theorem, they 
show that so long as $L^2 \geq I(\Zmin)$, and the cone angles are at most $2\pi$, the tube radius $R(t)$ will stay bounded below by $\Rmin \geq \arctanh(1/\sqrt{3})$. Thus the deformation can be continued all the way up to cone angle $2\pi$, where we reach the complete hyperbolic metric on $M$. The link $\Sigma$ is geodesic in each cone metric $g_t$, hence is also geodesic in the non-singular metric at cone angle $2\pi$.
\end{proof}

We conclude this section with a particularly natural choice of the harmonic form $\omega$.

\begin{remark}\label{Rem:OmegaUniqueness}
Recall from \refsec{Background} that an infinitesimal deformation of the cone-metric $g_t$ is determined by a harmonic $1$--form $\omega$ defined on $X = M - \Sigma$, with values in the bundle $E \cong TX \oplus iTX$ of infinitesimal isometries of $X$. By \refthm{LocalConeDeformation}, the local family of cone-metrics $g_t$ is determined up to isometry by its cone-angles, but different choices of $\omega$ within the same cohomology class in $H^1(X,E)$ lead to different choices of cone-metric within the same isometry class. In our bilipschitz theorem in \refsec{Bilip}, it will be important to have a natural way to identify points of $(M,\Sigma, g_a)$ with points of $(M, \Sigma, g_b)$, for the purpose of comparing the metrics $g_a$ and $g_b$ at a point $p \in X$. To that end, we pin down a canonical choice of $\omega$.

Suppose $\widetilde \omega$ is a smooth $E$--valued $1$--form on $X = M - \Sigma$. In \cite[Theorem 2.7]{hk:ConeRigidity}, Hodgson and Kerckhoff prove that so long as all cone angles are at most $2\pi$, which is always the case in our setting, there is a unique closed and co-closed harmonic form $\omega$ such that $ [\widetilde \omega] = [\omega] \in H^1(X,E)$, and furthermore $\widetilde \omega - \omega = ds$, where $s$ is a globally defined $L^2$ section of $E$. This choice of $\omega$ determines the one-parameter family of cone-metrics $g_t$ on the nose, and defines a \emph{natural identity map} $\id \from (M -\Sigma, g_a) \to (M -\Sigma, g_b)$ that allows us to compare the metrics at any given point. Because of the canonical way in which $\omega$ is chosen, the identity map conjugates every isometry of $(M,\Sigma, g_a)$ to an isometry of  $(M,\Sigma, g_b)$. We will say, for short, that the identity map is \emph{equivariant with respect to the symmetry group of $(M, \Sigma)$}.

In Sections~\ref{Sec:Bilip} and~\ref{Sec:Margulis}, we will always use this $1$--form $\omega$ and the accompanying identity map. In \refsec{ShortGeodesic}, where we will need the flexibility to enlarge $\Sigma$ to a larger link $\Sigma^+$ containing a non-singular geodesic, we will accordingly choose a harmonic form $\omega$ with reference to $\Sigma^+$.
\end{remark}

\section{Bounding the boundary terms}\label{Sec:BoundaryBound}

In this section, we will find explicit bounds on certain boundary terms that arise in the cone deformation. These boundary terms were used in \cite{hk:ConeRigidity} to prove that there are no infinitesimal deformations of hyperbolic cone manifolds fixing the cone angles. They have been used in many other applications of cone deformations to obtain geometric control. We will use boundary terms in \refsec{ShortGeodesic} to bound the change in length of a non-singular geodesic, and in \refsec{Bilip} to get bilipschitz estimates in the thick part of a manifold.

This section is quite technical, reviewing definitions and results from \cite{hk:ConeRigidity, hk:univ, hk:shape} that require significant work from analysis and differential geometry to state and to prove. For our applications, we need only the results (technical though they are), and not the analysis. Therefore we will skim over some of the definitions and results quickly, sweeping the complicated work of \cite{hk:ConeRigidity,hk:univ,hk:shape} into the references, pointing the reader to statements in those papers for careful definitions and details. Our goal in being brief is to attempt to avoid unnecessary complications that are peripheral to our applications. The reader interested only in the applications can view this section as a black-box, while the reader with more interest in cone deformations can still turn to the references for details.

\subsection{Definitions and setup}

Throughout, $(M,\Sigma)$ will be a hyperbolic cone manifold. We will also consider a submanifold $X\subset M$ with \emph{tubular boundary}: This means that $X$ is either a model tube, or the complement of some number of model tubes. 
We orient the boundary of $X$ by \emph{inward} normal vectors. This orientation will be important, as it affects the signs of our results.

Recall from \refsec{Background} that an infinitesimal deformation of a cone manifold structure can be represented by a harmonic $1$--form $\omega$, and that we made a canonical choice of $\omega$ in \refrem{OmegaUniqueness}. The harmonic form $\omega$ decomposes as $\omega = \eta + i *D\eta$, as in \refeqn{OmegaEta}. In \cite[Proposition~1.3 and page~36]{hk:ConeRigidity}, Hodgson and Kerckhoff show that integrating by parts over the submanifold $X$, again oriented by inward normal, gives
\begin{equation}\label{Eqn:BoundaryTerm}
\int_X \| \omega \|^2 dV  = 
\int_X ||D\eta||^2 + ||\eta||^2 \, dV =
\int_{\bdy X} *D\eta \wedge \eta.
\end{equation}
See also \cite[Lemma~2.3]{hk:univ} for a formulation of the result in notation that better matches ours.

The term on the far right of \refeqn{BoundaryTerm} is important. Thus Hodgson and Kerckhoff define the \emph{boundary term} $b_X$ on $TX$--valued 1--forms $\mu$ and $\nu$ as follows.
\begin{equation}\label{Eqn:BoundaryTermDef}
b_X(\mu,\nu) = \int_{\bdy X} *D\mu\wedge\nu.
\end{equation}
Thus the term on the far right of \refeqn{BoundaryTerm} becomes $b_X(\eta,\eta)$.

Next, recall from \refeqn{omega} and the ensuing discussion that $\omega$ can be written as a sum $\omega = \omega_0 +\omega_c$ where $\omega_0$ is written in terms of the explicit forms $\omega_m$ and $\omega_\ell$ that affect meridian and longitude, and $\omega_c$ is a correction term. We may write $\omega_0 = \eta_0 +i*D\eta_0$ and $\omega_c = \eta_c+i*D\eta_c$. Then \refeqn{BoundaryTerm} becomes
\begin{equation}\label{Eqn:Weitzenbock}
  \int_X \| \omega \|^2 dV = b_X(\eta, \eta) =  b_X(\eta_0, \eta_0) + b_X(\eta_c, \eta_c),
\end{equation}
using \cite[Lemma~2.5]{hk:univ} (the cross terms vanish). 
See \cite[Equation (6) and (7)]{hk:shape}, where integration is implicit in their definition of the $L^2$ norm. 

We emphasize that the above formulas \refeqn{BoundaryTerm}--\refeqn{Weitzenbock} hold both when $X$ is a model tube and when $X$ is the complement of some number of model tubes. This flexibility will be important in \refsec{ShortGeodesic}.

For the rest of this section, and in \refsec{Bilip}, boundary terms will appear in the following specific context.
Let $\rr = (r_1, \ldots, r_n)$ be a vector of positive radii. Suppose that $U_\rr = U_\rr(\Sigma)$ is an embedded multi-tube about the singular locus $\Sigma$, as in \refdef{MultiTube}, and let $X_\rr =M-U_\rr$. The inward normal vectors that orient $\bdy X_\rr$ point away from $\Sigma$.
For any $TX_\rr$--valued $1$--forms $\mu$ and $\nu$, define \[
b_\rr(\mu, \nu) = b_{X_\rr}(\mu, \nu) = \int_{\bdy X_\rr} *D\mu \wedge \nu.
\]

\begin{lemma}\label{Lem:BoundaryTermSigns}
Let $\rr = (r_1, \dots, r_n)$, where $\tanh(r_j) \geq 1/\sqrt{3}$ for all $j$. Then
\[
b_\rr(\eta_c, \eta_c) \leq 0.
\]
\end{lemma}

\begin{proof}
Let $U$ be a solid torus of radius $r$. Then the principal curvatures of $\bdy U$ are $k_1 = \tanh r$ and $k_2 = \coth r$. See, for instance, \cite[Page 1066]{hk:shape}. Thus, under the hypotheses of the lemma, the principal curvatures along every component $\bdy X$ satisfy $1/\sqrt{3} \leq k_1 \leq k_2 \leq \sqrt{3}$.

Under this hypothesis on principal curvatures, Hodgson and Kerckhoff prove in  \cite[Theorem 4.2]{hk:shape} that $b_\rr(\eta_c, \eta_c) \leq 0$. 
\end{proof}

We remark that the hypotheses of \reflem{BoundaryTermSigns} also imply $b_\rr(\eta_0, \eta_0) > 0$. See \cite[Corollary 4.3]{hk:shape}.

We will need an upper bound on $b_\rr(\eta, \eta)$. 
By \reflem{BoundaryTermSigns}, this amounts to finding an upper bound on $b_\rr(\eta_0, \eta_0)$.

\begin{lemma}\label{Lem:hkshape5.4}
With the parametrization $t = \alpha^2$, the boundary term $b_\rr(\eta_0, \eta_0)$ satisfies
\[
b_\rr(\eta_0, \eta_0) 
  \leq \sum_{j=1}^n \frac{4(1-z_j^2)}{z_j^2(3-z_j^2)} \cdot\frac{1}{16\alpha^4}\cdot \calA_j
\]
where $z_j = \tanh(r_j)$, and $\alpha$ is the cone angle, and $\calA_j$ is the visual area.
\end{lemma}

\begin{proof}
This result  is contained in the proof of~\cite[Proposition~5.4]{hk:shape}. On the bottom of page 1074 and the top of page 1075, it is shown that
\[
b_\rr(\eta_0, \eta_0) \leq \sum_{j=1}^n \frac{4a_jc_j - b_j^2}{4a_j}s^2\calA_j
\]
where $s=\frac{-1}{2\alpha}\frac{d\alpha}{dt}$, and where
$a_j$, $b_j$, and $c_j$ are as in \cite[equation~(32)]{hk:shape}:
\begin{gather*}
a_j = \frac{-\sinh^2 r_j}{\cosh^2 r_j}(2\cosh^2 r_j + 1), \qquad
b_j = \frac{-2}{\cosh^2 r_j}, \qquad
c_j = \frac{2\cosh^2 r_j - 1}{\sinh^2 r_j \cosh^2 r_j}.
\end{gather*}
As in \cite[page~1079]{hk:shape}, we let $t=\alpha^2$. Thus, as in \refeqn{SjDef}, we have $s=-1/(4\alpha^2)$. 
Rewriting $a_j$, $b_j$, and $c_j$ in terms of $z_j$, using \reflem{TanhSinhCosh}, gives the result.
\end{proof}

\subsection{Controlling length and visual area}

The next several lemmas prove estimates relating how visual area changes under cone deformations. These results culminate in an estimate relating the normalized length $L$, measured on cusps in the complete metric on $M - \Sigma$, to the length $\ell = \len_{4\pi^2}(\Sigma)$ in the complete metric on $M$. This will feed into the bound on boundary terms later in this section.

\begin{remark}\label{Rem:BdryNotation}
We recall notation that will be used below. As usual, we are assuming that a cone deformation $M_t$ is parameterized by $t=\alpha^2$ where $\alpha$ is the cone angle along each component of the singular locus $\Sigma$. We let $R$ denote the smallest radius in a maximal multi-tube $U$ about $\Sigma$ in $M_t$. 
If $\sigma_j$ is a component of the singular locus, with length $\ell_j$, recall from \refdef{VisualArea} that the visual area of the tube component $U_j$ of $U$ about $\sigma_j$ is defined to be $\calA_j = \alpha\ell_j$. The visual area of the union of all tubes is $\calA = \sum \calA_j$. 

In the proof of Lemma~\ref{Lem:LenMonotonicity}, we introduced the
variable $v=\calA/\alpha^2$. We now let $u = 1/v$. As above, we set $Z = \tanh(R)$.
\end{remark}

\begin{lemma}\label{Lem:hkp1079}
Suppose that $Z = \tanh R \geq 1/\sqrt{3}$. Let $u(t)=u(\alpha^2)=\alpha^2/\calA$. Then $du/dt$ satisfies
\[ 
-G(Z) \leq \frac{du}{dt} \leq \widetilde{G}(Z)
\]
where
\[
G(z) = \frac{1+z^2}{6.7914\, z^3} 
\quad \mbox{and} \quad 
\widetilde{G}(z) = \frac{(1+z^2)^2}{6.7914\, z^3\, (3-z^2)}. 
\]
Furthermore, $G(z)$ and $\widetilde{G}(Z)$ are strictly decreasing on the interval $(0,1)$.
\end{lemma}

\begin{proof}
The bound on $du/dt$ is proved on page 1079 of \cite{hk:shape}. The behavior of $G(z)$ and $\widetilde{G}(z)$ can be checked by differentiation.
\end{proof}

\begin{lemma}\label{Lem:uBounds}
Suppose we have a cone deformation from cone angle $0$ to $\alpha >0$, 
so that throughout the deformation, the maximal multi-tube about $\Sigma$ has radius $R \geq \Rmin$, where $\Zmin = \tanh \Rmin \geq 1/\sqrt{3}$. 
Then the function
$u(\alpha^2)  = u(t) = \alpha^2/\calA$ satisfies:
\[
L^2 - G(\Zmin) \alpha^2 
  < u(\alpha^2) 
  < L^2 + \widetilde{G}(\Zmin) \alpha^2.
\]
Here $L$ is the total normalized length of the meridians of the drilled manifold $M-\Sigma$, as in \refdef{NormalizedLengthIntro}. 
In particular, for $0<\alpha \leq 2\pi$,
\[
L^2 - G(\Zmin) (2\pi)^2 
  < u(\alpha^2) 
  < L^2 + \widetilde{G}(\Zmin) (2\pi)^2.
\]
\end{lemma}

\begin{proof}
Hodgson and Kerckhoff showed that as cone angle decreases to $0$, we have
\[ u(0) = \lim_{t\to 0} \, u(t) = L^2. \]
See \cite[page~1076]{hk:shape}. Then at time $\alpha^2$, we have
\begin{align*}
  u(t) 
  &= u(0) + \int_0^{\alpha^2} \frac{du}{d \tau} d\tau \\
  & \geq L^2 - \int_0^{\alpha^2} G(Z(\tau)) \, d\tau \quad \mbox{by \reflem{hkp1079}}\\\
  & > L^2 - \int_0^{\alpha^2} G(\Zmin) \, d\tau \quad \mbox{using strict monotonicity of $G$}\\
  & = L^2 - G(\Zmin)\, \alpha^2 . 
\end{align*}
The upper bound is obtained similarly, using the strict monotonicity of $\widetilde{G}$. 
\end{proof}

A very similar argument gives the following.

\begin{lemma}\label{Lem:uBoundsEll}
Suppose we have a cone deformation from cone angle $2\pi$ to $\alpha$, so that throughout the cone deformation, the maximal multi-tube about $\Sigma$ has radius $R \geq \Rmin$, where $\Zmin = \tanh \Rmin \geq 1/\sqrt{3}$. 
Then for $\alpha<2\pi$, the function $u(\alpha^2) = u(t) = \alpha^2/\calA$ satisfies:
\[
\frac{2\pi}{\ell} - \widetilde{G}(\Zmin) ((2\pi)^2 - \alpha^2)
< u(\alpha^2)
< \frac{2\pi}{\ell} + G(\Zmin) ((2\pi)^2 - \alpha^2)
\]
where $\ell$ denotes the total length of $\Sigma$ at cone angle $2\pi$.
In particular, 
\[
\frac{2\pi}{\ell} - \widetilde{G}(\Zmin) (2\pi)^2
  < u(\alpha^2) <
\frac{2\pi}{\ell} + G(\Zmin) (2\pi)^2.
\]
\end{lemma}

\begin{proof}
At cone angle $2\pi$, we have $t = (2\pi)^2$, hence
\[ u((2\pi)^2) = \frac{\alpha^2}{\calA} = \frac{\alpha^2}{\sum \alpha\len(\sigma_j)} = \frac{2\pi}{\ell} .\]
Now, we can set up an integral, as above:
\begin{align*}
  u(t) 
  & = u((2\pi)^2) - \int_{\alpha^2}^{(2\pi)^2} \frac{du}{d\tau} d\tau \\
  & \geq \frac{2\pi}{\ell} - \int_{\alpha^2}^{(2\pi)^2} \widetilde{G}(Z(\tau)) \, d\tau \quad \mbox{by \reflem{hkp1079}} \\
  & > \frac{2\pi}{\ell} - \int_{\alpha^2}^{(2\pi)^2} \widetilde{G}(\Zmin) \, d\tau \quad \mbox{using strict monotonicity of $\widetilde{G}$}\\
  & = \frac{2\pi}{\ell} - \widetilde{G}(\Zmin)\, ((2\pi)^2 - \alpha^2).
\end{align*}
The upper bound is obtained similarly, using the monotonicity of $G$. 
\end{proof}

We can now relate the total normalized length $L$ to the total length $\ell = \ell(\Sigma)$  at cone angle $2\pi$. This result generalizes a lemma of Magid \cite[Lemma~4.7]{magid:deformation} to cone deformations with multiple components, while sharpening the estimate and making hypotheses explicit. 
It also converts the asymptotic formula of Neumann and Zagier~\cite[Proposition~4.3]{neumann-zagier} into a two-sided inequality.



\begin{lemma}\label{Lem:MagidLengthGeneral}
Suppose that $M$ is a complete finite volume hyperbolic $3$--manifold. Fix a constant $\Rmin > 0$ such that $\Zmin = \tanh \Rmin \geq 0.6622$; note this is strictly larger than $1/\sqrt{3}$. Suppose that $\Sigma \subset M$ is a geodesic link such that one of the following hypotheses holds.
\begin{enumerate}
\item
\label{Itm:cusped-assum} 
In the complete structure on $M - \Sigma$, the total normalized length of the meridians of $\Sigma$ satisfies $L^2 \geq I(\Zmin)$, where $I$ is the function of \refdef{Ifunction}. 
\item
\label{Itm:filled-assum} 
In the complete structure on $M$, each component of $\Sigma$ has length at most $0.0996,$ while the total length of $\Sigma$ is $\ell \leq \haze(\Zmin)/(2\pi)$, where $\haze$ is the function of \refrem{Haze}.
\end{enumerate}
Then we have the double-sided inequality
\[
\frac{2\pi}{L^2 + \widetilde{G}(\Zmin)(2\pi)^2} 
  < \ell
  < \frac{2\pi}{L^2 - G(\Zmin)(2\pi)^2},
\]
where $G$ and $\widetilde{G}$ are defined in \reflem{hkp1079}.

Furthermore, $M - \Sigma$ and $M$ are connected by a cone deformation maintaining a multi-tube about $\Sigma$ of radius $R$, where $\tanh R = Z > \Zmin$ throughout.
\end{lemma}

\begin{proof}
If \refitm{cusped-assum} holds, \refthm{UpwardConeDefRBounds} proves the existence of a cone deformation from $M-\Sigma$ to $M$ that maintains $Z > \Zmin$ throughout.

If \refitm{filled-assum} holds, then \reflem{Meyerhoff} implies that the maximal tube about $\Sigma$ has radius $R > 0.531$. Furthermore, since $\ell \leq \haze(\Zmin)/(2\pi) = h(\Rmin)/(2\pi)$, \refthm{ConeDefExistsRBounds} proves the existence of a cone deformation from $M$ to $M-\Sigma$ that maintains $Z > \Zmin$ throughout.

Applying \reflem{uBounds} with $\alpha \leq 2\pi$ gives
\begin{equation}\label{Eqn:uL-Bound}
 L^2 - G(\Zmin)(2\pi)^2
  \: < \: 
 u \: < \:
 L^2 + \widetilde{G}(\Zmin)(2\pi)^2.
\end{equation}
Now, we substitute $u = \alpha/\ell$ and $\alpha = 2\pi$, obtaining
\begin{equation}\label{Eqn:li-bound}
L^2 - G(\Zmin)(2\pi)^2
  \: < \: 
 \frac{2\pi}{\ell} \: < \:
 L^2 + \widetilde{G}(\Zmin)(2\pi)^2.
\end{equation}

We need to make sure that the lower bound on $2\pi/\ell$ is strictly positive, to invert the three quantities in \refeqn{li-bound}.

If $L^2 \geq I(\Zmin)$ with $\Zmin\geq 1/\sqrt{3}$, then $L^2\geq I(1/\sqrt{3}) > 57.504$ by the monotonicity of $I$; see \refdef{Ifunction}. Meanwhile, $G(\Zmin)(2\pi)^2 \leq G(1/\sqrt{3}) < 40.274$, hence the lower bound is positive in this case.

If $\ell \leq \haze(\Zmin)/(2\pi)$, the second inequality in \refeqn{li-bound} ensures that
\[ L^2 > (2\pi)^2 \left( \frac{1}{\haze(\Zmin)} - \widetilde{G}(\Zmin) \right), \]
hence
\[ L^2 - G(\Zmin)(2\pi)^2 > (2\pi)^2\left(\frac{1}{\haze(\Zmin)} - \widetilde{G}(\Zmin) - G(\Zmin)\right).
\]
The right hand side is positive when $\Zmin = 0.6622$. 
Using the fact that the functions $\haze$, $G$, and $\widetilde{G}$ are all strictly decreasing for $Z>\Zmin$ (\refrem{Haze} and \reflem{hkp1079}), it follows that the left hand side is also positive for $Z \geq \Zmin \geq 0.6622$. 

Thus all terms  in \refeqn{li-bound} are positive, and we can take the reciprocal of each term. Solving for $\ell$ completes the proof.
\end{proof}

We will apply \reflem{MagidLengthGeneral} to obtain relations between $L$ and $\ell$. The following is the simplest and most exportable version of the lemma, using $\Zmin = 0.6624$.

\begin{corollary}\label{Cor:MagidLength}
Suppose that $M$ is a complete, finite volume hyperbolic $3$--manifold and $\Sigma \subset M$ is a geodesic link such that one of the following hypotheses holds.
\begin{enumerate}
\item
In the complete structure on $M - \Sigma$, the total normalized length of the meridians of $\Sigma$ is $L \geq 7.823$.
\item
In the complete structure on $M$, each component of $\Sigma$ has length at most $0.0996$ and the total length of $\Sigma$ is $\ell \leq 0.1396$.
\end{enumerate}
Then 
\[
\frac{2\pi}{L^2 + 16.17} 
  < \ell
  < \frac{2\pi}{L^2 - 28.78}.
\]  
\end{corollary}

\subsection{Boundary terms for general tubes}

The following proposition gives an explicit bound on boundary terms along a general tube about $\Sigma$. This bound will be used in \refsec{ShortGeodesic} to control the change in length of a short non-singular geodesic. Versions of \refprop{BoundaryBound} with stronger hypotheses (see \refthm{BoundaryDelta}) will also be used in the bilipschitz estimates of \refsec{Bilip}.

\begin{proposition}\label{Prop:BoundaryBound}
Let $M$ be a complete, finite volume hyperbolic $3$--manifold, and $\Sigma$ a geodesic link in $M$.
Let $M_t$ be a cone-manifold occurring along a deformation between $M - \Sigma$ and $M$, as in \refthm{ConeDefExists}.

Let $U_\rr(\Sigma)$ be an embedded (not necessarily maximal) multi-tube about the cone locus $\Sigma$. 
Suppose the smallest radius of a tube is $r$, and let $z = \tanh r$. Suppose that the area of each tube boundary is at least $A$.

Suppose that in the complete structure on $M$, each component of $M$ has length at most $0.0996$ and the total length of $\Sigma$ is $\ell \leq \haze(\Zmin)/(2\pi)$, where $\Zmin\geq 0.6622$ and $\haze$ is the function of \refrem{Haze}. Then
\[
b_\rr(\eta_0, \eta_0)
  \leq  \frac{1}{4Az(3-z^2)}\cdot \left( \frac{\ell}{2 \pi - 4\pi^2 \widetilde{G}(\Zmin)\ell} \right)^2,
\]
where $\widetilde{G}$ is as in \reflem{hkp1079}. In particular, if $\ell \leq 0.075$, then
\[
b_\rr(\eta_0, \eta_0)
  \leq  \frac{1}{4Az(3-z^2)}\cdot \left( \frac{\ell}{2 \pi - 12.355 \ell} \right)^2.
\]
\end{proposition}

\begin{proof}
We compute as follows:
\begin{align*}
b_\rr(\eta_0, \eta_0) & \leq \sum \frac{4(1-z_j^2)}{z_j^2(3-z_j^2)} \cdot \frac{1}{16\alpha^4} \calA_j \qquad\mbox{ by \reflem{hkshape5.4}} \\
& = \sum \frac{1}{4\calA_j}\cdot\frac{1-z_j^2}{z_j^2(3-z_j^2)}\calA_j^2\cdot 16\, \cdot \frac{1}{16\alpha^4} \cdot \frac{\sinh r_j \cosh r_j}{\sinh r_j \cosh r_j} \\
& = \sum \frac{1}{4\area(\bdy U_j)}\cdot \frac{1-z_j^2}{z_j^2(3-z_j^2)}\cdot \frac{\calA_j^2}{\alpha^4}\cdot \frac{z_j}{1-z_j^2} \quad\mbox{ by \refeqn{TorusArea} and \reflem{TanhSinhCosh}} \\
& = \sum \frac{1}{4\area(\bdy U_j)} \cdot \frac{1}{z_j(3-z_j^2)} \cdot v_j^2 \qquad\mbox{ where } v_j = \calA_j/\alpha^2\\
& \leq \sum \frac{1}{4A}\cdot \frac{1}{z(3-z^2)}\cdot v_j^2 \qquad \mbox{using monotonicity of $(z(3-z^2))^{-1}$} \\
& \leq \frac{1}{4A}\cdot\frac{1}{z(3-z^2)} \cdot v^2 \qquad\qquad \mbox{where } v=\sum v_j \\
& = \frac{1}{4A}\cdot\frac{1}{z(3-z^2)}\cdot \left(\frac{1}{u}\right)^2  \qquad \mbox{where } u=1/v.
\end{align*}

Now, observe that under our hypotheses, 
\reflem{MagidLengthGeneral}
ensures a cone deformation between cone angle $\alpha$ and cone angle $2\pi$ for which the $\tanh$ of the maximal tube stays bounded below by
$\Zmin \geq 0.6622$. 
Thus, by \reflem{uBoundsEll},
\[
u \geq \frac{2\pi}{\ell} - \widetilde{G}(\Zmin)\cdot (2\pi)^2
= \frac{2\pi - 4 \pi^2 \, \widetilde{G}(\Zmin) \, \ell}{\ell}.
\]
Note that the right hand side is positive, because $\ell \leq h(\Zmin)/(2\pi)$ and $\Zmin>0.6622$ implies $2\pi/\ell -\widetilde{G}(\Zmin)(2\pi)^2 \geq 28.8>0$.
Now, we may invert the lower bound on $u$ to obtain the desired upper bound on $b_\rr(\eta_0, \eta_0)$.

In the specific case $\ell<0.075 \leq h(0.8477)/(2\pi)$, \reflem{MagidLengthGeneral} ensures the cone deformation exists with $\Zmin > 0.8477$. Substituting that value into $\widetilde{G}(z)$ gives the bound.
\end{proof}


\subsection{Boundary terms along thin tubes}\label{Sec:BoundaryThinTube}
We close this section by establishing certain versions of \refprop{BoundaryBound} in the specific situation where the multi-tube $U_\rr$ is defined by a small injectivity radius. See \refthm{BoundaryDelta} for a detailed statement.

This result will be used to prove bilipschitz estimates; see \refthm{Bilip} and \refthm{BilipBis} in \refsec{Bilip}. 
The conclusion that a multi-tube $U_\rr$ has a certain depth will also prove crucial in controlling Margulis numbers in \refsec{Margulis}. On the other hand, the results of this subsection are not needed  in \refsec{ShortGeodesic}. Thus, a reader who is mainly interested in the application to cosmetic surgeries can skip ahead to \refsec{ShortGeodesic}.

\begin{definition}\label{Def:DeltaThinTube}
Let $U$ be a tube about a component of $\sigma_j \subset \Sigma$. For $\delta> 0$, we say that $U$ is a \emph{$\delta$--thin tube} if $\injrad(x,U) = \delta/2$ for a point $x \in \bdy U$. (Recall \refdef{InjRadTube}.) We emphasize that the term \emph{$\delta$--thin tube} refers only to injectivity radius in $U$, not in all of $M$.
\end{definition}

\begin{lemma}\label{Lem:DeltaTubeEmbeds}
Fix $0 < \delta < 0.9623$.
Supppose $M$ is a complete, finite volume hyperbolic $3$--manifold and $\Sigma = \sigma_1 \cup \ldots \cup \sigma_n$ is a geodesic link in $M$. 
Suppose that $\len(\sigma_j) \leq 0.0996$ for every $j$, while the total length of $\Sigma$ is
\begin{equation}\label{Eqn:DeltaEmbedsHypoth}
\ell = \len(\Sigma) \leq \min \left \{ 0.261 \delta, \: \frac{1}{2\pi} \haze \left( \frac{\delta + 0.1604 }{1.1227}  \right) 
\right \},
\end{equation}
where $\haze$ is defined in \refrem{Haze}.
Then $M - \Sigma$ is connected to $M$ via a cone deformation $M_t$, while maintaining a multi-tube of radius $R \geq h^{-1}(2 \pi \ell) \geq 0.7555$.

Fix a cone-manifold $M_t$ in the interior of the deformation. For each component $\sigma_j$, let $r_j(\delta) = r_{\alpha,\lambda,\tau}(\delta)$ be the tube radius of the $\delta$--thin tube about $\sigma_j$ in the metric $g_t$. Set 
$\rr(\delta) = (r_1(\delta) , \ldots, r_n(\delta) ).
$
Then
\begin{enumerate}
\item For every $j$, we have $r_j(\delta)  > 1.001 (\delta/2)$.

\item The multi-tube $U_{\rr(\delta)}$ is embedded in $M_t$. Furthermore, each $\delta$-thin tube of radius $r_j(\delta)$ is properly contained in a component of the maximal tube $U_{\max}$.
\end{enumerate}
\end{lemma}

Conclusion \refitm{DeltaTubeDeep} can be interpreted as follows: $\injrad(x,U)$ is realized by a ball $B = B_{\delta/2}(\widetilde x)$ bumping into another translate of $B$, and furthermore the bumping does not occur along the singular locus. The translate is somewhere else.

See \reffig{DeltaEmbeds} for a graph of the upper bound on $\ell = \len(\Sigma)$. Roughly speaking, the first hypothesis $\ell \leq 0.261 \delta$ corresponds to conclusion \refitm{DeltaTubeDeep}, while the second hypothesis $2\pi\ell \leq \haze  \left( \frac{\delta + 0.1604 }{1.1227}  \right)  $ corresponds to conclusion \refitm{DeltaTubeEmbedded}. Both conclusions also require $\ell \leq 0.1453$, where $0.1453$ is (approximately) the peak of the graph in \reffig{DeltaEmbeds}.

\begin{proof}[Proof of \reflem{DeltaTubeEmbeds}]
We begin by analyzing the two functions whose minimum is the bound on $\ell$ in \refeqn{DeltaEmbedsHypoth}.  Recall from \reflem{HBound}
that $\frac{1}{2\pi} h(r) = \frac{1}{2\pi}\haze(\tanh r)$ has a single critical point for $r > 0$, and that this critical point is a global maximum.  As a consequence, it is an easy exercise to check that the two functions of \refeqn{DeltaEmbedsHypoth} intersect exactly once, at
\[
\delta = \delta_{\rm cut} = 0.556369\ldots
\]
Define
\[
R_{\rm cut} = h^{-1} (2 \pi \cdot 0.261 \delta_{\rm cut} ) = \arctanh \left( \frac{\delta_{\rm cut} + 0.1604 }{1.1227}  \right) = 0.75552\ldots, 
\]
and note that $\tanh R_{\rm cut} > 1/\sqrt{3}=0.57735\ldots$. This helps us analyze the two functions.

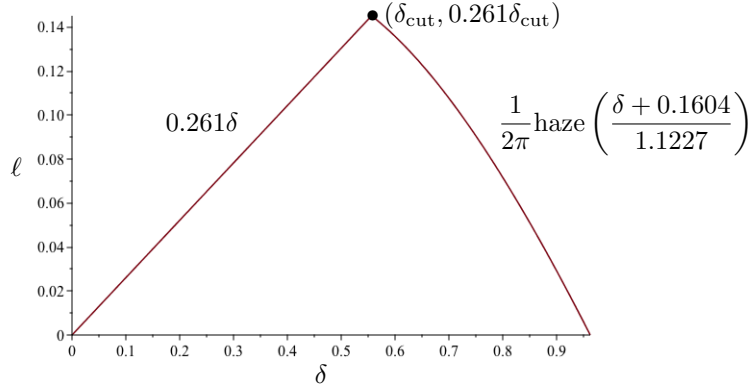
\begin{figure}
\begin{overpic}[width=3in]{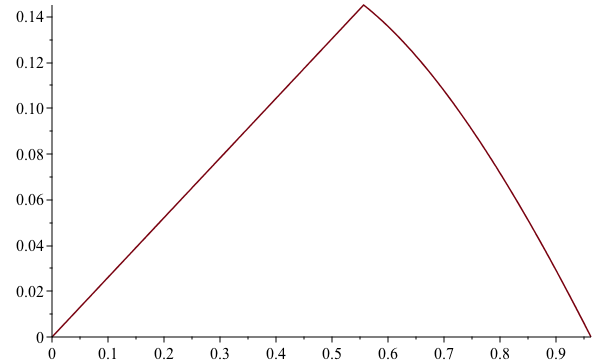}
\put(25,40){$0.261 \delta $}
\put(83,40){$ \displaystyle { \frac{1}{2\pi} \haze \left( \frac{\delta + 0.1604 }{1.1227}  \right)  } $}
\put(51,-4){$\delta$}
\put(-2,32){$\ell$}
\put(60,59){$\bullet \, (\delta_{\rm cut}, 0.261 \delta_{\rm cut} )$}
\end{overpic}
\caption{The function of $\delta$  that provides an upper bound on $\ell$ in Equation~\refeqn{DeltaEmbedsHypoth}.}
\label{Fig:DeltaEmbeds}
\end{figure}

If $0 < \delta < \delta_{\rm cut}$, then the minimum in \refeqn{DeltaEmbedsHypoth} is achieved by $0.261 \delta$, an increasing linear function. See \reffig{DeltaEmbeds}, left.
On the other hand, if $ \delta_{\rm cut} \leq \delta < 0.9623$, then the minimum is achieved by 
$\frac{1}{2\pi} \haze(z)$, where 
\[
z = \frac{\delta + 0.1604}{1.1227} \geq \frac{\delta_{\rm cut} + 0.1604}{1.1227} > \sqrt{\sqrt{5}-2}.
\]
In particular, $z$ is large enough that the function $\frac{1}{2\pi} \haze(z)$ is decreasing in $z$ by \reflem{HBound} and \refrem{Haze},
hence decreasing in $\delta$. See \reffig{DeltaEmbeds}, right. Thus the largest possible upper bound on $\ell$ occurs at $\delta = \delta_{\rm cut}$.

We conclude that for all values $0 < \delta < 0.9623$, Equation~\refeqn{DeltaEmbedsHypoth} requires the visual area of $\Sigma$ to satisfy
\[
\calA = 2\pi \ell \leq 2\pi \cdot 0.261 \delta_{\rm cut} = h(R_{\rm cut}) \leq 2\pi 
\cdot 0.14522.
\]
Furthermore, 
under our hypotheses on $\len(\sigma_j)$, \reflem{Meyerhoff} implies that the maximal tube about $\Sigma$ has radius $R > 0.531$. 
It follows that \refthm{ConeDefExistsRBounds} guarantees a cone deformation from 
$M$ to $M-\Sigma$, maintaining a multi-tube $U_{\max}$ of radius
\[
R \geq \Rmin = h^{-1}(2 \pi \ell) \geq R_{\rm cut}.
\]
where the last inequality uses the decreasing property of $h^{-1}$. 

We are now ready to prove conclusion \refitm{DeltaTubeEmbedded}. Consider a component $U_j \subset U_{\max}$, of radius $R_j$, and recall that
 $R_j \geq \Rmin \geq R_{\rm cut} > \arctanh(1/\sqrt{3})$.  Define $\Zmin = \tanh \Rmin$, and $Z_{\rm cut} = \tanh R_{\rm cut}$, as usual. \refthm{MaxTubeInjectivity} says that for every  $x \in \bdy U_j$, we have
\begin{equation*}
2 \, \injrad(x, U_j) > 1.1227 \Zmin - 0.1604.
\end{equation*}
If $0 < \delta \leq \delta_{\rm cut}$, we have
\[
1.1227 \Zmin - 0.1604 \geq 1.1227 \, Z_{\rm cut} - 0.1604 = \delta_{\rm cut} \geq \delta,
\]
Meanwhile, if $\delta_{\rm cut} \leq \delta < 0.9623$,
Equation~\refeqn{DeltaEmbedsHypoth} and the decreasing property of $\haze^{-1}$ imply
\[
1.1227 \Zmin - 0.1604 = 1.1227 \, \haze^{-1}(2 \pi \ell) - 0.1604 \geq \delta.
\]
Thus, in either case, we can combine the above equations to conclude
\begin{equation}\label{Eqn:InjecStrictContain}
2 \, \injrad(x, U_j) > 1.1227 \, \Zmin - 0.1604 \geq \delta,
\end{equation}
which implies  $r_j(\delta) < \Rmin \leq R_j$.  
This means each component of $U_{\rr(\delta)}$ is properly contained in the corresponding component of $ U_{\max}$.


To prove conclusion \refitm{DeltaTubeDeep}, let $\calA_j = \alpha \lambda_j$ be the visual area of $\sigma_j$, as in \refdef{VisualArea}. Recall that by \reflem{LenMonotonicity}, $\calA = \sum \calA_j$ increases as the cone angle increases. At the complete structure, we have
\[
\calA = 2 \pi \ell \leq 2 \pi \cdot 0.261 \delta \leq 1.64 \, \delta.
\]
Thus, for every intermediate cone-manifold $M_t$, we also have 
 $\calA_j \leq  \calA \leq 1.64 \, \delta$.
Since $\alpha < 2\pi$, \reflem{SingularTubeRad} applies to give
\begin{equation}\label{Eqn:SecondInjecEstimate}
r_j(\delta) \geq \frac{1}{2} \arcsinh \left( \frac{\sqrt{3} \, \delta^2 }{\calA_j} \right) \geq \frac{1}{2} \arcsinh \left( \frac{\sqrt{3} \, \delta ^2} {1.64 \,  \delta}  \right) \geq \frac{1}{2} \arcsinh \left( 1.056 \,  \delta \right).
\end{equation}

Now suppose that $0<\delta\leq\delta_{\rm cut}$.  Consider the secant line for $\arcsinh(1.056\,\delta)$ between $\delta=0$ and $\delta=\delta_{\rm cut}$; this line has slope $\arcsinh(1.056\delta_{\rm cut})/\delta_{\rm cut} > 1.002$. By calculus, $\arcsinh(1.056\,\delta)$ is strictly increasing and concave down for $\delta >0$, hence it lies over its secant line when $0 < \delta \leq \delta_{\rm cut}$. Thus, in this range we have
\[
r_j(\delta) \geq  \frac{1}{2} \arcsinh \left( 1.056 \,  \delta \right) \geq 1.002 \left( \frac{\delta}{2} \right).
\]
Meanwhile, if $\delta_{\rm cut} \leq \delta \leq 0.9623$,
then we still have $\calA \leq 1.64 \, \delta_{\rm cut}$ at the complete structure, hence 
 $\calA_j \leq  \calA \leq 1.64 \, \delta_{\rm cut}$ at every intermediate cone-manifold. Thus we have the following analogue of \refeqn{SecondInjecEstimate}:
\[
r_j(\delta_{\rm cut}) \geq \frac{1}{2} \arcsinh \left( \frac{\sqrt{3} \, \delta_{\rm cut}^2 }{\calA_j} \right)
\geq \frac{1}{2} \arcsinh \left( \frac{\sqrt{3} \, \delta_{\rm cut} ^2} {1.64 \,  \delta_{\rm cut}}  \right)
\geq \frac{1}{2}  \left( 1.002 \,  \delta_{\rm cut} \right).
\]
Thus by \reflem{LowerBoundEasy},
\[
r_j(\delta) \: \geq \:  r_j(\delta_{\rm cut}) + \frac{\delta - \delta_{\rm cut}}{2} \: \geq \:   \frac{1}{2}  \big( 1.002 \,  \delta_{\rm cut} + ( \delta - \delta_{\rm cut}) \big) \: > \: \frac{1}{2}  \left( 1.001 \,  \delta \right),
\] 
where the second inequality is the above lower bound on $r_j(\delta_{\rm cut})$, and the third inequality holds because $\delta$ is less than twice as big as $\delta_{\rm cut}$.
\end{proof}

Now, we can combine \refprop{BoundaryBound} and \reflem{DeltaTubeEmbeds} to control the boundary terms along certain thin tubes.

\begin{theorem}\label{Thm:BoundaryDelta}
Let $M$ be a complete, finite volume hyperbolic $3$--manifold and $\Sigma$ a geodesic link in $M$. Let $\ell$ denote the length of $\Sigma$ in the complete structure on $M$. 

Fix $0 < \delta \leq \delmax \leq 0.938$, and suppose that $\ell \leq \delta^2 B(\delta)$, where $B(\delta)$ is a nondecreasing function of $\delta$, with $B(\delta)\leq 1/17.11$. (In particular, this assumption implies $\ell \leq 0.05143$.)

Fix a cone-manifold $M_t$ in the interior of the cone deformation from $M-\Sigma$ to $M$, with associated cone-metric $g_t$. For each component $\sigma_j$ of $\Sigma$, let $r_j(\delta) = r_{\alpha,\lambda,\tau}(\delta)$ be the tube radius of the $\delta$--thin tube about $\sigma_j$ in the metric $g_t$. 
Define 
$\rr_- = \rr_-(\delta) = (r_1(\delta) - \delta/2, \ldots, r_n(\delta)-\delta/2).
$
Then 
\begin{enumerate}
\item\label{Itm:DeltaTubeDeep} For all $j=1, \dots, n$,
\[
r_j(\delta) - \delta/2 \geq \half\arcsinh\left(\frac{\sqrt{3}}{2\pi B(\delmax)}\right) - \frac{\delmax}{2} \geq \arctanh(1/\sqrt{3}).
\]

\item\label{Itm:DeltaTubeEmbedded} The multi-tube $U_{\rr_-}$ is embedded in $M_t$. 

\item\label{Itm:BoundaryTerm} The boundary term along the tube $U_{\rr_-}$ satisfies
\[
b_{\rr_-}(\eta,\eta) \leq  \left( \frac{\ell}{7.935 \, \delta} \right)^2 .
\]
\end{enumerate}
\end{theorem}

\begin{proof}
We start by proving \refitm{DeltaTubeEmbedded}. We will do this by applying \reflem{DeltaTubeEmbeds} to a multi-tube  $U_{\rr(\gamma)}$ for a certain value $\gamma < \delta$.
Specifically, define $\gamma > 0$ so that
\[
\cosh \gamma - 1 = \frac{\cosh \delta - 1}{\cosh^2 (\delta/2)}, \quad \mbox{ i.e.\ } \quad \gamma = \arccosh\left(\frac{\cosh \delta - 1}{\cosh^2 (\delta/2)} + 1\right).
\]
Observe that $\gamma$ is strictly increasing in $\delta$ on $(0, 0.938)$, and that $\gamma<\delta$ in this range. Moreover, as $\delta$ approaches $0$, $\gamma$ approaches $0$, and as $\delta$ approaches $0.938$, $\gamma$ approaches $\gamma_{\max} = 0.84904\ldots$.

Consider a model solid torus $N_j = N_{\alpha,\lambda,\tau}$ such that a neighborhood of $\sigma_j$ is modeled on $N_j$ (compare \refdef{ConeManifold}). By \refthm{EffectiveDistLog3}, the radii of $\gamma$--thin and $\delta$--thin tori in $N_j$ satisfy
\[
r_j(\delta) - r_j(\gamma) \leq \arccosh \sqrt{ \frac{\cosh \delta -1}{\cosh \gamma - 1} } = \frac{\delta}{2},
\]
where the equality holds by the definition of $\gamma$.

With an eye toward \reflem{DeltaTubeEmbeds}, we claim that our hypotheses imply
\begin{equation}\label{Eqn:DeltaEmbedsHypoth_new}
\ell \leq \frac{\delta^2 }{17.11} \leq  \min \left \{ 0.261 \gamma, \: \frac{1}{2\pi} \haze \left( \frac{\gamma + 0.1604 }{1.1227}  \right) 
\right \}.
\end{equation}
The first inequality holds by hypothesis.
For the next inequality $\delta^2/17.11 \leq 0.261 \gamma$, consider the function $g(\delta)=(0.261\cdot 17.11)\gamma - \delta^2$.
One can show by calculus\footnote{The second derivative of $\gamma$ is negative, so the same is true for the second derivative of $g$. Thus the minimum of $g'$ occurs at $\delta_{\max}$, and this minimum is positive.}  that $g$ is strictly increasing for $0<\delta<\delta_{\max}$. The minimum value of $g$ thus occurs as $\delta\to 0$, and we have $g(\delta)>g(0)=0$ on this domain. 

As in the proof of \reflem{DeltaTubeEmbeds}, we need to verify the remaining inequality only for $\delta_{\rm cut} < \gamma < \gamma_{\max}$. On this domain, $\delta^2$ is strictly increasing whereas $\haze(\cdot)$ is strictly decreasing, by \reflem{HBound}. Thus it suffices to plug in the maximal $\delta$-value $0.938$, which corresponds to $\gamma_{\max} = 0.84904 \ldots$. Plugging in these values of $\delta$ and $\gamma$ gives
 \[
\frac{(0.938)^2 } {17.11} = 0.05142\ldots < 0.05147\ldots =  \frac{1}{2\pi} \haze \left( \frac{\gamma_{\max} + 0.1604 }{1.1227}  \right),
\]
hence \refeqn{DeltaEmbedsHypoth_new} holds for all $\delta \leq \delmax \leq 0.938$.
Thus, by \reflem{DeltaTubeEmbeds}, the multi-tube $U_{\rr(\gamma)}$ of radius $\rr(\gamma)$ is embedded. Since $r_j(\delta) - \delta/2 \leq r_j(\gamma)$ for every $j$, it follows that the multi-tube $U_{\rr_-(\delta)}$ is embedded as well.


The proof of condition \refitm{DeltaTubeDeep} is similar to the corresponding argument in \reflem{DeltaTubeEmbeds}. Let $\calA_j = \alpha \lambda_j$ be the visual area of $\sigma_j$. By \reflem{LenMonotonicity}, $\calA = \sum \calA_j$ increases as the cone angle increases.
Since $\calA = 2\pi\ell \leq 2\pi \delta^2 B(\delta)$
at the complete structure on $M$, we also have $\calA_j \leq  2\pi \delta^2 B(\delta)$
at every intermediate cone-manifold $M_t$.
Since $\alpha < 2\pi$, \reflem{SingularTubeRad} applies to give
\begin{equation}\label{Eqn:RjRadius}
r_j(\delta) \geq \frac{1}{2} \arcsinh \left( \frac{\sqrt{3} \, \delta^2 }{\calA_j} \right) \geq \frac{1}{2} \arcsinh \left( \frac{\sqrt{3} }{2\pi B(\delta)} \right)
\geq 1.1276 \ldots.
\end{equation}
Since $\delta \leq \delmax$ and $B(\delta)^{-1} \geq 17.11$, we conclude that
\[
r_j(\delta) - \delta/2 \geq \half\arcsinh \left( \frac{\sqrt{3}}{2\pi B(\delta)} \right) - \frac{\delmax}{2}
> 0.65847 \ldots = \arctanh(1/\sqrt{3}),
\]
proving  \refitm{DeltaTubeDeep}.


It remains to bound $b_{\rr_-}(\eta, \eta)$, establishing conclusion \refitm{BoundaryTerm}. Since the smallest coordinate of $\rr_-$ is larger than $\arctanh(1/\sqrt{3}$), \reflem{BoundaryTermSigns} and \refeqn{Weitzenbock} tell us that $b_{\rr_-}(\eta,\eta) \leq b_{\rr_-}(\eta_0, \eta_0)$.

We will bound $b_{\rr_-}(\eta_0, \eta_0)$ using \refprop{BoundaryBound}. 
By \refeqn{TorusArea} and \reflem{EpsilonArea}, the torus $T^\delta = T_{r_j(\delta)}$ has area
\[
\area (T_{r_j(\delta)}) = \frac{\calA_j}{2} \sinh (2 r_j(\delta)) \geq \sqrt{3}\delta^2/2. 
\]

To apply  \refprop{BoundaryBound}, we need a lower bound on the area of each boundary torus of $U_{\rr_-}$, where the $j$-th torus has radius $r_j(\delta) - \delta/2$. This can be computed using \reflem{EpsilonArea} and \reflem{SinhDiff}.
Note that by hypothesis, $\delta \leq \delmax \leq 0.938$. By \refeqn{RjRadius}, we have
\[ \tanh(2 r_j(\delta)) \geq \tanh\left( \arcsinh\left(\frac{\sqrt{3}}{2\pi B(\delmax)}\right) \right) = \frac{\sqrt{3}}{\sqrt{3+4\pi^2 B(\delmax)^2}} =: z_{\min}.
\]
Observe that $\tanh(\delta_{\max}) = 0.7343\ldots$ and $z_{\min}\geq 0.9782\ldots$, hence $\tanh(\delta_{\max})\leq z_{\min}\leq \tanh(2 r_j(\delta))$, and \reflem{SinhDiff} applies with $r=2r_j(\delta)$, $s=\delta$, and $z_{\min}$ as above.
Therefore,
\begin{align}
\area (T_{r_j(\delta)- \delta/2}) 
& = \frac{\calA_j}{2} \sinh (2r_j(\delta)- \delta)
& \mbox{ by \refeqn{TorusArea}}  \nonumber\\
& \geq \frac{\calA_j}{2} \sinh (2r_j(\delta)) \left( \cosh (\delta_{\max}) - \frac{\sinh(\delta_{\max})}{z_{\min}} \right)
& \mbox{ by \reflem{SinhDiff}}  \nonumber\\
& = \area (T_{r_j(\delta)})  \left( \cosh (\delmax) - \frac{\sinh(\delmax)}{\zmin} \right)
& \mbox{ by \refeqn{TorusArea}}  \nonumber\\
& \omit\rlap{$\displaystyle{\geq \delta^2\,\frac{\sqrt{3}}{2} \left( \cosh (\delmax) - \frac{\sinh(\delmax)\sqrt{3+4\pi^2B(\delmax)^2}}{\sqrt{3}} \right)}$}
\nonumber\\
& =: \delta^2 \, A_{\mathrm{Bd}}(\delmax, B(\delmax)) . & 
    \label{Eqn:VisualAreaDelta}
\end{align}

We may now finish the proof using \refprop{BoundaryBound}. Recall that the function $(z(3-z^2))^{-1}$ is monotonically decreasing on $(0,1)$, 
where $z = \tanh(r_j(\delta)-\delta/2)$ in our setting. By conclusion \refitm{DeltaTubeDeep}, we obtain
\begin{equation}\label{Eqn:ZBoundaryDelta}
z\geq \tanh \left( \half\arcsinh\left(\frac{\sqrt{3}}{2\pi B(\delmax)} \right) - \frac{\delmax}{2}\right) =: \zbd(\delmax, B(\delmax) ) \geq \frac{1}{\sqrt{3}}. \end{equation}
Recall that 
$\ell \leq \delmax^2 B(\delmax) \leq \delmax^2/17.11 < 0.05143$. 
Thus we also have $\ell \leq \haze(\Zmin)/(2\pi)$ for 
$\Zmin = \haze^{-1}(2\pi\delmax^2B(\delmax)) \geq 0.8992$.
By \refprop{BoundaryBound}, 
\begin{align*}
b_{\rr_-}(\eta,\eta) 
\leq\, & b_{\rr_-}(\eta_0, \eta_0)  \\
\leq\, &  \frac{1}{4A} \cdot \frac{1}{z(3-z^2)} \cdot \left( \frac{\ell}{2\pi - 4\pi^2\widetilde{G}(\Zmin) \, \ell} \right)^2  \\
\leq\, &  \frac{1}{4\,\delta^2 A_{\mathrm{Bd}}} \cdot \frac{1}{\zbd(3-\zbd^2)} \cdot \left( \frac{\ell}{2\pi - 4\pi^2\widetilde{G}(\Zmin) \, \delmax^2B(\delmax)} \right)^2.
\end{align*}

In the particular case that $\delmax=0.938$ and $B(\delta)=1/17.11$, we obtain
\begin{align*}
b_{\rr_-}(\eta,\eta) 
&\leq  \frac{1}{4 \, \delta^2 \cdot 0.3181} \cdot  \frac{1}{(1/\sqrt{3})(3-1/3)} \cdot \left( \frac{\ell}{2\pi - 4\pi^2\widetilde{G}(0.8992)\cdot(0.05143)} \right)^2 \\
& \leq  \left( \frac{\ell}{7.935 \, \delta} \right)^2 . \qedhere
\end{align*}
\end{proof}

\begin{remark}\label{Rem:BoundaryDeltaMessy}
In the course of proving \refthm{BoundaryDelta}, we actually proved something more general. We showed that
\[
b_{\rr_-}(\eta,\eta) \leq \frac{1}{4\cdot \delta^2 \, A_{\mathrm{Bd}}} \cdot
\frac{1}{\zbd\left(3-\zbd^2\right)}
\cdot \left( \frac{\ell}{2\pi - 4\pi^2\widetilde{G}\left(\Zmin \right)
\delmax^2 B(\delmax)}\right)^2,
\]
where $A_{\mathrm{Bd}}(\delmax,B(\delmax))$ is defined in \refeqn{VisualAreaDelta}, and $\zbd(\delmax,B(\delmax))$ is defined in \refeqn{ZBoundaryDelta}, 
and $\widetilde{G}$ is the function of \reflem{hkp1079} with $\Zmin = \haze^{-1}(2\pi\delmax^2B(\delmax))$.
\end{remark}

Our applications often require much stronger bounds on $\delta$ and $B(\delta)$ than the maximum values allowed in \refthm{BoundaryDelta}. As a consequence, we can obtain better bounds on $b_{\rr_-}(\eta,\eta)$. For example, the following two stronger estimates will be useful in \refsec{Margulis}.

\begin{proposition}\label{Prop:BoundaryMediumDelta}
Let $M$ be a complete, finite volume hyperbolic $3$--manifold and $\Sigma$ a geodesic link in $M$, with $\ell$ the length of $\Sigma$ in the complete structure on $M$.
Fix $0 < \delta \leq 0.106$. Suppose $\ell\leq \delta^{5/2}/17.49$. 

Fix a cone-manifold $M_t$ in the interior of the cone deformation from $M-\Sigma$ to $M$. Define $\rr_- = (r_j(\delta) - \delta/2)$, and construct the multi-tube $U_{\rr_-}$ about $\Sigma$
exactly as in \refthm{BoundaryDelta}. 
Then 
\begin{enumerate}

\item\label{Itm:DeltaTubeDeepBis} $r_j(\delta) - \delta/2 \geq \arctanh(0.9277)$ for all $j$.

\item\label{Itm:DeltaTubeEmbeddedBis} 
The multi-tube $U_{\rr_-}$ is embedded in $M_t$. 

\item\label{Itm:BoundaryTermBis} The boundary term along the tube $U_{\rr_-}$ satisfies
\(
b_{\rr_-}(\eta,\eta) \leq  \left( \ell/ (15.616 \, \delta) \right)^2  .
\)
\end{enumerate}
\end{proposition}

\begin{proof}
This follows immediately from \refthm{BoundaryDelta} and \refrem{BoundaryDeltaMessy}, letting $\delmax = 0.106$ and $B(\delta) = \sqrt{\delta}/17.49$.
\end{proof}

\begin{proposition}\label{Prop:BoundaryTinyDelta}
Let $M$ be a complete, finite volume hyperbolic $3$--manifold and $\Sigma$ a geodesic link in $M$, with $\ell$ the length of $\Sigma$ in the complete structure on $M$.
Fix $0 < \delta \leq 0.012$. Suppose $\ell \leq  \delta^{5/2} / 16.62$.

Fix a cone-manifold $M_t$ in the interior of the cone deformation from $M-\Sigma$ to $M$. Define $\rr_- = (r_j(\delta) - \delta/2)$, and construct the multi-tube $U_{\rr_-}$ about $\Sigma$
exactly as in \refthm{BoundaryDelta}. 
Then 
\begin{enumerate}

\item
$r_j(\delta) - \delta/2 \geq \arctanh(0.9760)$ for all $j$.

\item
The multi-tube $U_{\rr_-}$ is embedded in $M_t$. 

\item
The boundary term along the tube $U_{\rr_-}$ satisfies
\(
b_{\rr_-}(\eta,\eta) \leq  \left( \ell/ (16.432 \, \delta) \right)^2  .
\)
\end{enumerate}
\end{proposition}

\begin{proof}
This follows immediately from \refthm{BoundaryDelta} and \refrem{BoundaryDeltaMessy}, letting $\delmax=0.012$ and $B(\delta) = \sqrt{\delta}/16.62$. 
\end{proof}

\section{Short geodesics in a cone-manifold}\label{Sec:ShortGeodesic}

The primary goal of this section is to control the complex length of a short geodesic during a cone deformation. Ineffective control of this type was previously shown by Bromberg \cite[Theorem 1.4]{bromberg:conemflds}. Following the theme of this paper, we combine some ideas in Bromberg's argument (specifically, \cite[Proposition 4.3]{bromberg:conemflds}) with our estimates from earlier sections in order to obtain an effective estimate on the change in length under explicit hypotheses. Our results in this vein are incorporated in \refthm{ShortStaysShort} (which provides control under hypotheses in the filled manifold) and \refthm{ShortStaysShortUpward} (which provides control under hypotheses in the cusped manifold).

One particular consequence of \refthm{ShortStaysShort} is that for (explicitly quantified) long Dehn fillings of a cusped manifold $N$, the union of cores of the filling solid tori is shorter than any other closed geodesic in the filled manifold $M = N(\mathbf s)$. See \refthm{UniqueShortest}. This tuple of shortest closed geodesics must be permuted by any isometry of $M$, providing an effective upper bound on the length of cosmetic fillings of a cusped $N$. As a consequence, we can prove the cosmetic surgery results that were stated in the introduction.

\subsection{Hyperbolic distance between lengths}

The following notation will be valid throughout the section. As above, we have a one-parameter family of cone manifolds denoted $(M, \Sigma, g_t)$ or $M_t$ for short. Let  $\gamma$ be a closed geodesic disjoint from $\Sigma$. We denote the complex length of $\gamma$ in the cone-metric $g_t$ by 
\[
\calL_t (\gamma) = \len_t(\gamma) + i \tau_t(\gamma),
\]
where $\len_t$ is the real length and $\tau_t$ is the rotational component of $\gamma$. 
When the choice of metric $g_t$ is clear from context, we may drop the subscript $t$. All derivatives of $\calL$ are presumed to be with respect to the cone deformation parameter $t$.

Since $\len_t(\gamma) > 0$, the ``rotated'' complex length $i \calL(\gamma)$ is an element of the upper half-plane $\HH^2$, which we identify with the hyperbolic plane. As we will see in \reflem{ComplexLengthDeriv}, it is natural to control the change in complex length using the hyperbolic metric on $\HH^2$.

\begin{definition}\label{Def:HypDistanceLength}
Given complex numbers $v$ and $w$ with positive real part, define the \emph{hyperbolic distance}
\[
\dhyp(v,w) = d_{\HH^2} (iv, \, iw).
\]
\end{definition}

The distance $\dhyp$ has the following interpretation.
Consider a closed geodesic $\gamma$ lying at the core curve of a (non-singular) model solid torus $N =N_{2\pi,\lambda,\tau}= \HH^3 / \langle \varphi \rangle$. The cyclic group $\langle \varphi \rangle$ has two fixed points $p_+, p_- \in \bdy \HH^3$, and acts by conformal covering transformations on $S^2 - \{p_\pm\}$. The quotient torus $(S^2 - \{p_\pm\})/\langle \varphi \rangle$ inherits a conformal structure, which can be viewed as the \emph{conformal boundary at infinity}, denoted $\bdy_\infty N$. In the Teichm\"uller space $\mathcal{T}(T^2)$ of conformal structures on a torus, the conformal boundary $\bdy_\infty N$ is the limit of the conformal structures on equidistant tori $T_r \subset N$.

The Teichm\"uller metric $d_{\mathcal T}$ on $\mathcal{T}(T^2)$ is isometric to $\HH^2$. Thus, given a pair of closed geodesics $\gamma$ and $\gamma'$, with complex lengths $\calL(\gamma) = \lambda+i\tau$ and $\calL(\gamma') = \lambda'+i\tau'$, we have
\[
d_{\mathcal T} (\bdy_\infty N_{2\pi,\lambda,\tau}, \, \bdy_\infty N_{2\pi,\lambda',\tau'}) = d_{\HH^2} (i\lambda - \tau, \, i\lambda'-\tau') = \dhyp(\calL(\gamma),\calL(\gamma')).
\]
See Minsky \cite[Section 6.2]{minsky:punctured-tori}, where this perspective is fleshed out further.

\begin{definition}
\label{Def:FDefine}
For $z \in (0,1)$ and $\ell \in (0, 0.5085)$, define a function
\begin{equation}\label{Eqn:FDefine}
F(z, \ell) =  \frac{(1+ z^2)}{  z^3 (3-z^2)} \cdot \frac{\ell}{10.667 - 20.977 \ell} \, .
\end{equation}
Note that $F$ is positive everywhere on its domain, decreasing in $z$, and increasing in $\ell$.
\end{definition}

The following is an effective version of a result of Bromberg~\cite[Proposition~4.3]{bromberg:conemflds}. Our proof follows Bromberg's line of argument while inserting the explicit estimates of \refsec{BoundaryBound}.

\begin{lemma}\label{Lem:ComplexLengthDeriv}
Suppose $(M, \Sigma, g_t)$ is a cone deformation from $M - \Sigma$ to $M$, param\-etrized by $t\in [0,(2\pi)^2]$. Let $\gamma \subset M$ be a simple closed curve disjoint from $\Sigma$, and let $\Sigma^+ = \Sigma \cup \gamma$. Let $[a,b]$ be a sub-interval of $[0,(2\pi)^2]$. Suppose that the following hold:
\begin{enumerate}
\item In the complete structure on $M$,  the total length of $\Sigma$ is $\ell \leq 0.075$. 
\item For $t \in [a,b]$, the curve $\gamma$ is a geodesic in the cone metric $g_t$.
\item For all $t \in [a,b]$, there is an embedded maximal multi-tube $U_{\max}(\Sigma^+)$ in the $g_t$ metric, such that all constituent tubes have radius at least $\Rmin$, where $\Zminhat = \tanh \Rmin \geq 1/\sqrt{3}$.
\end{enumerate}
Then, for $t \in [a,b]$, the time derivative of the complex length $\calL_t(\gamma)$ satisfies
\[
 \frac{| \calL_t'(\gamma) |}{ \len_t(\gamma) }  \leq F(\Zminhat, \ell)  ,
\]
where $F$ is the function of \refdef{FDefine}. Consequently,
\begin{equation}\label{Eqn:KenPropComplex}
\dhyp(\calL_a(\gamma), \,  \calL_b(\gamma)) \leq |b-a| \,  F(\Zminhat, \ell).
\end{equation}
\end{lemma}

\begin{proof}
Let $U_\gamma$ be the tube about $\gamma$ in the maximal multi-tube $U_{\max}(\Sigma^+)$. 
By \refthm{LocalConeDeformation}, there is a local cone deformation on $M$ that treats $\Sigma^+$ as its singular locus but does not change the cone angle on $\gamma$. By the rigidity statement in \refthm{LocalConeDeformation}, the cone metric $g_t^+$ on $(M, \Sigma^+)$ is entirely determined by the angles on $\Sigma$, hence it is isometric to the cone metric $g_t$ in the statement of the lemma. 
For the rest of the proof, we will not distinguish between $g_t$ and $g_t^+$.

As in \refsec{Background}, we may parametrize the infinitesimal deformation in $U_{\gamma}$ by cylindrical coordinates. By equation~\refeqn{omega}, we find that the infinitesimal cone deformation in $U_\gamma$ is given by
\[
\omega = s(\gamma) \,\omega_m + (x+iy)\omega_\ell + \omega_c,
\]
where $\omega_m = \eta_m+i*D\eta_m$ and $\omega_\ell = \eta_\ell + i*D\eta_\ell$ are standard harmonic forms, with $\omega_m$ giving infinitesimal change in cone angle, and $\omega_\ell$ giving infinitesimal change in holonomy of the boundary of the tube $U_{\gamma}$ but leaving the cone angle unchanged. The harmonic form $\omega_c$ is a correction term, with real part $\eta_c$. The terms $s$, $x$, and $y$ are real-valued functions of $t$. Moreover, recall from \refthm{LocalConeDeformation} that $s_{n+1} = s(\gamma)$ determines the local change in cone angle at time $t$; since the cone angle about $\gamma$ remains unchanged throughout the deformation, $s(\gamma)=0$. Thus, by \refeqn{ComplexLengthDeriv}, the function $x+iy$ can be calculated to be:
\[
x+iy = \frac{\calL'(\gamma)}{2\,\len_t(\gamma)},
\]
where $\calL_t(\gamma)$ is the complex length of $\gamma$, $\calL'(\gamma)$ is its time derivative, and $\len_t(\gamma)$ is the real length of $\gamma$ at time $t$. 

As in \refsec{BoundaryBound}, we 
integrate over the submanifold $U_\gamma$. 
Recall the definition of boundary terms from \refeqn{BoundaryTermDef}. By \refeqn{Weitzenbock}, we have
\begin{equation}\label{Eqn:WeitzenbockU}
\int_{U_{\gamma}} \|\omega\|^2\,dV = 
b_{U_\gamma}(\eta,\eta) = b_{U_\gamma}(\eta_0, \eta_0) + b_{U_\gamma}(\eta_c, \eta_c).
\end{equation}
Here, $\eta$ is the real part of $\omega$ and $\eta_0$ is the real part of $(x+iy) \omega_\ell$. Meanwhile,  \cite[Lemma~2.6]{hk:univ} implies that $b_{U_\gamma}(\eta_c, \eta_c) \geq 0$. (This conclusion is the reverse of \reflem{BoundaryTermSigns} because $\bdy U_\gamma$ is oriented by the inward normal, pointing toward $\gamma$.) Therefore, \refeqn{WeitzenbockU} implies
\begin{equation}\label{Eqn:BBoundU}
\int_{U_{\gamma}} \|\omega\|^2\,dV \geq 
 b_{U_\gamma}(\eta_0, \eta_0).
 \end{equation}

Since $s(\gamma)=0$, the formulas in \cite[page 382]{hk:univ} imply
\begin{equation}\label{Eqn:BBoundEtaU}
b_{U_\gamma}(\eta_0, \eta_0) = |x+iy|^2 b_{U_\gamma}(\eta_\ell, \eta_\ell) = \left(\frac{|\calL'(\gamma)|}{2\,\len(\gamma)}\right)^2 b_{U_\gamma}(\eta_\ell,\eta_\ell).
\end{equation}
An explicit formula for $b_{U_\gamma}(\eta_\ell, \eta_\ell)$ was computed in \cite[equation~(13)]{hk:univ}:
\begin{align}
b_{U_\gamma}(\eta_\ell, \eta_\ell) &= \frac{\sinh R_{\gamma}}{\cosh R_{\gamma}}\left( 2 + \frac{1}{\cosh^2 R_{\gamma}}\right) \area(\partial U_{\gamma}) \label{Eqn:OppSign} \\
&= Z_\gamma (3-Z_\gamma^2)\area(\partial U_{\gamma}) \nonumber \\
&\geq \Zmin (3-\Zmin^2)\area(\partial U_{\gamma}). \label{Eqn:BEtaEllU}
\end{align}
Here, $R_\gamma$ is the tube radius of $U_\gamma$, and $Z_\gamma = \tanh R_\gamma$ as usual.
Note that \refeqn{OppSign} differs from the formula in \cite{hk:univ} by a negative sign, again because $\bdy U_\gamma$ is oriented inward.
Thus, putting together equations  \eqref{Eqn:BBoundU}, \eqref{Eqn:BBoundEtaU}, and \eqref{Eqn:BEtaEllU}, we obtain
\begin{equation}\label{Eqn:OmegaBoundU}
 \Zmin(3-\Zmin^2) \area(\partial U_{\gamma})\cdot \left( \frac{|\calL'(\gamma)|}{2 \len(\gamma)}\right)^2
\leq b_{U_\gamma}(\eta_0, \eta_0)
\leq \int_{U_{\gamma}} \|\omega\|^2\,dV .
\end{equation}

Next, we will bound $\int \|\omega\|^2\,dV$ using \refprop{BoundaryBound}. Let $U_1, \ldots, U_n$ be the components of $U_{\max}$ whose cores are the geodesics of $\Sigma$. Let $\mathbf{R} = (R_1, \dots, R_n)$ be the vector of radii of these tubes. Since $U_{\max}(\Sigma^+) = U_\mathbf{R}(\Sigma) \cup U_\gamma$ is an embedded multi-tube, we know that $U_\mathbf{R}$ is also an embedded multi-tube, and $U_\gamma$ is embedded in $M - U_\mathbf{R}$.

Since $\tanh R_i \geq \Zmin \geq 1/\sqrt{3}$ by hypothesis, we have the following estimate:
\begin{align}
\int_{U_{\gamma}} \|\omega\|^2\,dV & \leq \int_{M - U_\mathbf{R}} \|\omega\|^2\,dV \nonumber \\
& = b_{\mathbf{R}}(\eta, \eta) \nonumber \\
& \leq b_{\mathbf{R}}(\eta_0, \eta_0) \quad \nonumber \\
& \leq \frac{1}{4\,A\,\Zmin(3-\Zmin^2)}\left(\frac{\ell}{2\pi - 12.355\, \ell}\right)^2 . \label{Eqn:OmegaBoundRell} 
\end{align}
Here, the first inequality uses the set containment $U_\gamma \subset M - U_\mathbf{R}$, the equality uses \refeqn{Weitzenbock}, the next inequality uses 
\reflem{BoundaryTermSigns}, and the final inequality uses \refprop{BoundaryBound}.
In \refeqn{OmegaBoundRell},
recall that 
$\ell = \ell(\Sigma)$ is the sum of the lengths of all components of $\Sigma$ in the non-singular metric on $M$.
Meanwhile, $A$ is any lower bound on the area of each torus $\bdy U_i$. This area can be estimated using \refthm{MaxTubeArea}:
\begin{equation}\label{Eqn:BBoundTorusArea}
\area(\bdy U_i) \geq A := 1.69785 \,  \frac{\sinh^2 \Rmin}{\cosh (2 \Rmin)} = 1.69785 \, \frac{\Zmin^2}{1+\Zmin^2} \, .
\end{equation}
Furthermore, \refthm{MaxTubeArea} also implies that $\bdy U_\gamma$ satisfies the same lower bound.

Combining \refeqn{OmegaBoundU}, \refeqn{OmegaBoundRell}, and \refeqn{BBoundTorusArea} gives
\begin{align*}
 \left( \frac{|\calL'(\gamma)|}{2\len(\gamma)}\right)^2
 & \leq \frac{1}{4\,A^2\,\Zmin^2 \big( 3-\Zmin^2 \big)^2}\left(\frac{\ell}{2\pi - 12.355\, \ell}\right)^2 \\
 & = \frac{(1 + \Zmin^2)^2}{2^2\,\Zmin^6 \big( 3-\Zmin^2 \big)^2}\left(\frac{\ell}{ 1.69785 \, ( 2\pi - 12.355\, \ell ) }\right)^2 ,
 \end{align*}
which simplifies to the desired bound on $| \calL_t'(\gamma) | / \len_t(\gamma)$.

It remains to prove \refeqn{KenPropComplex}. To that end, we offer the following interpretation.
The one-parameter family $i \calL_t(\gamma)$ is a curve in $\HH^2$. The speed with which this curve travels through $\HH^2$ (in the hyperbolic metric) is precisely  $| \calL_t'(\gamma) | / \len_t(\gamma) $. Since this speed is bounded by $F(\Zmin, \ell)$, integrating from $a$ to $b$ shows that
the hyperbolic distance between $\calL_a(\gamma)$ and  $\calL_b(\gamma)$ is at most $ |b-a| F(\Zmin, \ell) $.
\end{proof}

\begin{lemma}\label{Lem:KenPropH2}
Let $\calL(\gamma) = \len(\gamma) + i \tau(\gamma)$ and $\calL(\delta) = \len(\delta)+i\tau(\delta)$ be complex lengths of geodesics satisfying $ \dhyp(\calL(\gamma), \,  \calL(\delta)) \leq K$ for some $K > 0$.
Then we have the following control on the real and imaginary parts of length:
\begin{equation}\label{Eqn:KenPropLen}
e^{- K } \: \leq \: \frac{ \len(\gamma) }{ \len(\delta)} \: \leq \: e^K ,
\end{equation}
\begin{equation}\label{Eqn:KenPropTheta}
|\tau(\gamma) - \tau(\delta)| \: \leq \: \sinh ( K  )  \cdot \min \{ \len(\gamma), \, \len(\delta) \}.
\end{equation}
%
\end{lemma}

\begin{proof}
Let $B_K$ be the closed ball of hyperbolic radius $K$ about $i \calL(\gamma)$. 
The top and bottom of this ball lie at Euclidean height $e^{\pm K} \len(\gamma)$ from $\bdy \HH^2$.
Since the highest possible value of $\len(\delta) = \Im(i \calL(\delta))$ occurs at the highest point of $B_K$, and similarly for the lowest, we conclude that
\[
e^{-K} \len(\gamma) \leq \len(\delta) \leq e^K \len(\gamma),
\]
which is equivalent to \refeqn{KenPropLen}. 

To derive \refeqn{KenPropTheta}, observe that $|\tau(\delta) - \tau(\gamma)|$ is at most the Euclidean radius of $B_K$. Since the highest and lowest points of $B_K$ lie at height $e^{\pm K} \len(\gamma)$, it follows that
\[
|\tau(\delta) - \tau(\gamma)| \leq \frac{e^K + e^{-K}}{2} \len(\gamma) = \sinh K \cdot \len(\gamma).
\]
An identical argument, interchanging the roles of $\gamma$ and $\delta$, yields  $|\tau(\delta) - \tau(\gamma)| \leq  \sinh K \cdot \len(\delta)$ and completes the proof of  \refeqn{KenPropTheta}.
\end{proof}

\subsection{The change in length}

We can now show that the complex length of a short geodesic does not change too much under a cone deformation connecting $M - \Sigma$ to $M$. We handle upward and downward cone deformations in two separate theorems.

To handle downward cone deformations, we need a lemma about the visual area of $\Sigma$.

\begin{lemma}\label{Lem:VisualAreaSigma}
Let $M$ be a complete, finite volume hyperbolic manifold and $\Sigma$ a geodesic link in $M$. Suppose that there is a cone deformation from $M - \Sigma$ to $M$, parametrized by $t= \alpha^2$ and maintaining an embedded tube about $\Sigma$ of radius  at least $R_0$, where $Z_0 = \tanh R_0 \geq  1/ \sqrt{3}$.
Then the visual area of $\Sigma$ in the $g_t$ metric, denoted $\calA_t(\Sigma) $, satisfies
\[
\frac{ \calA_t(\Sigma) }{\calA_{4\pi^2}(\Sigma)} \leq \left( \frac{\sqrt{t}}{2\pi} \right)^{q(Z_0)} 
\quad \text{where} \quad
q(z) =  \left(\frac{3z^2 - 1}{z^2(3-z^2)} + 1 \right) \geq 1.
\]
\end{lemma}

\begin{proof}
By \reflem{LenMonotonicity},  $\calA_t = \calA_t(\Sigma)$ satisfies the differential inequality
\[
\frac{d \calA_t}{dt} \geq \frac{\calA_t}{2t}  \left(\frac{3Z_0^2 - 1}{Z_0^2(3-Z_0^2)} + 1 \right) = \frac{\calA_t}{2t} \, q(Z_0),
\]
and furthermore $q(z) \geq 1$ for $z \geq 1/\sqrt{3}$. The above inequality can be rewritten
\[
\frac{d \calA_t}{\calA_t} \geq \frac{ q(Z_0) }{2} \, \frac{dt}{t} .
\]
Integration over the interval $[a, 4\pi^2]$ gives
\[
\log \left( \frac{  \calA_{4\pi^2}  }{ \calA_a} \right) \geq  \frac{ q(Z_0) }{2} \, \log \left( \frac{ 4 \pi^2}{ a} \right) = \log \left( \left( \frac{ 4 \pi^2}{ a} \right)^{q(Z_0)/2} \right),
\]
which simplifies to the statement in the lemma after substituting $t=a$.
\end{proof}

In fact, we will actually need a bound on the visual area of $\Sigma\cup \gamma$, for $\gamma$ a closed geodesic disjoint from $\Sigma$. The previous lemma and the following lemma together will give us the bound we need.

\begin{lemma}\label{Lem:AreaBoundCapped}
Let $M$ be a complete, finite volume hyperbolic manifold, $\Sigma^+ =\Sigma \cup \gamma$ a geodesic link in $M$, $\ell = \len_{4\pi^2}(\Sigma)$ and $m= \len_{4\pi^2}(\gamma)$ the lengths of $\Sigma$ and $\gamma$ in the complete metric on $M$. 
Suppose $0 \leq \ell \leq 0.0735$ and $0 \leq m \leq 0.0996 - 0.352\cdot \ell$.
Let
\[
Z_0 = \haze^{-1}(2\pi \ell)  
\qquad \text{and} \qquad
\Zmin = \haze^{-1}(2\pi (\ell+m + \ERROR)) ,
\]
where $\haze^{-1}$ is defined as in \refeqn{HazeInv}, and recall the functions $F$ of \refdef{FDefine} and $q$ of \reflem{VisualAreaSigma}. 
Then the function
\[
f(t) = f_{\ell,m} (t) =  2\pi \ell \left( \frac{\sqrt{t}}{2\pi} \right)^{q(Z_0)} +  2\pi m \exp((4\pi^2 - t) F(\Zmin, \ell) ) 
\]
satisfies
\[
f_{\ell,m} (t) < f_{\ell,m}(4\pi^2) + 2 \pi \cdot \ERROR = 2 \pi (\ell + m + \ERROR).
\]
\end{lemma}

In the proof of \refthm{ShortStaysShort}, we will see that $f_{\ell,m}(t)$ serves as an upper bound on the total visual area $\calA(\Sigma \cup \gamma)$ in the $g_t$ metric. Thus \reflem{AreaBoundCapped} will allow us to bound the visual area $\calA(\Sigma\cup \gamma)$ at time $t$ in terms of the visual area  at time $4 \pi^2$. Graphing suggests that the inequality $f_{\ell,m} (t) \leq f_{\ell,m}(4\pi^2)$ holds without any error term, but the computer-assisted proof requires a (tiny) error term. 

\begin{proof}[Proof of \reflem{AreaBoundCapped}]
Define an auxiliary function
\begin{equation*}
g_{\ell,m}(t) = f_{\ell,m}(4\pi^2) - f_{\ell,m} (t) = \int_t^{4\pi^2} f'_{\ell,m}(x) \, dx.
\end{equation*}
Then the conclusion of the lemma can be phrased as $g_{\ell,m}(t) > -2\pi \cdot \ERROR$ for all values of $\ell, m, t$.

We claim that $g_{\ell, m}(t)$ is smallest when $m$ is largest. This can be seen from the derivative
\[
f'(t) = 2\pi   \ell  \, \frac  {q(Z_0)}{2}  \cdot \frac{ t^{q(Z_0)/2 - 1}}{(2\pi)^{q(Z_0)}} - 2\pi m F(\Zmin, \ell) \exp \big( (4\pi^2 - t) F(\Zmin, \ell) \big) .
\]
Observe that $\Zmin$ is a decreasing function of $m$, while $F(z,\ell)$ is a decreasing function of $z$. Thus $F(\Zmin,\ell)$ is largest when $m$ is largest. Hence the subtracted term in $f'(t)$ is largest when $m$ is largest. Thus $f'(t)$ is smallest when $m$ is largest, and the claim follows.

Now, we set $m = 0.0996 - 0.352\cdot \ell$ and claim that $g_{\ell,m}(t) > -2\pi \cdot \ERROR$ for all $\ell, t$ with this value of $m$. This is established using interval arithmetic in Sage; see the ancillary files \cite{FPS:Ancillary}.
\end{proof}

\begin{theorem}\label{Thm:ShortStaysShort}
Let $M$ be a complete, finite volume hyperbolic $3$--manifold.
Let $\Sigma$ be a geodesic link in $M$, and $\gamma$ a closed geodesic disjoint from $\Sigma$.
Let $\ell = \len_{4\pi^2}(\Sigma)$ and $m = \len_{4\pi^2}(\gamma)$ be the lengths of $\Sigma$ and $\gamma$ in the complete metric on $M$. 
Suppose $\ell \leq 0.0735$ and $m \leq 0.0996 - 0.352\cdot \ell$.
 Let
\[
\Rmin = h^{-1}(2\pi(\ell + m + \ERROR)) \geq \arctanh(0.6288). 
\]

Then $M - \Sigma$ is connected to $M$ via a cone deformation parametrized by $t = \alpha^2$, so that for all $t$,
\begin{enumerate}
\item\label{Itm:GammaGeodesic} The curve $\gamma$ is a geodesic in the cone metric $g_t$. Furthermore, the cone deformation preserves $\gamma$ setwise.
\item\label{Itm:BigZ} There is an embedded multi-tube about $\Sigma \cup \gamma$ of radius greater than $\Rmin$.
\item\label{Itm:LengthControl} The complex length of $\gamma$ satisfies
\[
\dhyp(\calL_t(\gamma), \,  \calL_{4\pi^2}(\gamma))
\: \leq \:  (4\pi^2 - t)\,F(\Zmin, \ell) ,
\]
where $\Zmin = \tanh \Rmin$ and $F(z, \ell)$ is the function of \refdef{FDefine}.
\end{enumerate}
In particular, the length of $\gamma$ in the $M$ and $M-\Sigma$ satisfies $\dhyp(\calL_0(\gamma), \,  \calL_{4\pi^2}(\gamma))
\: \leq \:  4\pi^2 \,F(\Zmin, \ell) $.
\end{theorem}

\begin{proof}
The proof is a crawling argument in the spirit of \refthm{ConeDefExistsRBounds}. By that theorem, there is a cone deformation connecting $M - \Sigma$ to $M$, parametrized by $t = \alpha^2$. Furthermore, this cone deformation maintains an embedded tube about $\Sigma$ of radius at least $R_0 = h^{-1}(2\pi \ell)$. (Note that for this proof, the lower bound on tube radius about $\Sigma$ is denoted $R_0$ rather than $\Rmin$.)

Let $I$ be the maximal sub-interval of $[0, (2\pi)^2]$ containing $(2\pi)^2$, such that conclusions \refitm{GammaGeodesic}, \refitm{BigZ} and \refitm{LengthControl} hold for $t \in I$.

First, we show that $(2\pi)^2 \in I$, hence $I$ is non-empty. Note that \refitm{GammaGeodesic} holds by hypothesis, while \refitm{LengthControl} is vacuous for $t = (2\pi)^2$. Let $U_{\max}$ be the maximal multi-tube about $\Sigma^+$ in the complete metric on $M$, and let $R$ be the smallest radius of a tube in $U_{\max}$.
By \reflem{Meyerhoff}, we know $R > 0.531$. Now, \refcor{HInverse} says that 
\[
R \geq h^{-1} (\calA(4\pi^2)) = h^{-1} (2\pi (\ell + m)) > \Rmin
\]
where the strict inequality is by definition of $\Rmin$.
Thus \refitm{BigZ} holds for $t = (2\pi)^2$, implying that $I \neq \emptyset$.

To see that $I$ is open, let $0 < t_0 \in I$. 
By \refthm{LocalConeDeformation}, there is a local cone deformation on $M$ that treats $\Sigma^+$ as its singular locus but does not change the cone angle on $\gamma$. Hence \refitm{GammaGeodesic} holds in an open neighborhood of $t_0$, as does \refitm{BigZ} because it is an open condition. Now, \reflem{ComplexLengthDeriv}  implies that \refitm{LengthControl} holds on the union of $I$ and this open neighborhood. Hence $I$ is open.

To see that $I$ is closed, let $a = \inf I$. Since the tube radius about $\Sigma^+$ must remain \emph{at least} $\Rmin$ by continuity, and in particular does not degenerate, \refthm{hk3.12} implies that the cone deformation preserving $\gamma$ setwise can be extended to  $t=a$.
Since \refitm{LengthControl} is a closed condition,  it holds at $t = a$ by continuity. Thus, for $t \in [a, (2\pi)^2]$,
\reflem{KenPropH2} gives
\begin{align*}
2\pi \len_t(\gamma) 
& \leq 2\pi \len_{4\pi^2}(\gamma) \cdot \exp((4\pi^2 - t)\,F(\Zmin, \ell) ) \\
&= 2\pi m \,  \exp((4\pi^2 - t)\,F(\Zmin, \ell) ).
\end{align*}
Similarly, by \reflem{VisualAreaSigma}, the visual area of $\Sigma$ satisfies 
\[
 \calA_t(\Sigma)  \: \leq \:  \calA_{4\pi^2}(\Sigma) \left( \frac{\sqrt{t}}{2\pi} \right)^{q(Z_0)} 
\! \! = \: 2\pi \ell \left( \frac{\sqrt{t}}{2\pi} \right)^{q(Z_0)} 
\]
where $Z_0 = \tanh R_0$ and the function $q(z)$ is given in \reflem{VisualAreaSigma}. Combining the last two equations, we conclude that the total visual area of $\Sigma^+ = \Sigma \cup \gamma$ satisfies
\begin{align*}
\calA_t(\Sigma^+) 
&\leq 2\pi \ell \left( \frac{\sqrt{t}}{2\pi} \right)^{q(Z_0)} \! \! + 2\pi m \,  \exp((4\pi^2 - t)\,F(\Zmin, \ell) ) 
= f_{\ell,m}(t).
\end{align*}

By \reflem{AreaBoundCapped}, the function $f_{\ell,m}(t)$ is bounded above in terms of $f_{\ell,m}(4\pi^2)$.
In symbols,
\[
\calA_t(\Sigma^+) \leq f_{\ell,m}(t)  < 2\pi (\ell +  m + \ERROR) = h(\Rmin).
\]
Thus  \refcor{HInverse} implies that at $t=a$, the maximal tube $U_{\max}$ has smallest radius $R > \Rmin$. This means condition \refitm{BigZ} holds at $a = \inf I$, hence $I$ is closed. Thus
\refitm{GammaGeodesic}, \refitm{BigZ} and \refitm{LengthControl} hold for all $t \in [0, (2\pi)^2]$.
\end{proof}

\begin{corollary}\label{Cor:ParticularLenBound}
Let $M$ be a complete, finite volume hyperbolic $3$--manifold.
Let $\Sigma$ be a geodesic link in $M$, and $\gamma$ a closed geodesic disjoint from $\Sigma$.
Let $\ell = \len_{4\pi^2}(\Sigma)$ and $m = \len_{4\pi^2}(\gamma)$ be the lengths of $\Sigma$ and $\gamma$ in the complete metric on $M$. Suppose that $\max(\ell, m) \leq 0.0735$. Then $\gamma$ is isotopic to a geodesic in the complete metric $g_0$ on $M - \Sigma$, and furthermore
\[
1.9793^{-1} \leq \frac{\len_0(\gamma)}{\len_{4\pi^2}(\gamma)} \leq 1.9793
\qquad \text{and} \qquad
|\tau_0(\gamma) - \tau_{4\pi^2}(\gamma) | \leq 0.05417.
\]
\end{corollary}

\begin{proof}
First, observe that the hypothesis $m \leq 0.0735$ implies $m \leq 0.0996 - 0.352\cdot \ell$ when $0 \leq \ell \leq 0.0735$. Thus our hypotheses are stronger than those of \refthm{ShortStaysShort}.
Now, by \refthm{ShortStaysShort},
\[
\dhyp(\calL_0(\gamma), \,  \calL_{4\pi^2}(\gamma))
\: \leq \:  (4\pi^2)\,F(\Zmin, \ell), 
\]
where $\Zmin=\tanh\Rmin = \haze^{-1}(2\pi(\ell+m+\ERROR)) \geq 0.6299$.
Now, substitute 
\[
K = (2\pi)^2  F(\Zmin, \ell) \leq 0.6827
\]
into \reflem{KenPropH2}, with the given bounds on $\ell$ and $\Zmin$, and the given bound on $m$.
\end{proof}

We also have a version of \refthm{ShortStaysShort} for upward cone deformations.

\begin{theorem}\label{Thm:ShortStaysShortUpward}
Let $M$ be a complete, finite volume hyperbolic $3$--manifold and $\Sigma$ a geodesic link in $M$. Suppose that, in the complete hyperbolic structure on $M - \Sigma$, the total normalized length of the meridians of $\Sigma$ satisfies $L^2 \geq 128$. 
Let $\gamma \subset M - \Sigma$ be a closed geodesic of length $m = \len_{0} (\gamma) \leq 0.056$.
Define
\[
\Rmin = h^{-1}\left(\frac{(2\pi)^2}{L^2 - 14.7} + 2\pi \cdot 1.656 \, m \right) > \arctanh (0.624).
\]

Then $M - \Sigma$ is connected to $M$ via a cone deformation parametrized by $t = \alpha^2$, so that for all $t$,
\begin{enumerate}
\item The curve $\gamma$ is a geodesic in the cone metric $g_t$. Furthermore, the cone deformation preserves $\gamma$ setwise.
\item There is an embedded multi-tube about $\Sigma \cup \gamma$ of radius greater than $\Rmin$.
\item The complex length of $\gamma$ satisfies
\[
\dhyp(\calL_0(\gamma), \,  \calL_t(\gamma))
\: \leq \:  t \,F(\Zmin, \ell) ,
\]
where $\ell \leq \frac{2\pi}{L^2 - 14.7}$
 and $F(z,\ell)$ is as in \refdef{FDefine}.
\end{enumerate}
In particular, the length of $\gamma$ in the complete structures on $M$ and $M-\Sigma$ satisfies $\dhyp(\calL_0(\gamma), \,  \calL_{4\pi^2}(\gamma))
\: \leq \:  4\pi^2 \,F \big( \Zmin, \frac{2\pi}{L^2 - 14.7} \big) $.
\end{theorem}

As usual, $\ell$ denotes the length of $\Sigma$ in the hyperbolic metric on $M$. However, to apply the theorem, one only needs geometric hypotheses on $M-\Sigma$ and the inequality $\ell \leq \frac{2\pi}{L^2 - 14.7}$.

\begin{proof}
We begin by noting that $L^2 \geq 128 \geq I(Z_0)$, where $I(z)$ is the function of \refdef{Ifunction} and $Z_0 \geq 0.8925 > 1/\sqrt{3}$. Thus, by \refthm{UpwardConeDefRBounds}, there is a cone deformation from $M - \Sigma$ to $M$, parametrized by $t = \alpha^2$, for which the tube radius about $\Sigma$ stays bounded below by $R_0$. (We denote this quantity by $R_0$ rather than $\Rmin$ because $\Rmin$ has a different meaning in the present theorem.) Now, \reflem{MagidLengthGeneral}, using $Z_0 \geq 0.8925$, shows that the length of $\Sigma$ in the complete metric on $M$ is
\begin{equation}\label{Eqn:EllBoundForShort}
\ell \leq \frac{2\pi}{L^2 - (2\pi)^2 G(Z_0)}  \leq \frac{2\pi}{L^2 - 14.7} \leq 0.05546.
\end{equation}
We will use this bound on $\ell$ in applying \reflem{ComplexLengthDeriv}.

The rest of the proof is a crawling argument analogous to \refthm{ShortStaysShort}.
By \refthm{LocalConeDeformation} and \refthm{hk3.12}, the cone deformation on $(M, \Sigma)$ can be thought of as a cone deformation on $(M, \Sigma^+)$, provided the tube radius about $\Sigma^+ = \Sigma \cup \gamma$ does not degenerate. Thus conclusion \refitm{GammaGeodesic} will be immediate once we establish \refitm{BigZ}.

Let $J$ be the maximal sub-interval of  $[0, (2\pi)^2]$, containing $0$, such that \refitm{BigZ} and \refitm{LengthControl} both hold on $J$. By \reflem{Meyerhoff} and \refcor{HInverse}, there is an embedded tube about $\gamma$ in $M - \Sigma$ of radius
\[
R \geq h^{-1}(2\pi m) > \Rmin.
\]
Meanwhile, the horospherical cusp neighborhoods can be thought of as tubes of infinite radius about $\Sigma$. Thus \refitm{BigZ} holds at $t=0$. Since \refitm{LengthControl} is vacuous at $t=0$, we conclude that $J$ is non-empty.

The interval $J$ is open for the same reason as in \refthm{ShortStaysShort}. Let $t_0 \in J$. Since \refitm{BigZ} is an open condition, it holds on an open neighborhood about $t_0$. Now, \reflem{ComplexLengthDeriv}, combined with the estimate
\refeqn{EllBoundForShort}, implies that \refitm{LengthControl} holds on the union of $J$ and this open neighborhood. Hence $J$ is open.

To see that $J$ is closed, let $a = \sup J$. Since \refitm{LengthControl} is a closed condition, it holds at $t=a$ by continuity. Thus, by \reflem{KenPropH2}, we have
\begin{equation}\label{Eqn:GammaLengthTemp}
\len_a(\gamma) \leq  \exp(a \, F(\Zmin,\ell)) \cdot \len_0(\gamma)  <  1.656 \, m,
\end{equation}
where the second inequality uses \refeqn{EllBoundForShort} and $\Zmin >0.624$, the fact that $F$ is decreasing in $z$ and increasing in $\ell$ (since $\ell<0.056$), and the fact that $a \leq (2\pi)^2$.
Recall that $m = \len_0(\gamma)$. Meanwhile, \reflem{LenMonotonicity} implies that $\len_t(\Sigma)$ is increasing in $t$.  Thus, at time $t=a$,
\[
\calA(a) = \sqrt{a} \cdot \len_a(\Sigma) + 2\pi \len_a(\gamma) < 2\pi \cdot \frac{2\pi}{L^2 - 14.7} + 2\pi \cdot 1.656 \, m = h(\Rmin).
\]
Here, the first equality is by the definition of visual area, the inequality is by \refeqn{EllBoundForShort} and \refeqn{GammaLengthTemp}, and the second equality is by the definition of $\Rmin$. Now, \refcor{HInverse} implies $R \geq h^{-1} \calA(a) > \Rmin$, hence \refitm{BigZ} holds at $t=a$ as desired.
Therefore, $J$ is closed, hence $J = [0,(2\pi)^2]$.
\end{proof}

\begin{corollary}\label{Cor:ParticularLenBoundUp}
Let $M$ be a complete, finite volume hyperbolic $3$--manifold and $\Sigma$ a geodesic link in $M$. Suppose that, in the complete hyperbolic structure on $M - \Sigma$, the total normalized length of the meridians of $\Sigma$ satisfies $L^2 \geq 128$. Let $\gamma \subset M - \Sigma$ be a closed geodesic of length $\len_{0} (\gamma) \leq 0.056$. Then $\gamma$ is isotopic to a closed geodesic in $M$, and furthermore
\[
1.657^{-1} \leq \frac{\len_0(\gamma)}{\len_{4\pi^2}(\gamma)} \leq 1.657
\qquad \text{and} \qquad
|\tau_0(\gamma) - \tau_{4\pi^2}(\gamma) | \leq 0.0295.
\]
\end{corollary}
\begin{proof}
Plug in $t=(2\pi)^2$, $\ell \leq \frac{2\pi}{L^2-14.7} \leq \frac{2\pi}{113.3}$, $m \leq 0.056$ and $\Zmin > 0.624$ into \reflem{KenPropH2} to obtain the result.
\end{proof}

\subsection{Application to cosmetic surgery}
We now present the main application of this section: effective control on cosmetic surgeries.

\begin{definition}\label{Def:SystoleL}
Choose a real number $L \geq 10.1$. Let $F$ be the function of \refdef{FDefine}. Define
\[
\ell_{\max} = \ell_{\max}(L) = \frac{2 \pi}{L^2 - 16.03}
\] 
and
\[
\sysmin(L) = \ell_{\max} \exp( 4\pi^2 F ( \haze^{-1} ( 4 \pi \, \ell_{\max}  + 2\pi \, \ERROR), \, \ell_{\max} ) ) .
\]
\end{definition}

\begin{lemma}\label{Lem:SFunctionProps}
The function $\sysmin(L)$ is strictly decreasing in $L$. Furthermore, for $L \geq 10.1$,
\[
\frac{2\pi}{L^2} < \sysmin(L) < \frac{2\pi}{L^2 - 58} .
\]
\end{lemma}

\begin{proof}
To see that $\sysmin(L)$ is decreasing, we examine the ingredients of its definition.
By \refcor{HInverse}, and \refrem{Haze}, $\haze^{-1} ( 4 \pi \ell_{\max} +2\pi\ERROR)$ is a decreasing function of $\ell_{\max}$. By a derivative computation, the function
\[
\frac{F(z,\ell)}{\ell} = \frac{(1+ z^2)}{  z^3 (3-z^2)} \cdot \frac{1}{10.667 - 20.977 \ell}
\]
 is decreasing in $z$ and increasing in $\ell$. Thus
the combined function $F ( \haze^{-1} ( 4 \pi \ell_{\max} +2\pi\ERROR), \, \ell_{\max} )$ is increasing in $\ell_{\max}$. Since $\ell_{\max} = \ell_{\max}(L)$ is strictly decreasing in $L$, we conclude that $\sysmin(L)$ is strictly decreasing.

The lower bound on $\sysmin(L)$ holds because
\(
\sysmin(L) > \ell_{\max}(L) > 2\pi / L^2.
\)
Meanwhile, by the definition of $\ell_{\max}$, the desired upper bound on $\sysmin(L)$ is equivalent to
\[
\frac{\sysmin(L)}
{\ell_{\max}} 
\: < \: \frac{2\pi}{L^2 - 58} \cdot \frac{L^2 - 16.03}{2\pi} \, .
\]
After substituting the definition of $\sysmin(L)$, taking logarithms, and dividing both sides by $\ell_{\max}$ again, the upper bound becomes equivalent to the inequality
\begin{equation}\label{Eqn:SUpper}
4\pi^2 \, \frac{F ( \haze^{-1} ( 4 \pi \ell_{\max} + 2\pi \ERROR), \, \ell_{\max} ) }{\ell_{\max}}
\: < \: \log \left ( \frac{L^2 - 16.03}{L^2 - 58}  \right)   \cdot \frac{L^2 - 16.03}{2\pi} \, . 
\end{equation}
It remains to prove \refeqn{SUpper}.

 A derivative calculation shows that the right-hand side of \refeqn{SUpper} is decreasing in $L$. As $L \to \infty$, its limit is
\begin{align*}
\lim_{L \to \infty} -\log \left ( \frac{L^2 - 58}{L^2 - 16.03}  \right)   \cdot \frac{L^2 - 16.03}{2\pi}
&= \lim_{L \to \infty} -\log \left ( 1 - \frac{41.97}{L^2 - 16.03}  \right)   \cdot
\frac{L^2 - 16.03}{2\pi} \\
& = \lim_{L \to \infty} \frac{41.97}{L^2 - 16.03} \cdot \frac{L^2 - 16.03}{2\pi} \\
& = \frac{41.97}{2\pi} = 6.679 \ldots.
\end{align*}
Here, the second equality uses the linear approximation $\log(1-x) \sim -x$ for $x$ near $0$. Therefore, the right-hand side is at least $6.679$ for all values of $L$ in the domain.

Meanwhile, we have already checked that the left-hand side of \refeqn{SUpper} is increasing in $\ell_{\max}$, hence decreasing in $L$. Direct calculation shows that the left-hand side equals $6.674 \ldots$ when $L = 11$. Thus inequality \refeqn{SUpper} holds for all $L \geq 11$.

Finally, for $L \in [10.1, 11]$, inequality \refeqn{SUpper} is established using interval arithmetic in Sage. See the ancillary files \cite{FPS:Ancillary} for details.
\end{proof}

\begin{theorem}\label{Thm:UniqueShortest}
For a real number $L_0 \geq 10.1$, let $\sysmin(L_0)$ be the function of \refdef{SystoleL}.
Let $N$ be a cusped hyperbolic $3$--manifold whose systole is at least $\sysmin(L_0)$. Let $\mathbf{s}$ be a tuple of surgery slopes on the cusps of $N$, whose normalized length is $L = L(\mathbf{s}) \geq L_0$. 

Then the Dehn filled manifold $M = N(\mathbf s)$ is hyperbolic. The core $\Sigma$ of the Dehn filling solid tori is isotopic to a geodesic link with an embedded tubular neighborhood of radius at least $1.281$.
Finally, the only geodesics in $M$ of length at most $\ell = \len(\Sigma)$ are the components of $\Sigma$ itself.
\end{theorem}

\begin{proof}
Let $M = N(\mathbf s)$, and let $\Sigma$ be the union of the cores of the filled solid tori. We will apply \refthm{UpwardConeDefRBounds}. Note the hypotheses imply that the normalized length is at least
$L_0^2 \geq (10.1)^2 \geq I(0.8568),$ where $I$ is the function of \refdef{Ifunction}. 
By \refthm{UpwardConeDefRBounds}, $M$ admits a hyperbolic metric that is connected to the complete metric on $N = M - \Sigma$ by a cone-deformation with singular locus along $\Sigma$. Moreover, the cone deformation maintains a tube about $\Sigma$ of radius $R_0\geq \arctanh(0.8568)\geq 1.281$. At the end of this cone deformation, $\Sigma$ becomes a geodesic link in the complete metric on $M$. 
The length of $\Sigma$ in this complete metric  satisfies
\[
\ell = \len_{4\pi^2}(\Sigma) <  \frac{2 \pi}{L^2 - 16.03} = \ell_{\max} < 0.0731.
\]
Here, the first inequality follows by \reflem{MagidLengthGeneral}, plugging in the value $4\pi^2 G(\tanh(R_0)) = 16.028 \ldots$. Meanwhile, the last inequality uses the hypothesis $L \geq 10.1$.

Suppose, for a contradiction, that $M$ contains a closed geodesic $\gamma \not\subset \Sigma$ such that $\len_{4\pi^2}(\gamma) =m \leq \ell_{\max}$. By Meyerhoff's theorem \cite[Section 7]{meyerhoff}, $\gamma \cap \Sigma = \emptyset$. Thus $\Sigma^+ = \Sigma \cup \gamma$ is a geodesic link satisfying the hypotheses of \refthm{ShortStaysShort}. 
Matching the definition of $\Rmin$ in \refthm{ShortStaysShort},
 we define
\[ 
\Zmin 
= \haze^{-1}(2 \pi  (\ell+m + \ERROR))
>  \haze^{-1}(4\pi  \ell_{\max}+2\pi \cdot \ERROR) >
0.6337.
\]
By Conclusion \refitm{LengthControl} of \refthm{ShortStaysShort}, we have
\[
\dhyp(\calL_0(\gamma), \,  \calL_{4\pi^2}(\gamma))
\: \leq \:  4\pi^2 \,F(\Zmin, \ell) \: < \:  4\pi^2 \,F(\Zmin, \ell_{\max}),
\]
hence \reflem{KenPropH2} implies the length of $\gamma$ in $M_0 = M - \Sigma = N$ is
\begin{align*}
\len_0(\gamma) 
& <   \len_{4\pi^2}(\gamma) \, \exp( 4\pi^2 F(\Zmin, \ell_{\max})) \\
& <  \ell_{\max} \,  \exp( 4\pi^2 F(\haze^{-1}(4\pi\ell_{\max}+2\pi\ERROR), \ell_{\max})) \\
& = \sysmin(L) \: \leq \: \sysmin(L_0),
\end{align*}
using the fact that $F(z,\ell)$ is decreasing in $z$. 
But this contradicts the hypothesis that $\systole(N) \geq \sysmin(L_0)$. This contradiction implies that the components of $\Sigma$ are the \emph{only} geodesics in $M$ of length less than $\ell_{\max}$, completing the proof.
\end{proof}

Now, \refthm{UniqueShortest} combined with topological rigidity of hyperbolic manifolds \cite{Gabai:TopologicalRigidity, GabaiMeyerhoffThurston} implies

\begin{theorem}\label{Thm:Cosmetic}
Let $N$ be a cusped hyperbolic $3$--manifold. Suppose that $\mathbf s_1, \mathbf s_2$ are distinct tuples of slopes on the cusps of $N$, where the normalized length of each $\mathbf s_i$ satisfies
\[
L(\mathbf s_i) \geq \max \left \{ 10.1, \, \sqrt{ \frac{2\pi}{\systole(N)}  + 58 } \right \}.
\]
Then any homeomorphism $\varphi \from N(\mathbf s_1) \to N(\mathbf s_2)$ restricts (after an isotopy) to a self-homeomorphism  of $N$ sending $\mathbf s_1$ to $\mathbf s_2$. 
In particular, if $\systole(N) \geq 0.1428$, then the above conclusions hold for all pairs $(\mathbf s_1, \mathbf s_2)$ of normalized length at least $10.1$.
\end{theorem}

\begin{proof}
For $i = 1,2$, let $L_i = L(\mathbf s_i)$ be the normalized length of the tuple of slopes $\mathbf s_i$. By hypothesis, $L_i \geq 10.1$ and $L_i\geq\sqrt{2\pi/\systole(N) + 58}$. Combining these hypotheses with \reflem{SFunctionProps}, we have
\[
\sysmin(L_i) \leq \sysmin \left( \sqrt{ \frac{2\pi}{\systole(N)}  + 58 } \right) < \systole(N).
\]

Let $\Sigma_i \subset N(\mathbf s_i)$ be the union of cores of the Dehn filling solid tori. Then, by \refthm{UniqueShortest}, the $k$ components of $\Sigma_i$ are the shortest $k$-tuple of geodesics in the hyperbolic manifold $N(s_i)$. Furthermore, there is a  tube of radius more than $1$ about $\Sigma_i$. Note if $N(\mathbf s_1) \cong N(\mathbf s_2)$, the number $k$ of components of $\Sigma_1$ must equal the number of components of $\Sigma_2$.

Now, consider a homeomorphism $\varphi \from N(\mathbf s_1) \to N(\mathbf s_2)$. By Mostow rigidity, combined with a theorem of Gabai \cite{Gabai:TopologicalRigidity}, $\varphi$ is isotopic to an isometry. (See also Gabai, Meyerhoff, and Thurston \cite{GabaiMeyerhoffThurston}.) 
This isometry must carry the shortest $k$-tuple of geodesics in $N(\mathbf s_1)$ to the shortest $k$-tuple of geodesics in $N(\mathbf s_2)$. Thus, after adjusting $\varphi$ by an isotopy, we may suppose that $\varphi(\Sigma_1) = \Sigma_2$. Hence $\varphi$ restricts to a homeomorphism from $N = N(\mathbf s_1) - \Sigma_1$ to $N = N(\mathbf s_2) - \Sigma_2$, sending $\mathbf s_1$ to $\mathbf s_2$.
\end{proof}

When $N$ has one cusp, we have the following corollary.

\begin{theorem}\label{Thm:CosmeticOneCusp}
Let $N$ be a one-cusped hyperbolic $3$--manifold. 
Suppose that $s_1$ and $s_2$ are distinct slopes on the cusp of $N$, where the normalized length of each $s_i$ satisfies
\[ L(s_i) \geq \max \left\{ 10.1, \sqrt{\frac{2\pi}{\systole(N)} + 58} \right\}. \]
Then $(s_1, s_2)$ cannot be a purely cosmetic pair. If $(s_1, s_2)$ is a chirally cosmetic pair, then there is a symmetry of $N$ sending $s_1$ to $s_2$.
\end{theorem}

In particular, 
Conjectures~\ref{Conj:Cosmetic} and~\ref{Conj:CosmeticMultiCusp} both hold for pairs $(s_1, s_2)$ satisfying the above bound on length.

\begin{proof}[Proof of \refthm{CosmeticOneCusp}]
Suppose there is a homeomorphism $\varphi \from N(s_1) \to N(s_2)$. Then, by \refthm{Cosmetic}, $\varphi$ restricts to a homeomorphism of $N$ sending $s_1$ to $s_2$.  That $\varphi \from N \to N$ must be orientation-reversing follows from a standard argument, as in \cite[Lemma 2]{BleilerHodgsonWeeks}. 

Let $\lambda$ be the unique null-homologous slope on the cusp of $N$. Thus $\varphi \vert_N$ must send $\lambda$ to $\lambda$. By Mostow--Prasad rigidity, $\varphi \vert_N$ is homotopic to an isometry. If an isometry of $N$ is orientation-preserving and stabilizes $\lambda$, then it must stabilize every slope, implying $s_1 =  s_2$. Since we have assumed that $ s_1 \neq s_2$, it follows that $\varphi \vert_N$ is orientation-reversing, hence $\varphi$ is also. 
\end{proof}

\subsection{Controlling multiple geodesics}

The following theorem is included here for future use. 

\begin{theorem}\label{Thm:HoldShortGeodesics}
Let $M$ be a complete, finite volume hyperbolic $3$--manifold. Let $\Sigma  = \sigma_1 \cup \ldots \cup \sigma_n$ and $\Sigma^+ = \sigma_1 \cup \ldots \cup \sigma_{n+k}$ be geodesic links in $M$, where $k \geq 1$. Assume that, in the complete structure on $M$, we have $\ell = \len_{4\pi^2}(\Sigma) \leq 0.735$. In addition, define
\[
m = \max \, \big\{ \len_{4\pi^2}(\sigma_j) \: : \: n+1 \leq j \leq n + k \big\}
\]
and assume that $\ell + 2m \leq 0.14$. 

Then $M_0 = M - \Sigma$ is connected to $M_{4\pi^2} = M$ via a cone deformation that preserves $\Sigma^+$ setwise and keeps each component of $\Sigma^+$ geodesic.  
\end{theorem}

One novelty of
\refthm{HoldShortGeodesics} is that it does not care about the total length of $\Sigma^+$. All that the theorem needs is for the drilling locus $\Sigma$  to be short, and for each separate component of 
$\Sigma^+ - \Sigma$ to be (uniformly) short. Under these hypotheses, one may use \refthm{ShortStaysShort} to estimate the change in length of each component of $\Sigma^+ - \Sigma$.

\begin{proof}[Proof of \refthm{HoldShortGeodesics}]
If $k=1$, that is if $\Sigma^+ = \Sigma \cup \gamma$ for a single closed curve $\gamma$, this theorem is already covered by \refthm{ShortStaysShort}. In the general case, when $k \geq 2$, the proof closely parallels that proof. Define $R_0 = h^{-1}(2 \pi \ell) \geq \arctanh(1/\sqrt{3})$. By \refthm{ConeDefExistsRBounds}, there is a cone deformation connecting $M - \Sigma$ to $M$, which maintains an embedded tube about $\Sigma$ of radius at least $R_0$.

Next, define 
\begin{equation}\label{Eqn:RminHoldGeodesics}
\Rminhat = h^{-1}(2\pi(\ell + 2m + \ERROR)) > 0.794,
\end{equation}
and set $\Zminhat = \tanh (\Rminhat)$ as usual.
Our hypothesis on $\ell+2m$ ensures that $2\pi(\ell + 2m + \ERROR) < \hmax$, hence \refcor{HInverse} ensures that $h^{-1}$ is decreasing. 
We claim that the above cone deformation can be chosen such that 
 the following hold for all $t \in [0, 4\pi^2]$:
\begin{enumerate}
\item
The link $\Sigma^+$ is a union of geodesics in the cone metric $g_t$. Furthermore, the cone deformation preserves $\Sigma^+$ setwise.
\item
There is an embedded multi-tube about $\Sigma^+$ of radius greater than $\Rminhat$.
\item
For every curve $\sigma_j$ with $j > n$, the complex length  satisfies
\[
\dhyp(\calL_t(\sigma_j), \,  \calL_{4\pi^2}(\sigma_j))
\: \leq \:  (4\pi^2 - t)\,F(\Zminhat, \ell) .
\]
\end{enumerate}

Let $I$ be the maximal sub-interval of $[0, (2\pi)^2]$ containing $(2\pi)^2$, such that conclusions \refitm{GammaGeodesic}, \refitm{BigZ} and \refitm{LengthControl} hold for $t \in I$.

First, we show that $(2\pi)^2 \in I$, hence is non-empty. Note that \refitm{GammaGeodesic} holds by hypothesis, while \refitm{LengthControl} is vacuous for $t = (2\pi)^2$. To verify \refitm{BigZ}, choose an arbitrary pair of components $\sigma_j, \sigma_{j'}$ with $n < j,j' \leq n+k$.
Let $U_{\max}$ be the maximal multi-tube about $\Sigma \cup \sigma_j \cup \sigma_{j'}$ in the complete metric on $M$, and let $R$ be the smallest radius of a tube in $U_{\max}$.
By \reflem{Meyerhoff}, we know $R > 0.531$. Now, \refcor{HInverse} says that 
\[
R \geq h^{-1} (\calA(4\pi^2)) = h^{-1} (2\pi (\ell + 2m)) > \Rminhat
\]
where the strict inequality is by definition of $\Rminhat$. In particular, the tubes of radius $\Rminhat$ about the components of $\Sigma \cup \sigma_j \cup \sigma_{j'}$ are pairwise disjoint. Since $j,j'$ were arbitrary, we conclude that the tubes of radius $\Rminhat$ about \emph{all} components of $\Sigma^+$ are pairwise disjoint. 
Thus \refitm{BigZ} holds for $t = (2\pi)^2$, implying that $I \neq \emptyset$.

To see that $I$ is open, let $0 < t_0 \in I$. 
By \refthm{LocalConeDeformation}, there is a local cone deformation on $M$ that treats $\Sigma^+$ as its singular locus but does not change the cone angles on $\Sigma^+ - \Sigma$. Hence \refitm{GammaGeodesic} holds in an open neighborhood of $t_0$, as does \refitm{BigZ} because it is an open condition. Now, \reflem{ComplexLengthDeriv}  implies that \refitm{LengthControl} holds on the union of $I$ and this open neighborhood. Hence $I$ is open.

To see that $I$ is closed, let $a = \inf I$. Since the tube radius about $\Sigma^+$ must remain \emph{at least} $\Rminhat$ by continuity, and in particular does not degenerate, \refthm{hk3.12} implies that the cone deformation preserving $\Sigma^+$ setwise can be extended to  $t=a$.
Since \refitm{LengthControl} is a closed condition,  it holds at $t = a$ by continuity. Thus, for every $j>n$ and every $t \in [a, (2\pi)^2]$,
\reflem{KenPropH2} gives
\begin{align*}
2\pi \len_t(\sigma_j) 
& \leq 2\pi \len_{4\pi^2}(\sigma_j) \cdot \exp((4\pi^2 - t)\,F(\Zminhat, \ell) ) \\
&\leq 2\pi m \,  \exp((4\pi^2 - t)\,F(\Zminhat, \ell) ).
\end{align*}
Focusing attention on two components $\sigma_j$ and $\sigma_{j'}$ with $j, j' > n$, we have
\[
2\pi (\len_t(\sigma_j) + \len_t(\sigma_{j'}) ) \leq 4\pi \cdot  m \,  \exp((4\pi^2 - t)\,F(\Zminhat, \ell) ).
\]
Meanwhile, by \reflem{VisualAreaSigma}, the visual area of $\Sigma$ satisfies 
\[
 \calA_t(\Sigma)  \: \leq \:  \calA_{4\pi^2}(\Sigma) \left( \frac{\sqrt{t}}{2\pi} \right)^{q(Z_0)} 
\! \! = \: 2\pi \ell \left( \frac{\sqrt{t}}{2\pi} \right)^{q(Z_0)} 
\]
where $Z_0 = \tanh R_0$ and the function $q(z)$ is given in \reflem{VisualAreaSigma}. Combining the last two equations, we conclude that the total visual area of $\Sigma \cup \sigma_j \cup \sigma_{j'}$ satisfies
\begin{align*}
\calA_t(\Sigma \cup \sigma_j \cup \sigma_{j'}) 
&\leq 2\pi \ell \left( \frac{\sqrt{t}}{2\pi} \right)^{q(Z_0)} \! \! + 4\pi m \,  \exp((4\pi^2 - t)\,F(\Zminhat, \ell) ) 
= \hat f_{\ell,m}(t).
\end{align*}
(Note that $\hat f_{\ell,m}(t)$ differs from the function $f_{\ell,m}(t)$ of \reflem{AreaBoundCapped} in that the second term begins with $4\pi m$ instead of $2\pi m$. Note as well that $\Zminhat$ is defined via \refeqn{RminHoldGeodesics}, which differs from the definition of $\Zmin$ in \reflem{AreaBoundCapped}.) 

By
an interval arithmetic computation in Sage \cite{FPS:Ancillary}, exactly as in \reflem{AreaBoundCapped}, we learn that $\hat f_{\ell,m}(t)$ is bounded above in terms of $\hat f_{\ell,m}(4\pi^2)$.
In symbols, we verify the strict inequality below:
\begin{equation}\label{Eqn:HoldGeodesicsFhat}
\calA_t(\Sigma \cup \sigma_j \cup \sigma_{j'})   \leq \hat f_{\ell,m}(t)  
 < 2\pi (\ell +  2m + \ERROR) 
 = h(\Rminhat).
\end{equation}
Thus  \refcor{HInverse} implies that at $t=a$, the maximal tube about $\Sigma \cup \sigma_j \cup \sigma_{j'}$
has smallest radius $R > \Rminhat$. Since $j,j'$ were chosen arbitrarily, the maximal multi-tube about $\Sigma^+$ also has smallest radius $R > \Rminhat$.
This means condition \refitm{BigZ} holds at $a = \inf I$, hence $I$ is closed. Thus
\refitm{GammaGeodesic}, \refitm{BigZ} and \refitm{LengthControl} hold for all $t \in [0, (2\pi)^2]$, completing the proof.
\end{proof}

\section{Bilipschitz estimates in the thick part}\label{Sec:Bilip}

The main result of this section is \refthm{Bilip}. The theorem gives an effective bilipschitz bound on the change in geometry during a cone deformation.

\begin{definition}\label{Def:BilipConstant}
Given Riemannian metrics $g$ and $\hat{g}$ on a manifold $N$, we define the bilipschitz constant at a point $p\in N$ by
\begin{equation}\label{Eqn:BilipPointwise}
{\bilip}_p(g, \hat{g}) = \inf \left\{ K\geq 1 \: \left| \: \frac{1}{K} \leq \sqrt{\frac{\hat{g}(x, x)}{g(x,x)}} \leq K \right. \mbox{for all } x\in T_pN - \{ 0 \} \right\}.
\end{equation}
The bilipschitz constant between $g$ and $\hat{g}$ is
\[
\bilip_N (g, \hat{g}) = \sup \, \{ {\bilip}_p(g, \hat{g}) : p \in N \},
\]
with the convention that the supremum of an unbounded set is undefined. In the applications in this paper, the manifold $N$ will be compact, hence the supremum will actually be attained.
\end{definition}

\begin{theorem}\label{Thm:Bilip}
Fix $0 < \delta \leq 0.938$. Let $M$ be a complete, finite volume hyperbolic 3-manifold and $\Sigma$ a geodesic link in $M$. Suppose that one of the following hypotheses holds.
\begin{itemize}
\item
In the cusped structure on $M - \Sigma$, the total normalized length of the meridians of $\Sigma$ satisfies $L^2 \geq  107.6 / \delta^2 + 14.41$.
\item
In the complete structure on $M$, the total length of $\Sigma$ is $\ell \leq \delta^2 / 17.11$.
\end{itemize}
Then there is a cone deformation $M_t = (M, \Sigma, g_t)$ interpolating between the complete structures on $M - \Sigma$ and $M$. For $0 \leq a \leq b \leq (2\pi)^2$, the cone deformation defines a natural identity map $\id \from (M - \Sigma, g_a) \to (M - \Sigma ,g_b)$.

Suppose that $W$ is a submanifold of $M$ such that $W \subset M_t^{\geq \delta}$ for all  $t \in (a, b)$. Then  the identity map
$
\id \from (W, g_a) \to (W, g_b)
$
is $J$--bilipschitz, where
\[
J = \exp \left( \frac{7.193 \, \ell} {\delta^{5/2} } \right).
\]
\end{theorem} 

The natural identity map $\id \from (M, g_a) \to (M,g_b)$ was defined in \refrem{OmegaUniqueness}. It arises because we are keeping the pair of sets $(M, \Sigma)$ constant and varying the metric $g_t$ on $M-\Sigma$ according to a canonical choice of $1$--form $\omega$. One important property of this identity map is that it is equivariant with respect to the symmetry group of the pair $(M, \Sigma)$.

We remark that \refthm{Bilip} is an effective version of a theorem of Brock and Bromberg \cite[Corollary 6.10]{brock-bromberg:density}. Our proof follows their outline, with effective control on the boundary terms inserted into the calculation.

The following lemma shows that the hypothesis on $L$ of \refthm{Bilip} implies the hypothesis on $\ell$. Thus, in proving the theorem,  it suffices to assume the hypothesis on $\ell$. In addition, \reflem{LBoundImpliesEllBound} says that we may substitute $\ell \leq \frac{2\pi}{L^2-14.41}$ in bounding the bilipschitz constant $J$.

\begin{lemma}\label{Lem:LBoundImpliesEllBound}
Fix $0<\delta\leq 0.938$. Let $M$ be a complete, finite volume hyperbolic 3-manifold and $\Sigma$ a geodesic link in $M$.
Suppose that in the cusped structure on $M-\Sigma$, the total normalized length of the meridians of $\Sigma$ satisfies $L^2 \geq 107.6/\delta^2 + 14.41$. Then in the complete structure on $M$, the total length of $\Sigma$ is
\[ \ell < \frac{2\pi}{L^2-14.41} < \frac{\delta^2}{17.11}. \]
\end{lemma}

\begin{proof}
Since $\delta \leq 0.938$, we have
\[
L^2 \geq \frac{107.6}{\delta^2} + 14.41 \geq 136.7 = I(\Zmin),
\]
where $\Zmin \geq 0.9006$. Thus, by \reflem{MagidLengthGeneral}, we have 
\[
\ell \leq \frac{2\pi}{L^2 - 14.41} < \frac{\delta^2}{17.11}. \qedhere
\]  
\end{proof}

To prove \refthm{Bilip}, we recall from \refsec{ConeDef} that the cone deformation is governed by a harmonic form
$\omega$. By Equation~\refeqn{OmegaEta}, $\omega$ decomposes as
\begin{equation}\label{Eqn:OmegaRestate}
\omega = \eta + i *D\eta.
\end{equation}
It turns out that the real part $\eta = \Re(\omega)$ controls the infinitesimal change in metric. Hodgson and Kerckhoff pointed out that the metric inner product $g_t(x,y)$ between a pair of vectors
 $x,y \in T_pM$ satisfies the differential equation
\begin{equation}\label{Eqn:EtaCharacterization}
\frac{d g_t(x,y)}{dt} = 2g_t (x, \eta(y)).
\end{equation}
See, for instance, the displayed equation on page 46 of \cite{hk:survey}. We remark that \refeqn{EtaCharacterization} can be used to define the $TM$--valued $1$--form $\eta$, which is how Brock and Bromberg have defined it \cite[page 61]{brock-bromberg:density}.
Subsequently, they showed that the pointwise norm of $\eta$ controls the bilipschitz constant.

\begin{lemma}\label{Lem:BilipEta}
If $||\eta(p)|| \leq K$ for all $t\in[a, b]$, then
\[ {\bilip}_p(g_a, g_b) \leq e^{(b-a) K}. 
\] 
\end{lemma}

\begin{proof}
On \cite[pages~61--62]{brock-bromberg:density}, Brock and Bromberg show that  \refeqn{EtaCharacterization} implies
\[ \left| \frac{d g_t(x,x)}{dt} \right| \leq 2||\eta(p)|| \, g_t(x,x). \]
Here, the pointwise norm $||\eta(p)||$  should be evaluated at time $t$.
Integrating the above estimate,  we obtain
\[  e^{-2(b-a) K} \leq \frac{g_b(x,x)}{g_a(x,x)} \leq e^{2(b-a) K},
\quad \mbox{hence} \quad
{\bilip}_p(g_a, g_b) \leq e^{(b-a) K}. \qedhere
\]
\end{proof}

\subsection{Bounding the pointwise norm}
We will control the pointwise norm $ \|\eta(p)\| $ by combining the results of \refsec{BoundaryBound} and the following mean-value inequality due to Hodgson, Kerckhoff, and Bromberg \cite{bromberg:conemflds}.

\begin{theorem}\label{Thm:MeanValue}
Let $\omega$ be a harmonic form on a ball $B_r(p) \subset \HH^3$, where $r<\pi/\sqrt{2}$.  Then 
\begin{equation}\label{Eqn:MeanValue}
\| \omega(p)\| \leq \frac{3\sqrt{2\pi (\sinh(2r) - 2r)}}{4\pi f(r)} \sqrt{\int_{B_r} \| \omega \|^2 dV,}
\end{equation}
where
\[
f(r) = \cosh(r)\sin(\sqrt{2}r) - \sqrt{2}\sinh(r)\cos(\sqrt{2}r).
\]
\end{theorem}

\begin{proof}
See Bromberg \cite[Theorem 9.9]{bromberg:conemflds}.  We have substituted the formula $\vol(B_r(p)) = \pi (\sinh(2r) - 2r)$.  See, for example, Ratcliffe \cite[Exercise 3.4.5]{Ratcliffe:Foundations}.
\end{proof}

To simplify the bound of \refeqn{MeanValue}, we employ the following estimate.

\begin{lemma}\label{Lem:Puiseux}
  Let $0< \delta \leq 0.938$, and let $r=\delta/2$. Then the term in \refeqn{MeanValue} before the square root of the integral satisfies:
\begin{equation}\label{Eqn:MeanValueMultiplier}
  \frac{3\sqrt{2\pi(\sinh(2r)-2r)}}{4\pi f(r)} \leq \left( \frac{1.046}{2}\sqrt{\frac{3}{\pi}}\right) \left(\frac{\delta}{2}\right)^{-3/2}.
\end{equation}
  If $0<\delta\leq 0.106$ and $r = \delta/2$, the bound becomes
\begin{equation}\label{Eqn:MeanValueMultiplierBis}
  \frac{3\sqrt{2\pi(\sinh(2r)-2r)}}{4\pi f(r)} \leq \left( \frac{1.001}{2}\sqrt{\frac{3}{\pi}}\right) \left(\frac{\delta}{2}\right)^{-3/2}.
\end{equation}
\end{lemma}

\begin{proof}
Before giving the proof, we motivate the statement.  The function to be bounded in \refeqn{MeanValueMultiplier} and \refeqn{MeanValueMultiplierBis} can be expressed as a Puiseux series in the following way:
\[
\frac{3\sqrt{2 \pi (\sinh(2r) - 2r)}}{4\pi f(r)} = \left( \frac{1}{2} \sqrt{ \frac{3}{\pi} }  \right) r^{-3/2} + \left( \frac{1}{10} \sqrt{ \frac{3}{\pi} } \right) r^{1/2} + O(r^{5/2}).
\]
Since $r = \delta/ 2$, the bounds in \refeqn{MeanValueMultiplier} and \refeqn{MeanValueMultiplierBis} pick out the leading term of the series, with a bit of multiplicative error.

Now, we proceed to the proof. Set $C = 1.046$.
%
%
A bit of algebraic manipulation shows that the desired inequality \refeqn{MeanValueMultiplier} is equivalent to
\begin{equation}\label{Eqn:MeanValueAnalytic}
 \Phi(r)= 2  f(r)^2  C^2  - 3  r^3 (\sinh(2r) - 2r)   \geq 0 
  \qquad \text{on} \qquad r \in [0, 0.469].
\end{equation}
Since $f(r) = \cosh(r)\sin(\sqrt{2}r) - \sqrt{2}\sinh(r)\cos(\sqrt{2}r)$ is an analytic function, the entire function $\Phi(r)$ in \refeqn{MeanValueAnalytic} is analytic.
The $9$-th degree Taylor polynomial for $\Phi(r)$, centered at $r=0$, is
\[
P(r) = 4(C^2-1) r^6 - \tfrac{4}{5}(C^2+1) r^8.
\]
By Taylor's theorem with remainder, when $r \in [0,0.469]$ we have
\[
\Phi(r) - P(r) = \frac{\Phi^{(10)}(\rho)}{10!} r^{10}  ,
\qquad \text{for some }
\rho \in [0,0.469]
\text{ depending on } r.
\]
Using interval arithmetic in Sage \cite{FPS:Ancillary}, we verify that $\Phi^{(10)}(\rho)/10! \geq -0.085$ for all $\rho \in [0,0.469]$. Thus \[
\Phi(r) \geq 4(C^2-1) r^6 - \tfrac{4}{5}(C^2+1) r^8 - 0.085 r^{10},
\]
a function that  is easily seen to be non-negative for $r \in [0,0.469]$ because it factors into linear and quadratic terms.
This proves \refeqn{MeanValueAnalytic} and therefore \refeqn{MeanValueMultiplier}.

Inequality \refeqn{MeanValueMultiplierBis} is proved similarly, substituting $C = 1.001$.
\end{proof}

\begin{proposition}\label{Prop:PointwiseBound}
Fix $0<\delta\leq 0.938$. Let $M$ be a complete, finite volume hyperbolic 3-manifold and $\Sigma$ a geodesic link in $M$ with total length $\ell\leq\delta^2/17.11$. Let $M_t$ be a cone-manifold occurring along a deformation from $M - \Sigma$ to $M$. Let $p \in M_t^{\geq \delta}$.
Then
\[
 \| \omega(p) \| \leq  \frac{0.1822 \, \ell} {\delta^{5/2} }.
\]
\end{proposition}

\begin{proof}
We may assume without loss of generality that $t$ lies in the interior of the deformation interval $[0, (2\pi)^2]$, or in other words, we have cone angle $0 < \alpha < 2\pi$ along every component of $\Sigma$. Once we establish the desired bound on $\| \omega(p) \|$ for such times in the interior, the general case will follow by continuity.

For $p \in M_t^{\geq \delta}$, by \refdef{ThickThin} there is a round ball of radius $\delta/2$ centered at $p$, which is isometric to a ball in $\HH^3$. For each component $\sigma_j$ of $\Sigma$, let $r_j(\delta)$ be the tube radius of the $\delta$--thin tube about $\sigma_j$ in the metric $g_t$. By \refdef{DeltaThinTube}, this means that every point $q \in \bdy U_{r_j(\delta)}$ satisfies 
\[
\injrad(q) \leq \injrad(q, \, U_{r_j(\delta)}) = \delta / 2,
\]
where the inequality is \reflem{InjRadRelation}. 
Set ${\rr_-} = (r_1(\delta) - \delta/2, \ldots, r_n(\delta)-\delta/2)$.
By \refthm{BoundaryDelta}, the multi-tube $U_{\rr_-}$ is embedded. 
Moreover, $U_{\rr_-} \cap B_{\delta/2}(p) = \emptyset$.  Therefore, by \refeqn{Weitzenbock} and \refthm{BoundaryDelta}
\begin{equation}\label{Eqn:NormInBall}
\int_{B_{\delta/2}(p)} \| \omega \|^2 \, dV 
 \leq \int_{M - U_{\rr_-}} \| \omega \|^2 \, dV   
 =  b_{\rr_-} (\eta, \eta) \leq   
\left( \frac{\ell}{7.935 \, \delta} \right)^2 .
\end{equation}

The bound on the term in \refeqn{MeanValue} before the square root of the integral is bounded by \reflem{Puiseux}, particularly equation \refeqn{MeanValueMultiplier}.
Now, we plug the estimates of \refeqn{NormInBall} and \refeqn{MeanValueMultiplier} into \refeqn{MeanValue} to obtain
\[
\| \omega(p) \| \leq 
\left( \frac{ 1.046}{\delta^{3/2}} \,  \sqrt{ \frac{6}{\pi} }  \right)
\left( \frac{\ell}{7.935 \, \delta} \right) 
\leq  \frac{0.1822 \, \ell} {\delta^{5/2} }. \qedhere
\]
\end{proof}

We can now complete the proof of \refthm{Bilip}.

\begin{proof}[Proof of \refthm{Bilip}]
By \reflem{LBoundImpliesEllBound}, the hypothesis on $L$ implies the hypothesis on $\ell$, so we may assume the hypothesis on $\ell$. The existence of a cone deformation $(M, \Sigma, g_t)$ now follows from \refthm{ConeDefExists}.
Recall as well that in \refrem{OmegaUniqueness}, we made a canonical choice of harmonic form $\omega$ governing the family of cone-metrics $g_t$, and used this choice to define a natural identity map $\id \from (M-\Sigma, g_a) \to (M-\Sigma, g_b)$.

Now, we can check the bilipschitz estimate of the theorem.
Set $K = 0.1822 \, \ell \, \delta^{-5/2} $. Suppose that $W \in M_t^{\geq \delta}$ for all $t \in (a,b)$ and that $p \in W$. Then, by \refeqn{OmegaRestate} and \refprop{PointwiseBound}, we have
\[
\| \eta (p) \|
\leq 
\| \omega(p) \| \leq K.
\]
Set
\[
J = \exp \left( (2 \pi)^2 K \right) = \exp \left( (2\pi)^2 \cdot
\frac{0.1822 \, \ell} {\delta^{5/2} } \right)  \leq
\exp \left(  \frac{7.193 \, \ell} {\delta^{5/2} } \right) .
\]
By \reflem{BilipEta}, the bound $\| \eta(p) \| \leq K$ implies 
\[
 \bilip_p (g_a, g_b) \leq e^{ | b - a| K } \leq e^{ (2 \pi)^2 K }  = J,
\]
as claimed.
\end{proof}

\subsection{Corollaries and variations}
\refthm{Bilip} has the following pair of corollaries on effective bilipschitz bounds on drilling and filling. 
In both statements, $g_0$ denotes the complete hyperbolic metric on $M-\Sigma$, and $g_{4\pi^2}$ denotes the complete hyperbolic metric on $M$.

\begin{corollary}\label{Cor:EffectiveBBSpecial}
Fix $0<\delta\leq 0.938$ and $J>1$.
Let $M$ be a complete, finite volume hyperbolic $3$--manifold. Let $\Sigma \subset M$ be a geodesic link whose total length $\ell$
satisfies
\[
\ell \leq \min\left\{\frac{\delta^2}{17.11}, \, \frac{\delta^{5/2}\log(J)}{7.193}\right\}.
\]
Let $W \subset M$ be any submanifold such that $W\subset M_t^{\geq\delta}$ for all $t$. Then, for all $a,b \in [0, (2\pi)^2]$, the identity map
$\id\from (W, g_a) \to (W, g_b)$ is $J$--bilipschitz.
\end{corollary}

\begin{corollary}\label{Cor:EffectiveBBSpecialUp}
Fix any $0<\delta\leq 0.938$ and any $J>1$. Let $M$ be a complete, finite volume hyperbolic $3$--manifold and $\Sigma$ a geodesic link in $M$.
Suppose that in the cusped structure on $N = M-\Sigma$, the total normalized length $L$ of the meridians of $\Sigma$ satisfies
\[ 
L^2 \geq \max \left\{ \frac{107.6}{\delta^2}+14.41, \, 
\frac{45.20}{\delta^{5/2}\log(J)}+14.41 \right\}. 
 \]
Let $W \subset M$ be any submanifold such that $W\subset M_t^{\geq\delta}$ for all $t$. Then, for all $a,b \in [0, (2\pi)^2]$, the identity map
$\id\from (W, g_a) \to (W, g_b)$ is $J$--bilipschitz.
\end{corollary}

\begin{proof}
 This follows from \refthm{Bilip} and \reflem{LBoundImpliesEllBound}.
\end{proof}

We also have the following analogue of \refthm{Bilip}, with stronger hypotheses and a stronger bilipschitz estimate. This stronger statement will be used in \refsec{Margulis}.

\begin{theorem}
\label{Thm:BilipBis}
Fix $0 < \delta \leq 0.106$ and $1 < J \leq e^{1/5}$. 
Let $M$ be a complete, finite volume hyperbolic 3-manifold and $\Sigma$ a geodesic link in $M$. Suppose that in the complete structure on $M$, the total length of $\Sigma$ is bounded as follows:
\[ 
\ell \leq \frac{ \delta^{5/2} \log(J) }{3.324}
\quad \text{if} \quad \delta \leq 0.012,
\qquad \quad
\ell \leq \frac{ \delta^{5/2} \log(J) }{3.498}
\quad \text{if} \quad 0.012 < \delta \leq 0.106.
\]

Let $[a, b] \subset [0, (2\pi)^2]$ be an interval of time, and suppose that $W$ is a submanifold of $M$ such that $W \subset M_t^{\geq \delta}$ for all  $t \in (a, b)$. Then  the identity map
$\id \from (W, g_a) \to (W, g_b)$
is $J$--bilipschitz.
\end{theorem} 

\begin{proof}
First, suppose that $\delta \leq 0.012$. Since $\log(J) \leq 1/5$ and $5 \times 3.324 = 16.62$, our hypotheses are stronger than those of \refprop{BoundaryTinyDelta}.
The proof that $\id$ is $J$--bilipschitz  is almost identical to the proof of \refthm{Bilip}, with two small substitutions. 
Inside the proof of \refprop{PointwiseBound}, we replace \refeqn{NormInBall}, which uses \refthm{BoundaryDelta}, by the estimate of \refprop{BoundaryTinyDelta}: 
\begin{equation}\label{Eqn:NormInBallBis}
\int_{B_{\delta/2}(p)} \| \omega \|^2 \, dV 
 \leq \int_{M - U_{\rr_-}} \| \omega \|^2 \, dV   
=  b_{\rr_-} (\eta, \eta) \leq   \left( \frac{\ell}{16.432 \, \delta} \right)^2 .
\end{equation}
We also replace \refeqn{MeanValueMultiplier} by \refeqn{MeanValueMultiplierBis}. 
Now, using \refeqn{MeanValueMultiplierBis} and \refeqn{NormInBallBis} in place of  \refeqn{MeanValueMultiplier} and \refeqn{NormInBall}, \refprop{PointwiseBound} becomes
\begin{equation}\label{Eqn:PointwiseBoundBis}
\| \omega(p) \| \leq 
\left( \frac{ 1.001}{\delta^{3/2}} \,  \sqrt{ \frac{6}{\pi} }  \right)   \left( \frac{\ell}{16.432 \, \delta} \right)
\leq  \frac{0.08419 \, \ell} {\delta^{5/2} }. 
\end{equation}
Plugging \refeqn{PointwiseBoundBis} into the final part of the proof of \refthm{Bilip} gives
\[
 \bilip(g_a, g_b) \leq  \exp \left( (2\pi)^2 \cdot \frac{0.08419 \, \ell} {\delta^{5/2} } \right) \leq  \exp \left(  \frac{3.324 \, \ell} {\delta^{5/2} } \right) \leq J. 
\]

If $0.012 < \delta \leq 0.106$, the proof is again almost identical; we use \refprop{BoundaryMediumDelta} instead of \refprop{BoundaryTinyDelta} to get a slightly looser bound on $b_{\rr_-} (\eta, \eta)$.
\end{proof}

\section{Margulis numbers for cone-manifolds}\label{Sec:Margulis}

Our goal in this section is to give an effective estimate on the Margulis numbers of cone-manifolds that occur in the cone-deformations we have been studying. See  Theorems~\ref{Thm:MargulisConeMedConst} and~\ref{Thm:MargulisConeMfld} for effective statements in this vein. These estimates for cone-manifolds are used to control the Margulis number of the non-singular manifold $M$ at the end of the deformation, under hypotheses on either the drilled manifold $M - \Sigma$ (in \refthm{MargulisFilling}), or the filled manifold $M$ (in \refthm{MargulisDrilling}).

The proof of each of these results breaks into a topological statement and a geometric statement. The topological statement is \refthm{MargulisTopology}, which can be paraphrased as follows: so long as 
$\epsilon$ is a Margulis number for an initial manifold $M_0$, and so long as the $\epsilon$--thick part $M_0^{\geq \epsilon}$ stays $\delta$--thick in $M_t$ for every $t$, we learn that $\delta$ is a Margulis number for $M_t$. The geometric statement is \refthm{ThickStaysThick}: under strong hypotheses on length, the $\epsilon$--thick part $M_0^{\geq \epsilon}$ indeed stays almost $\epsilon$--thick in $M_t$ for all $t$. 
In fact, both the topological and the geometric statement require geometric hypotheses about $\ell$ or $L$, and rely on the estimates in the preceding sections.

\refthm{ThickStaysThick} has an additional application: it allows us to formulate a version of the bilipschitz \refthm{Bilip} whose hypotheses are only on the non-singular manifold at one end of a cone-deformation, without any pre-existing knowledge about intermediate cone manifolds. See \refthm{BilipEndpoints} for details.

\subsection{Tubes realizing injectivity radii}

The following theorem says that Margulis numbers in a cone-manifold $M_t$ are related to Margulis numbers in $M_0$, provided that we have set containment of the corresponding thin parts.

\begin{theorem}\label{Thm:MargulisTopology}
Fix 
$0 < \delta \leq \epsilon$, where $\delta < 0.9623$. Suppose $M$ is a complete, finite volume hyperbolic manifold, 
and $\Sigma = \sigma_1 \cup \ldots \cup \sigma_n$ is a geodesic link in $M$. 
Suppose that $\len(\sigma_j) \leq 0.0996$ for every $j$, while the total length of $\Sigma$ satisfies
\begin{equation}\label{Eqn:MargulisTopologyEll}
\ell = \len(\Sigma) \leq \min \left \{ 0.261 \delta, \: \frac{1}{2\pi} \haze \left( \frac{\delta + 0.1604 }{1.1227}  \right)
\right \}.
\end{equation}

Let $M_0$ be the complete metric on $M - \Sigma$, and assume that $\epsilon$ is a Margulis number for $M_0$. Suppose as well that for every $t$, we have
\begin{equation}\label{Eqn:ThinPartContainment}
(M_t^{\leq \delta} - \Sigma) \subset M_0^{< \epsilon}.
\end{equation}
Then $\delta$ is a Margulis number for $M_t$, for each $t$.
\end{theorem}

The hypotheses on $\delta$, $\len(\sigma_j)$, and $\ell$ in \refthm{MargulisTopology} match those of \reflem{DeltaTubeEmbeds}. 
Recall that the function in \refeqn{MargulisTopologyEll}, which expresses the upper bound on $\ell$ in terms of $\delta$, is depicted in \reffig{DeltaEmbeds} on page \pageref{Fig:DeltaEmbeds}.

The proof of \refthm{MargulisTopology} breaks up into several steps.
In \refprop{ImmersedTube}, we show that 
if $M_t$ is a cone-manifold occurring in the interior of our deformation, and $\injrad(x) \leq \delta/2$, then an appropriate subset of $M^{\leq \delta}$ is
 (loosely speaking) realized by a tube through $x$. This ``tube'' $U$ may be immersed rather than embedded, may be singular, and may be a horocusp. There may also be more than one such tube through $x$. In \reflem{MargulisTopologySingular}, we will show that, under the hypotheses of \refthm{MargulisTopology}, these tubes or cusps are in fact disjointly embedded in a singular cone-manifold $M_t$. This proves \refthm{MargulisTopology} for singular manifolds $M_t$, corresponding to a time parameter $t < (2\pi)^2$.
 
We complete the proof of \refthm{MargulisTopology} via a continuity argument. The function $\injrad_t(x)$ is not always continuous in $t$, but it comes close; the precise (and more subtle) continuity statement is established in \reflem{InjradContinuity}.
 
Recall from \refdef{Peripheral} that a non-trivial group element $\varphi \in \pi_1(M - \Sigma)$ is called peripheral if a loop representing $\varphi$ is freely homotopic into a cusp of $M - \Sigma$.

\begin{proposition}\label{Prop:ImmersedTube}
Fix 
$0 < \delta <  0.9623$. Let $M$ be a complete, finite volume hyperbolic manifold, 
and $\Sigma = \sigma_1 \cup \ldots \cup \sigma_n$ a geodesic link in $M$. 
Suppose that $\len(\sigma_j) \leq 0.0996$ for every $j$, while  $\ell = \len (\Sigma)$ satisfies \refeqn{MargulisTopologyEll}.

Let $M_t$ be a cone-manifold in the interior of the deformation from $M - \Sigma$ to $M$. 
Let $x \in M_t$ be a point such that $2 \injrad(x) \leq \delta$. Choose a lift $\widetilde x \in \bcover M_t$, the universal branched cover, let $\varphi \in \pi_1(M - \Sigma)$ be a group element guaranteed by \reflem{InjectivityLoop}, so that $d(\widetilde x, \varphi \widetilde x) = 2 \injrad(x)$, and let $G = C(\varphi)$ be the centralizer of $\varphi$ in $\pi_1(M - \Sigma).$
Then the following hold.
\begin{enumerate}
\item\label{Itm:TubeExists} $G$ stabilizes an open set $ \widetilde V  \subset \bcover M_t$, which is either a horoball or a regular neighborhood of a geodesic.
\item\label{Itm:LocalIsom} The quotient $V = \widetilde V / G$ admits a local isometry $f \from V \to M_t$. Thus $f$ is an immersed tube or immersed horocusp in $M_t$, as in \refdef{ImmersedTube}.
\item\label{Itm:VBig} We have $V^{\geq 1.51} \neq \emptyset$.
\item\label{Itm:PointOnBdy} There is a sub-tube or sub-horocusp $U \subset V$ and a point 
$y \in \bdy U$, 
such that $x = f(y)$. Furthermore, $\injrad(x) = \injrad(y,U)$.
\item\label{Itm:PeriphEmbeds} If $\varphi$ is peripheral, then $f \vert_{\overline U} \from \overline{U} \to M_t$ is an embedding
of a horocusp or singular tube. 
\end{enumerate}
\end{proposition}

\begin{proof}
Let $U_{\max}(\Sigma) = U_1 \cup \ldots \cup U_n$ be the maximal multi-tube about $\Sigma$, as in \refdef{MaximalTube}. Let $U_{n+1}, \ldots, U_m$ be horoball neighborhoods of the cusps of $M_t$ (if any), expanded until each $U_j$ bumps into a previously expanded cusp or tube. By \reflem{DeltaTubeEmbeds},
each tube $U_j$ has radius $R_j$, where
\[
R_j \geq \Rmin = h^{-1}(2\pi\ell) \geq 0.7555.
\]
By \refthm{MaxTubeInjectivity}, every point $z \in \bdy U_j$ satisfies
\begin{equation*}
2 \, \injrad(z, U_j)  > 1.1227 \tanh \Rmin  - 0.1604.
\end{equation*}
Furthermore, by \refrem{CuspTubeInjectivity}, this bound applies to both tubes and horocusps. Combining this with \refeqn{InjecStrictContain} in \reflem{DeltaTubeEmbeds}, we learn that
\begin{equation}\label{Eqn:TubeStabilizerLength}
2 \, \injrad(z, U_j)  > 1.1227 \tanh \Rmin  - 0.1604 \geq \delta
\end{equation}
for every point $z \in \bdy U_j$ on the boundary of a tube or horocusp.
 
\medskip

To begin proving the conclusions of the proposition, suppose first that $\varphi$ is peripheral. Then \reflem{PeripheralParabolic} says that $\varphi$ stabilizes either a horoball in $\bcover M_t$ or a singular geodesic $\bcover \sigma_j \subset \bcover \Sigma$, where $\bcover \Sigma$ is the preimage of $\Sigma$. In both cases, we will see that $x \in U_j$ for a tube or horocusp $U_j$ constructed above.

If $\varphi$ stabilizes a singular geodesic $\bcover \sigma_j \subset \bcover \Sigma \subset \bcover M_t$, then it stabilizes the universal branched cover $\widetilde U_j$ of some singular tube $U_j$. Alternately, $\varphi$ stabilizes the universal cover $\widetilde U_j \subset \bcover M_t$ of some horocusp $U_j$. In either case, we claim
that $\widetilde x \in \widetilde U_j$. This is because $\varphi$ moves $\widetilde x$ by distance $2\injrad(x) \leq \delta$, whereas \refeqn{TubeStabilizerLength} implies that $\varphi$ moves every point outside $\widetilde U_j$ by distance strictly greater than $\delta$. Thus $\widetilde x \in \widetilde U_j$, hence $x \in U_j$.

We can now construct the sets $U$ and $V$ claimed in the proposition. Let $\widetilde V$ be the maximal metric neighborhood of $\widetilde U_j$ that is disjoint from $\bcover \Sigma$, except possibly at the core of $\widetilde U_j$. Let $\widetilde U \subset \widetilde U_j$ be the proper sub-tube or sub-horoball defined by the property that $\widetilde x \in \bdy \widetilde U$.
Set $G = C(\varphi) \cong \ZZ \times \ZZ$ and consider the covering projection $\pi \from \widetilde V \to \widetilde V/G$.
Then we have a sequence of local isometries 
\begin{equation}\label{Eqn:ConstructLocalIsom}
V = \widetilde V / G \hookrightarrow \bcover M_t / G \longrightarrow M_t,
\end{equation}
whose composition we call $f$. Restricting attention to $U_j = \pi(\widetilde U_j)$, we recover the embedding $f \from U_j \hookrightarrow M_t$. Since $\overline U \subset U_j$, it follows that $f \vert_{\overline U}$ is an embedding as well. Since $\widetilde x \in \bdy \widetilde U$, we have a point $y = \pi(\widetilde x) \in \bdy U$ such that $f(y) = x$. Furthermore, since $\varphi \in G$, we have
\[
2 \, \injrad(x) = d(\widetilde x, \varphi \widetilde x) = 2 \, \injrad(y, U).
\]
This proves all the properties claimed of $U$ and $V$, except for \refitm{VBig}. We will check \refitm{VBig} after verifying the corresponding property in the non-peripheral case in \refclaim{VBigInjec}.

\medskip

Next, suppose that $\varphi$ is non-peripheral. Then \reflem{PeripheralParabolic} says that $\varphi$ stabilizes 
a geodesic axis $\widetilde \beta \subset \bcover M_t$, not contained in the singular locus, which covers a closed geodesic $\beta \subset M_t$. 
Observe that $\beta$ cannot be entirely contained in a singular tube $U_j$, because the only closed geodesic in $U_j$ is the singular core $\sigma_j$.

Furthermore, for every $U_j$, we have $d(\sigma_j, \bdy U_j)  \geq \Rmin \geq 0.7555.$ Thus if $\beta \cap U_j$ reaches the singular core $\sigma_j$, then $\len(\beta \cap U_j) \geq 2\cdot 0.7555 = 1.511$. However,
\[
\len(\beta \cap U_j) \leq \len(\beta) \leq d(\widetilde x, \varphi \widetilde x)
\leq \delta < 0.9623.
\]
So any geodesic arc $\beta \cap U_j$ cannot reach the singular core $\sigma_j$, implying that $\beta \cap \Sigma = \emptyset$. Consequently, $\widetilde \beta$ is disjoint from the singular locus $\bcover \Sigma \subset \bcover M_t$.

Let $\widetilde V \subset \bcover M_t$ be the maximal tubular neighborhood of $\widetilde \beta$ that is disjoint from the singular locus $\bcover \Sigma$. This neighborhood has finite radius because $\bcover M_t$ contains singular points. Then $\widetilde V$ is stabilized by the maximal cyclic subgroup $G = C(\varphi) \cong \ZZ$, where $\widetilde \beta / G = \beta$. Note that $\varphi \in G$.
Define a model tube $V = \widetilde V / G$. Then, as in \refeqn{ConstructLocalIsom}, we have a sequence of local isometries 
\[
V = \widetilde V / G \hookrightarrow \bcover M_t / G \longrightarrow M_t,
\]
whose composition we call $f$. 
Now, we have to show that the immersion $f \from V \to M_t$ has the properties claimed in the proposition.

\begin{claim}\label{Claim:VBigInjec}
We have $V^{\geq 1.51} \neq \emptyset$.
\end{claim}
  
Recall that $\widetilde V$ was defined to be the maximal open tube about $\widetilde \beta$ that is disjoint from $\bcover \Sigma$. Thus there is a point $\widetilde z \in \bdy \widetilde V \cap \bcover \Sigma$. Let $\widetilde z'$ be a closest translate of $\widetilde z$ by a non-trivial element of $G$. Then $\widetilde z, \widetilde z'$ must lie on distinct singular geodesics in $\bcover \Sigma$ that cover the same component $\sigma_j \subset \Sigma$. Since the tube $U_j$ about $\sigma_j$ must have radius $R_j \geq 0.7555$, it follows that
 \[
 d(\widetilde z, \widetilde z') \geq 2 R_j \geq 2 \Rmin > 1.51.
 \]
 Compare \refclaim{FarEndpoints} for a very similar setup.

Let $\pi \from \bcover M_t \to \bcover M_t / G$ be the covering projection, and  let $z = \pi(\widetilde z) = \pi(\widetilde z') \in \bdy V$.
Since $\widetilde z, \widetilde z'$ are a pair of closest lifts of $z$ under $G$, we have
\[
\injrad(z,V)>1.51.
\]
In particular, $V^{\geq 1.51} \neq \emptyset$. (The same argument applies to the tube or horocusp $V$ in the peripheral case, completing the proof of the proposition in that case.)
 
In the non-peripheral case, we have now checked that the immersion $f \from V \to M_t$ satisfies properties \refitm{TubeExists}--\refitm{VBig}. It remains to show that there is a sub-tube $U \subset V$ satisfying \refitm{PointOnBdy}. Note that \refitm{PeriphEmbeds} is vacuous for non-peripheral elements.
 
\begin{claim}
There exists a tube $U$ with $\overline U \subset V$ and a point $y \in \bdy U$, such that $x = f(y)$ and $\injrad(x) = \injrad(y, U)$. 
\end{claim}
 
Recall that $\widetilde x \in \bcover M_t$ and $\varphi \in \pi_1(M - \Sigma)$ were chosen to have the property that
 $d(\widetilde x, \varphi \widetilde x) = 2 \injrad(x) \leq \delta$. In particular, $\varphi \widetilde x$ is a closest translate of $\widetilde x$ in $\bcover M_t$. 
Recall as well that $\varphi \in G$. Letting $y = \pi(\widetilde x) = \pi(\varphi \widetilde x) \in \bcover M_t / G$, we get
\[
2 \injrad(y) = d(\widetilde x, \varphi \widetilde x) = 2 \injrad(x) \leq \delta < 1,
\]
where $\injrad(y)$ denotes the injectivity radius of $y$ in the cone-manifold $\hat M_t/G$.

On the other hand, by \refclaim{VBigInjec},
$V \subset \bcover M_t / G$
extends out to include points of injectivity radius $1.51$ in $\hat M_t/G$. Thus $y \in V$.

Let $U \subset V$ be the model tube such that $y \in \bdy U$. Then
\[
\injrad(x) = \injrad(y) = \injrad(y,V) = \injrad(y,U),
\]
because the realizing isometry $\varphi$ belongs to $G=\pi_1 V = \pi_1 U$.

By construction, the local isometry $f \from V \to M_t$ is a restriction of the covering map $\bcover M_t / G \to M_t$. Thus $f(y) = f \circ \pi(\widetilde x) = x$, completing the proof of \refitm{PointOnBdy}.
\end{proof}

\begin{lemma}\label{Lem:MargulisTopologySingular}
Fix 
$0 < \delta \leq \epsilon$, where $\delta < 0.9623$. Let $M$ be a complete, finite volume hyperbolic manifold, and let $\Sigma = \sigma_1 \cup \ldots \cup \sigma_n$ is a geodesic link in $M$. 
Suppose that $\len(\sigma_j) \leq 0.0996$ for every $j$, while the total length of $\Sigma$ satisfies \refeqn{MargulisTopologyEll}.

Let $M_0$ be the complete metric on $M - \Sigma$, and assume that $\epsilon$ is a Margulis number for $M_0$. For a cone-manifold $M_t$ in the interior of the cone-deformation from $M_0$ to $M_{4\pi^2}$, suppose that
\[
(M_t^{\leq \delta} - \Sigma) \subset M_0^{< \epsilon}.
\]
Then $\delta$ is a Margulis number for $M_t$.
\end{lemma}

\begin{proof}
Let $x \in M_t$ be a point such that $2\injrad(x) = \delta$. Choose a lift $\widetilde x \in \bcover M_t$. By \reflem{InjectivityLoop}, there is a group element $\varphi \in \pi_1(M - \Sigma)$ such that $d(\widetilde x, \varphi \widetilde x) = \delta$. By \refprop{ImmersedTube}, $\varphi$ defines a closed model tube or horocusp $\overline U$ with an isometric immersion $f\from \overline U \to M_t$, such that $x \in f(\bdy U)$. 

Note that the isometry $\varphi$ may not be unique. The proof amounts to checking that  $\overline U$ is unique and that $f$ is an embedding. 

If $\varphi$ is peripheral, \refprop{ImmersedTube} tells us that $U$ is a singular tube or horocusp and that $f$ is an embedding.
Furthermore,
there exists $y\in \bdy\overline{U}$ such that $x=f(y)$ and
$\injrad(x) = \injrad(y, U)$. Since local isometries can only reduce the injectivity radius (by \reflem{InjRadRelation}) and since $y \in \bdy \overline{U}$, every point $z \in \overline U$ satisfies
\[
\injrad(f(z)) \leq \injrad(z, U) \leq \injrad(y, U) = \injrad(x) = \delta/2.
\]
Then $f(\overline U) \subset M_t^{\leq \delta}$, hence by hypothesis $f(\overline U) - \Sigma \subset M_0^{< \epsilon}$.
Let $W$ be the component of $M_0^{< \epsilon}$ containing $f(\overline U) - \Sigma$. Since $U$ is a horocusp or singular tube, it follows that $W$ must be a horocusp of $M_0$. Here, we are using the hypothesis that $\epsilon$ is a Margulis number for $M_0$.

Suppose that $\varphi' \in \pi_1(M - \Sigma)$ also has the property that $d(\widetilde x, \varphi' \widetilde x) = \delta$. Let $U'$ be the tube or horocusp associated to $\varphi'$, with an isometric immersion $f' \from U' \to M_t$ and with $x \in f'(\bdy U')$.
Since $x \in \overline{U'} \cap W$, the hypotheses of the lemma imply $(f(\overline{U'}) - \Sigma) \subset W$. In particular, $\varphi' \in f'_* \pi_1(U') \subset \pi_1(W)$, where $\pi_1(W) \cong \ZZ \times \ZZ$ is a peripheral subgroup of $\pi_1(M - \Sigma)$.
It follows that $U$ and $U'$ are either both horocusps or both tubes about the same component of $\Sigma$, hence $U'$ is also embedded. Since $x \in \bdy U$ and $x \in \bdy U'$, it follows that $\overline{U} = \overline{U'}$ is the full component of $M_t^{\leq \delta}$ containing $x$.

One particular consequence of the above paragraph is that if $\varphi$ is peripheral, then $\varphi'$ must also be peripheral.

\smallskip

If $\varphi$ is non-peripheral, \refprop{ImmersedTube} tells us that there is a non-singular immersed tube $f \from U \to M_t$ and a point $y \in \bdy U$ such that $x = f(y)$ and $\injrad(x) = \injrad(y, U)$. 
Since isometric immersions can only reduce injectivity radius, every point $z \in \overline U$ satisfies
\[
\injrad(f(z)) \leq \injrad(z, U) \leq \injrad(y, U) = \injrad(x) = \delta/2.
\]
Thus  $f(\overline U) \subset (M_t^{\leq \delta} - \Sigma) \subset M_0^{< \epsilon}$.  
Let $W$ be the component of $M_0^{< \epsilon}$ containing $f(\overline U)$. Since the core of $f(U)$ is a non-singular geodesic $\beta \subset M_t$, 
it follows that $W$ must be a non-singular, embedded tube in $M_0$. Recall that $\epsilon$ is a Margulis number for $M_0$.

We check that $f$ is an embedding by considering the universal branched cover $\bcover M_t$. Recall from the construction of \refprop{ImmersedTube}, specifically from \refeqn{ConstructLocalIsom}, that the universal cover $\widetilde U$ is identified with a tubular neighborhood of a geodesic axis $\widetilde \beta \subset \bcover M_t$, and that $\pi_1(U) = G = C(\varphi)$, the stabilizer of $\widetilde \beta$.  Let $\widetilde W \subset \bcover M_t$ be a component of the preimage of $W$ such that $\widetilde U \subset \widetilde W$. Viewed in the complete hyperbolic metric of $\widetilde{M-\Sigma} = \widetilde{M_0} = \HH^3$, the component $\widetilde W$ is a tubular neighborhood of a geodesic, whose translates by elements of $\pi_1(M - \Sigma)$ are either disjoint from $\widetilde W$ or coincide with $\widetilde W$. Thus a translate $\eta \widetilde U$ for $\eta \in \pi_1(M - \Sigma)$ is either disjoint from $\widetilde W$ or contained in $\widetilde W$. In the first case, we have $\eta \widetilde U \cap \widetilde U = \emptyset$, and in fact $\eta (\bdy \widetilde U) \cap \bdy \widetilde U = \emptyset$. In the second case, $\eta$ must stabilize the endpoints of $\widetilde \beta$, in which case $\eta \in G$ stabilizes $\widetilde \beta$ and also $\widetilde U$. Thus $\widetilde U \cup \bdy \widetilde U$ is precisely invariant under the action of $\pi_1(M - \Sigma)$, which means the quotient tube $\overline U$ embeds in $M_t$.

Finally, suppose that $\varphi' \in \pi_1(M - \Sigma)$ also has the property that $d(\widetilde x, \varphi' \widetilde x) = \delta$. Let $U'$ be the tube  associated to $\varphi'$, with an isometric immersion $f' \from U' \to M_t$. We already checked that the peripheral and non-peripheral cases cannot overlap, so $U'$ must also be a non-singular tube. By the same argument as above, $f'$ must be an embedding, hence we consider $\overline U$ and $\overline{U'}$ to be subsets of $M_t$. As above, all of $\overline {U'}$ must be $\delta$--thin in $M_t$, hence $\overline {U'} \subset W \subset M_0^{< \epsilon}$. Since $x \in \overline U \cap \overline {U'} \subset W$, both $\pi_1(U)$ and $\pi_1(U')$ are subgroups of $\pi_1(W) \cong \ZZ$. Furthermore, both $\pi_1(U)$ and $\pi_1(U')$ are generated by primitive elements of $\pi_1(M - \Sigma)$, hence the two generators of $\ZZ$ must coincide up to inverses, hence the cores of $U$ and $U'$ map to the same geodesic $\beta \subset M_t$. Since $x \in \bdy U$ and $x \in \bdy U'$, it follows that $\overline U = \overline{U'}$ is the full component of $M_t^{\leq \delta}$ containing $x$.
\end{proof}

\reflem{MargulisTopologySingular} establishes the conclusion of \refthm{MargulisTopology} for all $t < (2\pi)^2$, when the cone-manifold $M_t$ is actually \emph{singular}. To finish the proof of \refthm{MargulisTopology}, we need a continuity argument as $t \to (2\pi)^2$. This is somewhat subtle, as injectivity radius can be discontinuous as a function of $t$ precisely when $t\to(2\pi)^2$.

We fix the following notation throughout. 
Let $M_t$ be a cone-manifold occurring along a deformation from $M-\Sigma$ to $M$, with metric $g_t$. Let $\injrad_t(x)$ denote the injectivity radius of $x$ in the $g_t$ metric, and let $d_t(\cdot, \cdot)$ denote distance in the $g_t$ metric.

\begin{lemma}\label{Lem:InjradContinuity}
Fix 
$0 < \delta <  0.9623$. Suppose $M$ is a complete, finite volume hyperbolic manifold, 
and $\Sigma = \sigma_1 \cup \ldots \cup \sigma_n$ is a geodesic link in $M$. 
Suppose that $\len(\sigma_j) \leq 0.0996$ for every $j$, while  $\ell = \len (\Sigma)$ satisfies \refeqn{MargulisTopologyEll}.

Consider the cone-deformation $M_t$ from $M - \Sigma$ to $M$.
Then, for every $x \in M$ and every $b < (2\pi)^2$,
we have
\[
\lim_{t \to b} \injrad_t(x) = \injrad_b(x).
\]
Furthermore, for $b = (2\pi)^2$,
\[
\left( \! d_b(x, \Sigma) \geq \frac{\delta}{2} \quad \text{or} \quad  \lim_{t \to b^-} \injrad_t(x) \geq \frac{\delta}{2} \right)
\quad \text{implies} \quad
\lim_{t \to b^-} \injrad_t(x) = \injrad_b(x).
\]
\end{lemma}

\begin{proof}
By \refdef{Injectivity},  $\injrad_t(x)$ varies continuously under a continuous change in cone metric when the cone-manifold stays singular, that is, $t < (2\pi)^2$. This proves the first assertion of the lemma. 

As $t \to b = (2\pi)^2$ and $\alpha \to 2\pi$, a discontinuity can arise in the following restricted way. If the cone angle on $\Sigma$ is $\alpha < 2\pi$,  a non-singular ball about $x$ cannot have radius larger than $d_t(x, \Sigma)$, hence $\injrad_t(x) \leq d_t(x, \Sigma)$. For $\alpha \geq \pi$, and for points sufficiently close to $\Sigma$, $\injrad_t(x)$ can in fact be equal to $d_t(x, \Sigma)$. As $t \to b$ and $\alpha \to 2\pi$, the link $\Sigma$ becomes non-singular, allowing $\injrad_b(x)$ to suddenly become larger than  $d_b(x, \Sigma)$ at time $b=(2\pi)^2$. Thus the ``furthermore'' statement also holds automatically, unless $x$ is a point satisfying
\[
d_b(x,\Sigma) = \lim_{t \to b^-} d_t(x,\Sigma) = \lim_{t \to b^-} \injrad_t(x).
\]

Suppose, for a contradiction, that $x \in M$ is a point that satisfies 
\begin{equation}\label{Eqn:ContinuityContrad}
d_b(x, \Sigma) = \lim_{t \to b^-} \injrad_t(x) \geq \delta/2.
\end{equation}
Then the distance $h= d_b(x,\Sigma)$ is realized by a geodesic segment $\beta$ (in the non-singular $g_b$ metric) from $x$ to some component $\sigma_j \subset \Sigma$. Let $y \in \beta$ be the point such that $d_b(y, \sigma_j) = \delta/2$, hence $d_b(x,y) = h - \delta/2$. We will obtain a contradiction by estimating $\injrad_t(y)$ for $t$ near $b$, in two different ways. First, \reflem{LowerBoundEasy} implies $\injrad_t(y) \geq \injrad_t(x) - d_t(x,y)$ for every $t$. Taking limits, we obtain
\begin{align}
\lim_{t \to b^-} \injrad_t(y) 
& \geq \lim_{t \to b^-} \injrad_t(x) - \lim_{t \to b^-} d_t(x,y) \nonumber \\
& = h - (h - \delta/2) \nonumber \\
& = \delta/2. \label{Eqn:HalfContradiction}
\end{align}

On the other hand, since $d_t(y,\sigma_j)$ is a continuous function of $t$, it must be the case that for all  $t < b$ sufficiently close to $b$, we have $d_t(y, \sigma_j) \leq 1.0006 ( \delta/2)$. For all such $t$, \reflem{DeltaTubeEmbeds} says that there is an embedded $\delta$--thin tube $U_t = U_t(\sigma_j)$ about $\sigma_j$, of radius $r_j(\delta) \geq 1.001 ( \delta/2)$ in the $g_t$ metric. Thus $y \in U_t$, and furthermore,
\[
d_t(y, \bdy U_t) \geq 0.0004(\delta/2) = 0.0002 \delta.
\]
By \reflem{DeltaTubeEmbeds}, $U_t$ is a $\delta$--thin tube, meaning $\injrad_t(z, U_t) = \delta/2$ for every $z \in \bdy U_t$. Define $\delta' < \delta$ by the property that
\[
0.0002 \delta = \arccosh \sqrt{ \frac{\cosh \delta - 1}{\cosh \delta' -1}    }.
\]
Then the upper bound of \refthm{EffectiveDistLog3} implies that $\injrad_t(y,U_t) \leq \delta'/2$. Combining this with  \reflem{InjRadRelation}, we obtain
\begin{equation}\label{Eqn:OtherHalfContradiction}
\injrad_t(y) \leq \injrad_t(y,U_t) \leq \delta'/2
\qquad 
\text{for all $t < b$ sufficiently close to $b$.}
\end{equation}
But now  \refeqn{OtherHalfContradiction} contradicts \refeqn{HalfContradiction}, because $\delta' < \delta$. Thus no point $x \in M$ can satisfy
 \refeqn{ContinuityContrad}, completing the proof.
\end{proof}

We can now complete the proof of \refthm{MargulisTopology}.

\begin{proof}[Proof of \refthm{MargulisTopology}]

Given \reflem{MargulisTopologySingular}, it remains to prove that $\delta$ is a Margulis number for $M_{4\pi^2}$. 

Set $b = 4\pi^2$. For every $t<b$, and every component $\sigma_j \subset \Sigma$, \reflem{DeltaTubeEmbeds} gives an embedded $\delta$--thin tube $U_t = U_t(\sigma_j)$, whose radius (in the $g_t$ metric) is at least $(1.001)\delta/2$. 
By  \refprop{ImmersedTube}, every point 
$x_t \in \bdy U_t$ satisfies $\injrad_t(x_t) = \injrad_t(x_t, U_t) = \delta/2$.
By \reflem{DeltaTubeEmbeds}, the tubes about different components of $\Sigma$ are disjointly embedded.
As $t \to b$, each tube $U_t(\sigma_j)$ converges (in the Hausdorff metric) to a tube $U_b(\sigma_j)$. Every convergent sequence of points $\{ x_t \in \bdy U_t \}$ limits to a point $x_b \in \bdy U_b$, with $d_b(x_b,  \sigma_j) = d_b(x_b, \Sigma) > \delta/2$. Thus, by  \reflem{InjradContinuity},
\[
\injrad_b(x_b) = \lim_{t \to b^-} \injrad_t(x_b) = \lim_{t \to b^-} \injrad_t(x_t) = \delta/2.
\]
We conclude that for every $\sigma_j \subset \Sigma$, there is a tube $U_b(\sigma_j) \subset M_b^{\leq \delta}$, such that $\bdy U_b(\sigma_j)$ consists of points where $\injrad_t$ is continuous in $t$. By \reflem{InjradContinuity}, all points outside these tubes also have the property that $\injrad_t$ is continuous in $t$. Thus we may apply continuity arguments outside the multi-tube $U_b(\Sigma)$.

Recall that for every $t$, we have \(
(M_t^{\leq \delta} - \Sigma) \subset M_0^{< \epsilon}.
\)
Since $M_0^{< \epsilon}$ is open and $\bigcup_t M_t^{\leq \delta}$ is closed, by continuity, there is a value $\delta_+ > \delta$ such that 
\[
(M_t^{\leq \delta_+} - \Sigma) \subset M_0^{< \epsilon}.
\]
Thus \reflem{MargulisTopologySingular} implies that $\delta_+$ is a Margulis number for $M_t$ for all $t < b$. For all $t < b$ sufficiently close to $b$, we have $M_b^{\leq \delta} \subset M_t^{\leq \delta_+}$. By choosing $\delta_+$ sufficiently close to $\delta$ and $t$ sufficiently close to $b$, we can ensure that every component of $M_t^{\leq \delta_+}$ contains a component of $M_t^{\leq \delta}$.

With this setup, let $U_b$ be an arbitrary component of $M_b^{\leq \delta}$. We will see that $U_b$ is a tube or horocusp by showing that 
$U_b = \lim_{t \to b} U_t$, where $U_t \subset M_t^{\leq \delta}$ is a $\delta$--thin tube or cusp in the $g_t$ metric.

For $t$ close to $b$, our chosen component $U_b$ is contained in a tube or horocusp component of $M_t^{\leq \delta^+}$, which contains a $\delta$--thin tube or horocusp $U_t \subset M_t^{\leq \delta}$. As $t \to b$, these tubes or horocusps $U_t$ converge in the Hausdorff topology to $U_b$. Note that disjoint components $U_t, U_t' \subset M_t^{\leq \delta}$ cannot collide as $t \to b$ because they are contained in disjoint components of $M_0^{<\epsilon}$. Similarly, a component $U_t$ cannot collide into itself because distinct lifts of $U_t$ to $\bcover M_b$ are contained in disjoint preimages of a component of $M_0^{<\epsilon}$. Thus every $U_b$ is an embedded tube or horocusp.
\end{proof}

\subsection{Thick parts stay almost as thick}

Next, we show that under strong hypotheses on the length $\ell$, the thick part of a cone-manifold $M_a$ stays almost as thick in every other $M_t$. This will enable us to apply \refthm{MargulisTopology} and control Margulis numbers. We begin with the following strengthened version of \refthm{EffectiveDistLog3}, which applies to an entire cone-manifold instead of a tube.

\begin{proposition}\label{Prop:DistThinParts}
Fix $0 < \epsilon \leq \log 3$, and $0 < \delta \leq \epsilon^2 / 7.256$.  Suppose $M$ is a complete, finite volume hyperbolic manifold, 
and $\Sigma$ is a geodesic link in $M$. 
Suppose $\ell = \len (\Sigma) \leq 0.261 \delta$.
Then, for all $t$,
\[
d(M_t^{\leq \delta}, M_t^{\geq \epsilon}) \geq \arccosh \left( \frac{\epsilon}{\sqrt{7.256 \delta}}  \right) - 0.1475.
\]
Furthermore, if $\epsilon \leq 0.3$, then
\[
d(M_t^{\leq \delta}, M_t^{\geq \epsilon} ) \geq \arccosh \left( \frac{\epsilon}{\sqrt{7.256 \delta}} \right) - 0.0424.
\]
\end{proposition}

\begin{proof}
We will use \refprop{ImmersedTube}. To check the hypotheses of that proposition, observe that our hypotheses require $0 < \delta \leq (\log 3)^2/7.256 < 0.1664$. For $\delta$ in this range, Equation~\refeqn{MargulisTopologyEll} becomes the simpler statement $\ell \leq 0.261 \delta$, which is what we require here. Similarly, every component $\sigma_j \subset \Sigma$ must have
 $\len(\sigma_j) \leq 0.261 \cdot 0.1664 < 0.0996$. Thus \refprop{ImmersedTube} applies under our hypotheses.

Now, let $x \in M_t^{\leq \delta}$. By \refprop{ImmersedTube}, there is an immersed tube (or immersed cusp) $f \from V \to M_t$, such that $x = f(y)$ for some point $y \in V$, and
\[
\injrad(y, V) = \injrad(x).
\]
Furthermore, $V^{\geq 1.51} \neq \emptyset$, hence $V^{\geq \epsilon} \neq \emptyset$.  By \refthm{EffectiveDistLog3}, we have
\[
d(V^{\leq \delta}, V^{\geq \epsilon}) \geq  \arccosh \left( \frac{\epsilon}{\sqrt{7.256 \delta}}  \right) - 0.1475 =: h.
\]
Consequently, every point $y' \in V$ such that $d(y,y') < h$ must lie in $V^{< \epsilon}$.

It follows that $f(V^{< \epsilon})$ contains the $h$--neighborhood of $x=f(y)$. Thus every point $x' \in M_t$ such that $d(x,x') < h$ must be the image of some $y' \in V^{< \epsilon}$, hence \reflem{InjRadRelation} gives
\[
2 \injrad(x') \leq 2 \injrad(y', V) < \epsilon.
\]
We conclude that every point of $M_t^{\geq \epsilon}$ lies further than $h$ from $x \in M_t^{\leq \delta}$.

Finally, if $\epsilon \leq 0.3$, we can use \refthm{EffectiveDistTubes} instead of \refthm{EffectiveDistLog3}, and repeat the same argument with $h = \frac{\epsilon}{\sqrt{7.256 \delta}} - 0.0424$.
\end{proof}

\begin{theorem}\label{Thm:ThickStaysThick}
Fix $0 < \epsilon \leq \log 3$ and $1 < J \leq e^{1/5}$.
Let $M$ be a complete, finite volume hyperbolic 3-manifold and $\Sigma \subset M$ a geodesic link. Suppose that  $\ell = \len(\Sigma)$ is bounded as follows:
\begin{equation}\label{Eqn:ThickStaysThickHypoth}
\ell 
 \leq \frac{\epsilon^5 \log J}{471.5 \, J^5 \cosh^5 (J \epsilon / 2 + 0.0424)} \quad \text{if} \quad 0 < \epsilon \leq 0.3,
 \end{equation}
or
\begin{equation}\label{Eqn:ThickStaysThickHypothLog3}
\ell 
 \leq \frac{\epsilon^5 \log J}{496.1 \, J^5 \cosh^5 (J \epsilon / 2 + 0.1475)} \quad \text{if} \quad 0.3 \leq \epsilon \leq \log 3.
\end{equation}
Then, for every $a, t \in[0,(2\pi)^2]$, the manifolds $M_a$ and $M_t$ in the deformation from $M - \Sigma$ to $M$ satisfy
\begin{equation}\label{Itm:ThickStaysThickish}
 M_a^{\geq \epsilon} \subset M_t^{> \epsilon/J}.
\end{equation}
Moreover, let $B$ be a closed ball of radius $\epsilon/2$ in the $g_a$ metric about a point $p \in M_a^{\geq \epsilon}$. Then
$\bilip_B(g_a, g_t) < J$.
\end{theorem}

\begin{figure}
\begin{overpic}[width=3in]{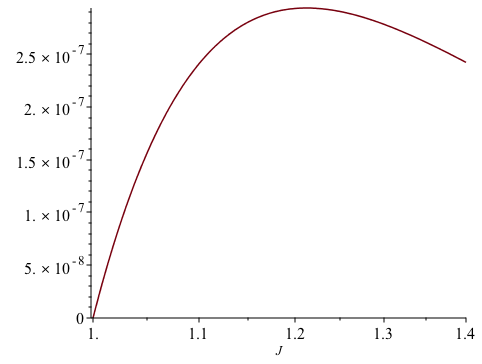}
\end{overpic}
\caption{The function of $J$  that provides an upper bound on $\ell$ in Equation~\refeqn{ThickStaysThickHypoth}, for the value $\epsilon = 0.292$. For any fixed $\epsilon$, this function has a global maximum when $J \leq e^{1/5}$.}
\label{Fig:FunctionOfJ}
\end{figure}

\refthm{ThickStaysThick} is reminiscent of a result of Brock and Bromberg  \cite[Theorem~6.11]{brock-bromberg:density}. Furthermore, the crawling argument employed in the proof below is inspired by the one used in \cite[Theorem~6.11]{brock-bromberg:density}. However, the statement of \refthm{ThickStaysThick} is stronger than that of \cite[Theorem~6.11]{brock-bromberg:density} in two distinct ways. First, following the theme of this paper, \refthm{ThickStaysThick} is effective. Second, \refthm{ThickStaysThick} provides better control over quantifiers. Brock and Bromberg's theorem says that for all sufficiently small $\epsilon$, \emph{there exists} $\epsilon' < \epsilon$ such that $M_t^{\leq \epsilon'} \subset M_0^{< \epsilon}$. Meanwhile, \refthm{ThickStaysThick} says that for all  $\epsilon \leq 0.3$, and \emph{for all} $\epsilon' = \epsilon / J < \epsilon$, we have $M_t^{\leq \epsilon'} \subset M_0^{< \epsilon}$,  provided $\ell$ is sufficiently short.

The seemingly artificial hypothesis $J \leq \exp(1/5)$ in \refthm{ThickStaysThick}  is justified as follows. For fixed $\epsilon$, the function in \refeqn{ThickStaysThickHypoth} or \refeqn{ThickStaysThickHypothLog3} is \emph{not} monotonic in $J$: it starts out at $0$ when $J = 1$, rises to a maximum, and then decreases toward $0$. Note that the function factors as $(\log (J)/ J^5)$ times a term that is decreasing in $J$. Hence the maximum of the function, corresponding to the mildest hypotheses on $\ell$, always occurs when $\log J \leq 1/5$.
(See \reffig{FunctionOfJ}, and compare \reflem{gJBehavior}.) For $J > \exp(1/5)$, the theorem requires stronger hypotheses on $\ell$ but produces a weaker conclusion, hence it makes sense to exclude those values.

\begin{proof}[Proof of \refthm{ThickStaysThick}]
For most of the proof, suppose that $\epsilon \leq 0.3$. This means we are working under the hypothesis \refeqn{ThickStaysThickHypoth}. 

Fix an arbitrary $a \in [0, (2\pi)^2]$. For the length of the proof, we will treat $a$ as a constant and $t$ as a variable.  Set
\begin{equation}\label{Eqn:DeltaDef}
\delta = \frac{(\epsilon/J)^2}{7.256 \cosh^2 (J \epsilon/2 + 0.0424)} \, ,
\quad \text{so that} \quad
\arccosh \left( \frac{\epsilon/J}{\sqrt{7.256\, \delta}} \right) - 0.0424 = J \epsilon / 2 .
\end{equation}
It follows that $\delta \leq 0.012$. For this value of $\delta$, we will actually prove the closely related condition 
\begin{equation}\label{Eqn:ThinPartSeparation}
d_a (M_t^{\leq \delta}, \, M_a^{\geq \epsilon}) > \epsilon/2 .
\end{equation}
As we shall see at the end of the proof, equation~\refeqn{ThinPartSeparation} quickly implies \refitm{ThickStaysThickish}.

Let $I$ be the maximal sub-interval of $[0,(2\pi)^2]$, containing $a$, such that \refeqn{ThinPartSeparation} holds for all $t \in I$. First, we check that $I$ is non-empty. This follows from \refprop{DistThinParts}: 
\[
d_a(M_a^{\leq \delta}, M_a^{\geq \epsilon}) \geq \arccosh \left( \frac{\epsilon/J}{\sqrt{7.256\, \delta}} \right) - 0.0424 = J \epsilon / 2 > \epsilon / 2.
\]
Thus $a \in I$, hence $I$ is nonempty. Also, $I$ is open because \refeqn{ThinPartSeparation} involves a strict inequality, hence is an open condition. 
We will show $I$ is closed, which will imply that $I = [0,(2\pi)^2]$.

Consider what can be said about $\overline{I}$. Let $b=\sup I$. \reflem{InjradContinuity} tells us that if $x \in M$ satisfies $\lim_{t \to b^-} \injrad_t(x) \geq \delta/2$, then $\injrad_b(x) \geq \delta/2$ as well. A stronger form of continuity holds for $\inf I$. In other words, \reflem{InjradContinuity} and the definition of $I$ imply
\begin{equation}\label{Eqn:WeakSeparation}
d_a (M_t^{\leq \delta}, \, M_a^{\geq \epsilon}) \geq \epsilon/2  \quad \mbox{for all $t \in \overline{I}$.}
\end{equation}

\begin{figure}
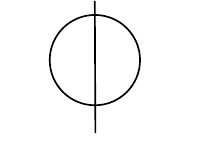
\hspace{.8in}
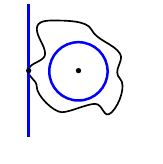
\caption{The proof of \refthm{ThickStaysThick}. The left and right panels illustrate how the region near $q$ looks in the $g_a$ metric and $g_b$ metric, respectively. Black objects are in the $g_a$ metric and blue objects are in the $g_b$ metric. 
}
\label{Fig:EpsilonDeltaTubes}
\end{figure}

Suppose, for a contradiction, that $I$ is not closed, hence either $\sup I \notin I$ or $\inf I \notin I$. We start by handling the supremum $b = \sup I$.
Suppose, in contradiction to \refeqn{ThinPartSeparation}, that there exist points $p \in \bdy M_a^{\geq \epsilon}$ and
$q \in \bdy M_b^{\leq \delta}$  such that $d_a(p,q) = \epsilon/2$. 
Since $p \in M_a^{\geq \epsilon}$, there is an embedded ball $B$  centered at $p$, of radius $\epsilon/2$ in the $g_a$ metric, such that
$q \in \bdy B$. See \reffig{EpsilonDeltaTubes}.

We will apply \refthm{BilipBis} to $B$. To check the hypotheses, note that
 \refeqn{DeltaDef} and $\epsilon \leq 0.3$ imply $\delta \leq 0.012$. In addition, 
note that  
\[ 3.324 \, (7.256)^{5/2} = 471.415 \ldots < 471.5,
\]
 hence
\refeqn{ThickStaysThickHypoth} and \refeqn{DeltaDef}  imply
\begin{equation}\label{Eqn:J'Bilip}
\ell  \leq \frac{(\epsilon/J)^5 \log J}{471.5 \cosh^5 (J \epsilon / 2 + 0.0424)} = \frac{\delta^{5/2} \log J'}{3.324} < \frac{\delta^{5/2} \log J}{3.324}.
\end{equation}
where $J'$ is ever so slightly less than $J$.
Finally, \refeqn{WeakSeparation} implies $B \subset M_t^{\geq \delta}$ for all $t\in \overline{I}$.

Thus we may apply \refthm{BilipBis} and get a $J'$--bilipschitz diffeomorphism on $B$. 
One direction of \refthm{BilipBis} says that distances from $p$ to points of $\bdy B$ can only shrink by a factor of $J'< J$ as we change metrics from $g_a$ to $g_b$. Thus $B$ contains a ball $B'$ of radius $\epsilon/(2J')$ in the $g_b$ metric, implying that $p \in M_b^{\geq \epsilon/J}$. On the other hand, since $d_a(p,q) = \epsilon/2$, the bilipschitz upper bound of \refthm{Bilip} implies 
\begin{equation}\label{Eqn:pqClose}
d_b(p,q) \leq J' \epsilon/2 < J \epsilon/2.
\end{equation}

Now, \refprop{DistThinParts} and \refeqn{DeltaDef} imply that the distance between thick and thin parts in the $g_b$ metric satisfies
\[
d_b(M_b^{\leq \delta}, M_b^{\geq \epsilon/J}) \geq \arccosh \left( \frac{\epsilon/J}{\sqrt{7.256\, \delta}} \right) - 0.0424 = J \epsilon / 2.
\]
But then $d_b(p,q) \geq J \epsilon / 2$, which contradicts \refeqn{pqClose}.
This contradiction implies that $b = \sup I \in I$. By the same argument, $\inf I \in I$. Thus $I$ is closed, hence \refeqn{ThinPartSeparation} stays true for all $t \in [0,(2\pi)^2]$.

Now, we can conclude the proof of the theorem for $0 < \epsilon \leq 0.3$.
For any $p \in M_a^{\geq \epsilon}$, there is an embedded ball $B$ centered at $p$, of radius $\epsilon/2$ in the $g_a$ metric. By \refeqn{ThinPartSeparation}, we have $B \subset M_t^{\geq \delta}$ for all $t$, hence \refthm{BilipBis} applies to give a $J'$--bilipschitz diffeomorphism on $B$. Thus, as in the above argument, we learn that for every $t$,  $B$ contains a ball $B'$, centered at $p$, of radius $\epsilon/(2J)$ in the $g_t$ metric. Therefore, $p \in M_t^{\geq \epsilon/ J}$, and  \refitm{ThickStaysThickish} holds.

Finally, if $0.3 < \epsilon \leq \log 3$, the argument is  identical apart from slightly different numbers. We define
\[
\delta = \frac{(\epsilon/J)^2}{7.256 \cosh^2 (J \epsilon/2 + 0.1475)} < 0.106,
\]
which means that hypothesis \refeqn{ThickStaysThickHypothLog3} enables us to apply the version of \refthm{BilipBis} for $\delta \leq 0.106$. Then we employ a crawling argument to prove  \refeqn{ThinPartSeparation} for this value of $\delta$, which implies \refitm{ThickStaysThickish}.
\end{proof}

\subsection{Applications}

Combining Theorems~\ref{Thm:MargulisTopology} and~\ref{Thm:ThickStaysThick} gives several results about the behavior of Margulis numbers under filling and drilling. We begin with a lemma.

\begin{lemma}\label{Lem:gJBehavior}
For $0 \leq \epsilon \leq \log 3$ and $1 \leq J \leq e^{1/5}$, consider the function
\[
g(\epsilon,J) = \frac{\epsilon^5 \log J}{496.1 \, J^5 \cosh^5 (J \epsilon / 2 + 0.1475)} 
\]
occurring in \refeqn{ThickStaysThickHypothLog3}. The maximum value of $g$ on this domain is $5.609\ldots \times 10^{-5}$, achieved when $\epsilon = \log 3$ and $J = 1.15203 \ldots $. Furthermore, $g(\epsilon, J)$ is increasing $\epsilon$ on its entire domain and increasing in $J$ when  $J \in [1, 1.15]$. 
\end{lemma}

We remark that $g(\epsilon, J)$ is \emph{not} increasing in $J$ on its entire domain; compare \reffig{FunctionOfJ}.

\begin{proof}
To see that $g(\epsilon,J)$ is increasing in $\epsilon$, we compute the partial derivative:
\[
\frac{\partial g}{\partial \epsilon} = \frac{5 \epsilon^4 \log J}{496.1 \, J^5 \cosh^6 (J \epsilon / 2 + 0.1475)} \cdot 
\big[ \cosh \big( \tfrac{J\epsilon}{2}  + 0.1475 \big) - \tfrac{J\epsilon}{2} \sinh \big( \tfrac{J\epsilon}{2}  + 0.1475 \big) \big].
\]
The first term in the above product is non-negative on the whole domain, and positive whenever $\epsilon > 0$ and $J > 1$. To analyze the second term, we substitute $x = J \epsilon / 2 $ and verify that the function $(\cosh (x + 0.1475) - x \sinh (x + 0.1475))$ is positive whenever $x \in [0,1]$. Since $J \epsilon/2 < 1$ on the entire domain, it follows that $\partial g/\partial \epsilon \geq 0$ on the entire domain.

In a similar fashion, we compute $\partial g / \partial J$:
\[
\frac{\partial g}{\partial J} = \frac{ \tfrac{5}{2} \epsilon^5 }{496.1 \, J^6 \cosh^6 (J \epsilon / 2 + 0.1475)} \cdot
\big[ \left(  \tfrac{2}{5} - 2 \log J \right) \cosh \big( \tfrac{J\epsilon}{2}  + 0.1475 \big) - \epsilon J \log J \sinh \big( \tfrac{J\epsilon}{2}  + 0.1475 \big) \big]
\]
The first term in the product is non-negative on the whole domain, and positive when $\epsilon > 0$. Using Sage \cite{FPS:Ancillary}, we verify that the second term is positive when $J \leq 1.15$, hence $g(\epsilon, J)$ is increasing in $J$ on this sub-domain.

Finally, we check the assertion about the maximum value of $g(\epsilon, J)$. By monotonicity in $\epsilon$, any maximum occurs when $\epsilon = \log 3$. We verify using Sage that $g(\log 3, J) < 5.610 \times 10^{-5}$ for every $J$ in the domain. We also check directly that the function attains a value greater than $5.609 \times 10^{-5}$ at $J = 1.15203$. See the ancillary files \cite{FPS:Ancillary} for full details.
\end{proof}

\begin{theorem}\label{Thm:MargulisFilling}
Fix $0 < \epsilon \leq \log 3$ and $1 < J \leq e^{1/5}$. Let $N$ be a cusped hyperbolic $3$--manifold such that $\epsilon$ is a Margulis number of $N$. Let $\mathbf s$ be a tuple of slopes on cusps of $N$ whose normalized length $L = L(\mathbf s)$ satisfies
\begin{equation*}
\frac{2\pi}{L(\mathbf s)^2 - 11.7} \leq \frac{\epsilon^5 \log J}{496.1 \, J^5 \cosh^5 (J \epsilon / 2 + 0.1475)} = g(\epsilon,J).
\end{equation*}
Then $\delta = \min \{ \epsilon/J, \, 0.962 \}$ is a Margulis number for $M = N(\mathbf s)$.
\end{theorem}

\begin{proof}
By \reflem{gJBehavior}, $g(\epsilon,J) < 5.61 \times 10^{-5}$ for all $\epsilon,J$.
Consequently, any tuple of slopes  $\mathbf s $ satisfying the hypotheses of the theorem must have normalized length $L = L(\mathbf s) \geq 334$.
By \refthm{UpwardConeDefRBounds},
there is a cone deformation from $M_0 = N$ to $M_{4\pi^2} = M = N(\mathbf s)$ maintaining a tube of radius $\Rmin$ about the singular locus $\Sigma$, where $\Zmin = \tanh \Rmin \geq 0.9998$. Consequently, \reflem{MagidLengthGeneral} says that the length of $\Sigma$ in the complete metric on $M$ is 
\[
\ell = \len_{4\pi^2}(\Sigma) \leq \frac{2\pi}{L^2 - 11.7} \leq g(\epsilon,J).
\]

By \reflem{gJBehavior}, $g(\epsilon, J)$ is increasing in $J$ when $J \leq 1.15$. Thus, given $\delta = \min \{ \epsilon/J, \, 0.962 \}$ as in the theorem statement, we define $J' = \epsilon/\delta = \max \{ J, \epsilon/0.962 \}$ and obtain
\[
\ell \leq g(\epsilon,J) \leq g(\epsilon, J').
\]
By \refthm{ThickStaysThick}, every cone-manifold $M_t$ occurring in the deformation satisfies $M_0^{\geq \epsilon} \subset M_t^{> \epsilon/ J'} = M_t^{> \delta}$. Taking complements of thick parts, and removing $\Sigma$, we obtain $(M_t^{\leq \delta} - \Sigma) \subset M_0^{< \epsilon}$.

We conclude the proof using \refthm{MargulisTopology}. The one hypothesis of that theorem that remains to be checked is  equation~\refeqn{MargulisTopologyEll}. The first inequality of \refeqn{MargulisTopologyEll} holds because 
\[
\ell \leq g(\epsilon,J') =  \frac{(\epsilon/J')^5 \,  \log J'}{496.1 \cosh^5 (J' \epsilon / 2 + 0.1475)} <  \frac{\delta^5 }{496.1} \ll 0.261 \delta.
\]
For the second inequality of \refeqn{MargulisTopologyEll}, we only need to consider $\delta \in [\delta_{\rm cut}, 0.962]$, where $\delta_{\rm cut}$ is as in the proof of \reflem{DeltaTubeEmbeds}; compare \reffig{DeltaEmbeds}.
By \reflem{HBound} and \refrem{Haze}, the function $\haze(\frac{\delta + 0.1604}{1.1227})$ is decreasing in this range. Thus
\[
\ell \leq g(\epsilon, J') < 5.61 \times 10^{-5} < 1.44 \times 10^{-4} <  \frac{1}{2\pi} \haze \left( \frac{0.962 + 0.1604}{1.1227}  \right) \leq \frac{1}{2\pi} \haze \left( \frac{\delta + 0.1604}{1.1227}  \right),
\]
hence \refeqn{MargulisTopologyEll} holds for every pair $(\epsilon, J)$. Thus we may use  \refthm{MargulisTopology} to conclude that $\delta$ is a Margulis number for every $M_t$, and in particular for the non-singular metric on $M$.
\end{proof}

\refthm{MargulisFilling} is useful in a situation where we have information about $N = M - \Sigma$ and its (optimal) Margulis number. However, Theorems~\ref{Thm:MargulisTopology} and \ref{Thm:ThickStaysThick} can also be used in a situation where we have information about $M$ and its short geodesics.

\begin{theorem}\label{Thm:MargulisConeMedConst}
  Let $M$ be a (non-singular) finite-volume hyperbolic 3-manifold with $k=0$, $1$, or $2$ cusps.
  Suppose the $(3-k)$ shortest geodesics in $M$ have total length at most $5.56 \times 10^{-5}$. Let $\Sigma$ denote the union of these geodesics. Then the geodesics are disjointly embedded, there exists a cone-deformation $M_t$ interpolating between the complete structure on $M-\Sigma$ and the complete structure on $M$, and for all $t$, the optimal Margulis number for $M_t$ is greater than $0.9536$. 
\end{theorem}

\begin{proof}
If $M$ contains $(3-k)$ closed geodesics of total length at most $5.56 \times 10^{-5}$, then by Meyerhoff's theorem \cite{meyerhoff}, those geodesics are disjointly embedded. By \refthm{ConeDefExists}, there is a cone-deformation $M_t$ interpolating between the complete structure on $M - \Sigma$ and the complete structure on $M$. Observe that $M_0 = M - \Sigma$ has three cusps. Thus $b_1(M_0) \geq 3$, hence \refthm{NonsingularMargulis}.\refitm{Log3Marg} says that $\epsilon = \log 3$ is a Margulis number for $M_0$. Setting $J = 1.152$ gives
\begin{equation*}
\ell \leq  5.56 \times 10^{-5} < \frac{\epsilon^5 \log J}{496.1 \, J^5 \cosh^5 (J \epsilon / 2 + 0.1475)} \, .
\end{equation*}
(As in \reflem{gJBehavior}, the value $J = 1.152$ was chosen because it very nearly places the mildest possible hypotheses on $\ell$.) Now, by \refthm{ThickStaysThick}, we have $M_t^{\leq \epsilon/J} \subset M_0^{< \epsilon}$ for every $t$. Since $\ell$ is small enough to satisfy equation~\refeqn{MargulisTopologyEll} for $\delta = \epsilon/J$, \refthm{MargulisTopology} implies that $ \epsilon/J > 0.9536$ is a Margulis number for every $M_t$.
\end{proof}

\begin{theorem}\label{Thm:MargulisConeMfld}
  Let $M$ be a finite-volume hyperbolic 3-manifold. Let $\ell = \systole(M)$ denote the length of a shortest geodesic $\Sigma \subset M$, and assume $\ell \leq 0.0996$. Then there exists a cone-deformation $M_t$ interpolating between the complete structure on $M_0=M-\Sigma$ and $M=M_{4\pi^2}$. Furthermore, the following hold for every $M_t$.
\begin{enumerate}
\item\label{Itm:TinyMargulisCone} If $\ell \leq 2.93 \times 10^{-7}$, then for any $t\in[0,4\pi^2]$,
the optimal Margulis number for $M_t$ is greater than $0.2408$.
\item\label{Itm:SmallMargulisCone} If $\ell\leq 2.73\times 10^{-8}$, then for any $t\in[0,4\pi^2]$, the optimal Margulis number for $M_t$ is greater than $0.29$.
\end{enumerate}
\end{theorem}

\begin{proof}
Since $\ell \leq 0.0996$, \refthm{ConeDefExists} implies the cone-deformation $M_t$ exists.

For \refitm{TinyMargulisCone}, assume $\ell \leq 2.93 \times 10^{-7}$.
By \refthm{NonsingularMargulis}.\refitm{CullerShalenMarg}, $\epsilon = 0.292$ is a Margulis number for $M_0$. Now, set $J = 1.2124$. Plugging in these values of $\epsilon$ and $J$ yields 
\begin{equation}\label{Eqn:TinyMargulisHypothCheck}
\ell \leq  2.93 \times 10^{-7} < \frac{\epsilon^5 \log J}{471.5 \, J^5 \cosh^5 (J \epsilon / 2 + 0.0424)} \, .
\end{equation}
(As above, the value $J = 1.2124 < e^{1/5}$ is chosen because it very nearly maximizes the function in \refeqn{TinyMargulisHypothCheck}, placing the mildest possible hypotheses on $\ell$.
See \reffig{FunctionOfJ}.)
Now, by \refthm{ThickStaysThick}, we have $M_t^{\leq \epsilon/J} \subset M_0^{< \epsilon}$ for every $t$.
Since $\ell$ is small enough to satisfy equation~\refeqn{MargulisTopologyEll} for $\delta = \epsilon/J$, \refthm{MargulisTopology} implies $ \epsilon/J > 0.2408$ is a Margulis number for every $M_t$.

Item \refitm{SmallMargulisCone} is obtained by an identical argument. Suppose that $\ell \leq 2.73\times 10^{-8}$, and set $\epsilon = 0.292$ and $J = 1.00689 < 0.292/0.29$. 
By \refthm{ThickStaysThick}, we have $M_t^{\leq \epsilon/J} \subset M_0^{< \epsilon}$ for every $t$.
Thus, by \refthm{MargulisTopology}, we have that $\epsilon/J > 0.29$ is a Margulis number for every $M_t$ in this case. 
\end{proof}

The above results imply the following. 

\begin{theorem}\label{Thm:MargulisDrilling}
Let $M$ be a non-singular hyperbolic 3-manifold. 
\begin{enumerate}
\item\label{Itm:TinyMargulis} If $\mu(M) \leq 0.2408$, then $M$ is closed and  $\vol(M) \leq 36.12$. Furthermore, $\systole(M) \geq 2.93 \times 10^{-7}$.
\item\label{Itm:SmallMargulis} If $\mu(M) \leq 0.29$, then $M$ is closed and  $\vol(M) \leq 52.78$. Furthermore, $\systole(M) \geq 2.73 \times 10^{-8}$.
\item\label{Itm:MediumMargulis} If $\mu(M) \leq 0.9536$, then $M$ has finite volume and  $k \in \{0,1,2\}$ cusps. The $(3-k)$ shortest geodesics in $M$ have total length at least $5.56 \times 10^{-5}$.
 \end{enumerate}
\end{theorem}

\begin{proof}
We begin by proving \refitm{TinyMargulis}.
Let $M$ be a hyperbolic 3-manifold with $\mu(M) \leq 0.2408$. By \refthm{NonsingularMargulis}.\refitm{CullerShalenMarg}, $M$ must be closed. By a theorem of Shalen \cite[Theorem~7.1]{Shalen:SmallOptimalMargulis}, we have $\vol(M) \leq 36.12$. Let $\Sigma$ be the shortest closed geodesic in $M$. Then \refthm{MargulisConeMfld}.\refitm{TinyMargulisCone} implies $\ell = \len(\Sigma) \geq 2.93 \times 10^{-7}$.

Turning to \refitm{SmallMargulis}, let $M$ be a hyperbolic 3-manifold with $\mu(M) \leq 0.29$. Then again, a theorem of Shalen \cite[Theorem~7.1]{Shalen:SmallOptimalMargulis} implies that $\vol(M) \leq 52.78$, and \refthm{MargulisConeMfld}.\refitm{SmallMargulisCone} implies $\systole(M) \geq 2.73 \times 10^{-8}$.

To check \refitm{MediumMargulis}, suppose that $\mu(M) \leq 0.9536$. Then, by \refthm{NonsingularMargulis}.\refitm{Log3Marg}, we know $\vol(M) < \infty$. Since a $k$--cusped manifold has $b_1(M) \geq k$, the same theorem implies that $M$ has $0 \leq k \leq 2$ cusps.
Then \refthm{MargulisConeMedConst} implies that the $(3-k)$ shortest geodesics in $M$ have total length at least $5.56 \times 10^{-5}$. 
\end{proof}

As a final application of the results of this section, we have a version of the bilipschitz theorem \refthm{Bilip} whose hypotheses are only on a non-singular manifold $M$, rather than the a cone-manifolds $M_t$ occurring in the middle of the deformation.

\begin{theorem}\label{Thm:BilipEndpoints}
Fix $0 < \epsilon \leq \log 3$. Let $M$ be a finite-volume hyperbolic 3-manifold and $\Sigma$ a geodesic link in $M$. Let $N = M - \Sigma$. Suppose that one of the following hypotheses holds:
\begin{enumerate}
\item\label{Itm:EllHypothesis} In the complete structure on $M$, the total length of $\Sigma$ satisfies
\begin{equation}\label{Eqn:TightEllBound}
\ell \leq \frac{\epsilon^5}{6771 \cosh^5(0.6 \epsilon + 0.1475)}.
\end{equation}
\item\label{Itm:LHypothesis} In the complete structure on $N = M - \Sigma$, the total length of the meridians of $\Sigma$ satisfies
\begin{equation}\label{Eqn:TightL2Bound}
L^2 \geq \frac{2\pi \cdot 6771 \cosh^5(0.6 \epsilon + 0.1475)}{\epsilon^5} + 11.7.
\end{equation}
\end{enumerate}
Then there is a cone-deformation $M_t$ connecting the complete hyperbolic metric $g_0$ on $N$ to the complete hyperbolic metric $g_{4\pi^2}$ on $N$. Furthermore, the cone deformation gives a natural identity map $\id \from (M - \Sigma, g_0) \to (M - \Sigma, g_{4\pi^2})$, such that $\id$ and $\id^{-1}$ restrict to
\[
\id \vert_{N^{\geq \epsilon}} \from N^{\geq \epsilon} \hookrightarrow M^{\geq \epsilon/1.2},
\qquad
\id^{-1} \vert_{M^{\geq \epsilon}} \from M^{\geq \epsilon} \hookrightarrow N^{\geq \epsilon/1.2}, 
\]
which are $J$--bilipschitz inclusions for
\[
J = \exp \left( \frac{11.35 \, \ell}{\epsilon^{5/2} } \right)
\qquad \text{and} \qquad 
\ell \leq \frac{2\pi}{L^2 - 11.7}.
\]
Furthermore, for any $x \in M^{\geq \epsilon} \cup N^{\geq \epsilon}$, we have $\frac{1}{1.2} \injrad_0(x) \leq  \injrad_{4\pi^2}(x) \leq 1.2 \,  \injrad_0(x)$. 
\end{theorem}

We remark that the strong hypotheses on $\ell$ or $L$ are driven by \refthm{ThickStaysThick}. Under these strong hypotheses, we do get very tight control on the bilipschitz constant: 
for every $\epsilon \leq \log 3$, the theorem gives a value $J\in (1, \, 1.0005)$. 

\begin{proof}[Proof of \refthm{BilipEndpoints}]
First, we check that hypothesis \refitm{LHypothesis} implies hypothesis \refitm{EllHypothesis}.
By \reflem{gJBehavior} (substituting the value $J_0 = 1.2$), the right-hand side of \refeqn{TightEllBound} is increasing in $\epsilon$, hence the right-hand side of  \refeqn{TightL2Bound} is decreasing in $\epsilon$. Thus the loosest possible upper bound on $L^2$ occurs when $\epsilon = \log 3$, and implies  $L^2 \geq 116,321$. As in the proof of \refthm{MargulisFilling}, we can now conclude using  \reflem{MagidLengthGeneral} that
\[
\ell \leq \frac{2\pi}{L^2 - 11.7},
\]
hence the hypothesis on $L$ implies the one on $\ell$, as claimed. By \refthm{ConeDefExistsRBounds} and \reflem{MagidLengthGeneral}, the estimate $\ell \leq \frac{2\pi}{L^2 - 11.7}$ also holds under hypothesis \refitm{EllHypothesis}.

\
Now, set $J_0 = 1.2$. Then our hypotheses imply 
\[
\ell \leq \frac{\epsilon^5 \log J_0}{496.1 J_0^5 \cosh^5(J_0 \epsilon/2 + 0.1475)}.
\]
Hence \refthm{ThickStaysThick} applies. By \refthm{ThickStaysThick}, we have
\[ 
M^{\geq \epsilon} \subset M_t^{\geq \epsilon/1.2}, \qquad N^{\geq \epsilon} \subset M_t^{\geq \epsilon/1.2} \qquad \mbox{ for all } t.
\]
The last conclusion of \refthm{ThickStaysThick} also implies that for any $a,b \in [0, (2\pi)^2]$, a point $x \in M_a^{\geq \epsilon}$ satisfies $\frac{1}{1.2} \injrad_a(x) \leq  \injrad_{b}(x) \leq 1.2 \,  \injrad_a(x)$. In particular, this holds when $\{a, b\} = \{0, (2\pi)^2\}$. 

To complete the proof, we set $W = N^{\geq \epsilon}$ and apply \refthm{Bilip} with $\delta = \epsilon/1.2$. By \refthm{Bilip}, the cone-deformation provides a natural $J$--bilipschitz map $\id \from (W, g_0) \to (W, g_{4\pi^2})$, as desired. Applying \refthm{Bilip} to $W = M^{\geq \epsilon}$ gives the reverse $J$--bilipschitz inclusion $\id^{-1}$.
\end{proof}

\begin{corollary}[\refthm{BilipEndpointsIntro}]\label{Cor:BilipEndpointsDrill}
Fix any $0 < \epsilon \leq \log 3$ and any $J>1$. Let $M$ be a finite-volume hyperbolic $3$--manifold and $\Sigma$ a geodesic link in $M$ whose total length $\ell$ satisfies
\begin{equation*}
\ell \leq \min\left\{ \frac{\epsilon^5}{6771 \cosh^5(0.6 \epsilon + 0.1475)}, \, \frac{\epsilon^{5/2}\log(J)}{11.35} \right\}.
\end{equation*}
Then, setting $N = M - \Sigma$,
there are natural $J$--bilipschitz inclusions
\[
\varphi \from M^{\geq \epsilon} \hookrightarrow N^{\geq \epsilon/1.2}, 
\qquad
\psi \from N^{\geq \epsilon} \hookrightarrow M^{\geq \epsilon/1.2},
\]
which are equivariant with respect to the symmetry group of the pair $(M, \Sigma)$.
\end{corollary}

\begin{proof}
The $J$--bilipschitz inclusions $\varphi$ and $\psi$ are restrictions of the natural identity maps $\id$ and $\id^{-1}$ from \refthm{BilipEndpoints}. Because $\id$ is defined by a canonical harmonic form $\omega$, in \refrem{OmegaUniqueness}, it is equivariant with respect to the symmetry group of $(M, \Sigma)$.
\end{proof}

\begin{corollary}\label{Cor:BilipEndpointsFill}
Fix any $0 < \epsilon \leq \log 3$ and any $J>1$. Let $M$ be a $3$--manifold with empty or toroidal boundary, and $\Sigma$ a link in $M$.
Suppose that  $N = M - \Sigma$ admits a complete, finite volume hyperbolic metric where the total normalized length of the meridians of $\Sigma$ satisfies
\begin{equation*}
L^2 \geq \max \left\{ \frac{2\pi \cdot 6771 \cosh^5(0.6 \epsilon + 0.1475)}{\epsilon^5} + 11.7, \:
\frac{2\pi \cdot 11.35}{\epsilon^{5/2}\log(J)} + 11.7 \right\}.
\end{equation*}
Then $M$ admits a complete hyperbolic metric in which $\Sigma$ is isotopic to a union of geodesics. Furthermore, there are natural $J$--bilipschitz inclusions
\[
\varphi \from M^{\geq \epsilon} \hookrightarrow N^{\geq \epsilon/1.2}, 
\qquad
\psi \from N^{\geq \epsilon} \hookrightarrow M^{\geq \epsilon/1.2},
\]
which are equivariant with respect to the symmetry group of the pair $(M, \Sigma)$. 
\end{corollary}

\begin{proof}
By \refthm{UpwardConeDefRBounds}, $M$ is hyperbolic and $\Sigma$ is a union of geodesics. Now, the $J$--bilipschitz inclusion $\varphi$ is a restrictions of the natural identity map $\id \from (M - \Sigma, g_0) \to (M - \Sigma, g_{4\pi^2})$ from \refthm{BilipEndpoints}, while $\psi$ is a restriction of $\id^{-1}$. By \refrem{OmegaUniqueness}, the identity map $\id$ is equivariant with respect to the symmetry group of $(M, \Sigma)$, hence so are $\varphi$ and $\psi$.
\end{proof}

\appendix

\section{Hyperbolic trigonometry}\label{Sec:Trig}

This appendix records several elementary facts that are used throughout the paper.

\begin{lemma}\label{Lem:TanhSinhCosh}
Let $z = \tanh r$ and $x=e^r$. Then
\[
e^r = \sqrt{ \frac{1+z}{1-z} }, \qquad \sinh r = \frac{xz}{1+z} = \frac{z}{\sqrt{1-z^2}}, \qquad \cosh r = \frac{x}{1+z} =\frac{1}{\sqrt{1-z^2}}.
\]
\end{lemma}

\begin{proof}
We may solve the (quadratic) equation
\[
z = \frac{x - x^{-1}}{x + x^{-1}}
\quad \mbox{to find} \quad
x = \sqrt{ \frac{1+z}{1-z} }.
\]
Now, substituting the formula for $x$ into
\[
\sinh r = x \cdot  \frac{1 - x^{-2}}{2}, \qquad \cosh r = x \cdot \frac{1 + x^{-2}}{2}
\]
gives the remaining identities.
\end{proof}

%

\begin{lemma}\label{Lem:SinhCoshGrowth}
Let $0 < r < s$. Then
\[
\frac{\cosh s}{\cosh r} < e^{s-r} < \frac{\sinh s}{\sinh r}.
\]
\end{lemma}

\begin{proof}
Let $h = s-r$. Then 
\begin{align*}
\cosh(s) & = \cosh(r+h) \\
& = \cosh r \cosh h + \sinh r \sinh h \\
& < \cosh r \cosh h + \cosh r \sinh h \\
& = \cosh r \cdot e^h,
\end{align*}
proving the first inequality. The second inequality is proved similarly.
\end{proof}

\begin{lemma}\label{Lem:SinhDiff}
Suppose that $0 < s \leq \smax$ and $\tanh \smax \leq \zmin \leq \tanh r$.
Then
\[
\sinh (r-s) \geq \sinh r \cdot f(\smax, \zmin) \geq e^r  \frac{\zmin}{1+\zmin} \cdot  f(\smax, \zmin),
\]
where
\[
f(s,z) = \cosh s - z^{-1} \sinh s.
\]
\end{lemma}

\begin{proof}
We set $z = \tanh r$ and compute:
\begin{align*}
\sinh(r-s)
& = \sinh r \cosh s - \cosh r \sinh s \\
& = \sinh r \cosh s - z^{-1} \sinh r \sinh s \\
& = \sinh r \cdot f(s,z).
\end{align*}
Since $s > 0$ and $z = \tanh r \in (0,1)$, it follows that
\[
\frac{\partial f}{\partial z} > 0
\quad \mbox{and} \quad
\frac{\partial f}{\partial s} < 0.
\]
Therefore,
\[
f(s,z) \geq f(\smax, \zmin), 
\]
proving the first inequality of the lemma. For the second inequality, note that the hypothesis $\tanh \smax \leq \zmin$ implies $f(\smax, \zmin) \geq 0$.
Now, we obtain
\[
\sinh r \cdot  f(\smax, \zmin) =  e^r \frac{z}{1+z} \cdot  f(\smax, \zmin) \geq e^r \frac{\zmin}{1+\zmin} \cdot  f(\smax, \zmin),
\]
where the equality is \reflem{TanhSinhCosh} and the inequality is the monotonicity of $\frac{z}{1+z}$.
\end{proof}

\bibliographystyle{amsplain}
\bibliography{biblio}

\end{document}